\documentclass[fms,times]{cuparticle}
\def\journalreprintline{}
\hypersetup{final=true,colorlinks=true,urlcolor=magenta,citecolor=blue,linkcolor=red}

\usepackage[figuresright]{rotating}
\usepackage{amsmath,amsfonts,amsthm,amssymb,latexsym,amscd,amsopn,eucal,mathrsfs,graphics,color,enumerate,lscape,booktabs}
\usepackage[justification=centering]{caption}
\usepackage{afterpage}
\usepackage{chngpage}
\usepackage[all,cmtip]{xy}

\usepackage{etex}
\reserveinserts{28}

\theoremstyle{plain}
  \newtheorem{theorem}{Theorem}[section]
  \newtheorem{proposition}[theorem]{Proposition}
  \newtheorem{corollary}[theorem]{Corollary}
  \newtheorem{lemma}[theorem]{Lemma}
  
\theoremstyle{definition}
  \newtheorem{definition}[theorem]{Definition}

  \newtheorem{remark}[theorem]{Remark}

\usepackage{tikz}
\usetikzlibrary{decorations.markings}

\newcommand{\R}{{\mathbb R}}
\newcommand{\C}{{\mathbb C}}
\newcommand{\Z}{{\mathbb Z}}
\newcommand{\Q}{{\mathbb Q}}
\newcommand{\F}{{\mathbb F}}
\newcommand{\ccdot}{{\,\cdot\,}}
\newcommand{\Proj}{{\mathbb P}}
\newcommand{\Pone}{\mathbb{P}^1}
\newcommand{\sO}{{\mathcal O}}
\newcommand{\cube}{A}

\newcommand{\sM}{{\mathcal M}}
\newcommand{\sF}{{\mathcal F}}

\newcommand{\NS}{\mathrm{NS}}
\newcommand{\Norm}{\mathrm{Norm}}
\newcommand{\Pic}{\mathrm{Pic}}
\newcommand{\Aut}{\mathrm{Aut}}
\newcommand{\tns}{\otimes}
\newcommand{\Sym}{\mathrm{Sym}}
\newcommand{\Ga}{\mathbb{G}_a}
\newcommand{\Gm}{\mathbb{G}_m}
\newcommand{\SL}{\mathrm{SL}}
\newcommand{\GL}{\mathrm{GL}}
\newcommand{\Sp}{\mathrm{Sp}}
\newcommand{\SO}{\mathrm{SO}}

\newcommand{\disc}{\mathrm{disc}\,}
\newcommand{\cH}{\mathrm{H}}

\newcommand{\subs}{\lrcorner \,}

\newcommand{\ch}{\mathrm{h}}
\newcommand{\T}{\mathrm{T}}
\newcommand{\PGL}{\mathrm{PGL}}
\newcommand{\id}{\mathrm{Id}}
\newcommand{\fd}{F}
\newcommand{\fdbar}{{\kern.1ex\overline{\kern-.3ex F \kern+.1ex}\kern.1ex}}

\DeclareMathOperator{\rank}{rank}

\volume{}
\doi{}

\begin{document}

\authorheadline{Manjul Bhargava, Wei Ho, and Abhinav Kumar}
\runningtitle{Orbit Parametrizations for K3 Surfaces}

\begin{frontmatter}

\title{Orbit Parametrizations for K3 Surfaces}
\author[1]{Manjul Bhargava}
\address[1]{Department of Mathematics, Princeton University, Princeton, NJ 08544
\ead{bhargava@math.princeton.edu}}
\author[2]{Wei Ho}
\address[2]{Department of Mathematics, University of Michigan, Ann Arbor, MI 48109
\ead{weiho@umich.edu}}
\author[3]{Abhinav Kumar}
\address[3]{Department of Mathematics, Massachusetts Institute of Technology, Cambridge MA 02139; Current address: Department of Mathematics, Stony Brook University, Stony Brook, NY 11794
\ead{thenav@gmail.com}}

\begin{abstract}

We study moduli spaces of lattice-polarized K3 surfaces in terms of
orbits of representations of algebraic groups.  In particular, over an
algebraically closed field of characteristic $0$, we show that in many
cases, the nondegenerate orbits of a representation are in bijection
with K3 surfaces (up to suitable equivalence) whose N\'eron-Severi
lattice contains a given lattice.  An immediate consequence is that
the corresponding moduli spaces of these lattice-polarized K3 surfaces
are all unirational.  Our constructions also produce many
fixed-point-free automorphisms of positive entropy on K3 surfaces in 
various families associated to these representations, giving a natural
extension of recent work of~Oguiso.

\end{abstract}


\end{frontmatter}

\tableofcontents

\setcounter{MaxMatrixCols}{50} 

\section{Introduction}

An important classical problem is that of classifying
the orbits of a representation, over a field or over a ring, in terms
of suitable algebraic or geometric objects over that field or ring;
conversely, one may wish to construct representations whose orbits
parametrize given algebraic or geometric objects of interest.

In recent years, there have been a number of arithmetic applications
of such ``orbit parametrizations'' for geometric objects of dimension
$0$ and of dimension $1$.  For example, for extensions of a field, or
ring extensions and ideal classes in those extensions, such
parametrizations have been studied in numerous papers, including
\cite{delonefaddeev, WrightYukie, hcl1, hcl2, hcl3, hcl4,
  melanie-binarynics,melanie-2nn}; these descriptions of the moduli
spaces have been used in an essential way in many applications (see,
e.g., \cite{davenportheilbronn, manjulcountquartic,
  manjulcountquintic,bst, tt, kevinthesis}).  In fact, many of the
cleanest and most useful such bijections between rings/ideal classes
and orbits of representations have arisen in cases where the
representation is {\em prehomogeneous}, i.e., where the ring of
(relative) invariants is a polynomial ring with one generator.

In the case of curves, recent work on orbit parametrizations in cases
of arithmetic interest over a general base field include
\cite{coregular} for genus one curves and
\cite{bg-hyper,jthorne-thesis} for various types of higher genus
curves.  Numerous examples have also been previously considered by
algebraic geometers, even classically, often giving descriptions of
the {\em coarse} moduli space as a GIT quotient.  As before, many of
the most arithmetically useful bijections between data relating to
algebraic curves and orbits of representations have arisen in cases
where the representation has somewhat simple invariant theory --- in
particular, when the representation is {\it coregular}, meaning that
the ring of relative invariants is a polynomial ring.  In these cases,
the coarse moduli space of the geometric data is thus (an open
subvariety of) a weighted projective space.  In conjunction with
geometry-of-numbers and other analytic counting and sieve arguments,
such representations have seen applications in bounding average ranks
in families of elliptic curves over $\Q$ (see, e.g.,
\cite{arulmanjul-bqcount, arulmanjul-tccount}) and
showing that many curves in families of higher genus have few rational
points (see \cite{bg-hyper,poonen-stoll,aruljerry-monichyp,manjul-hyper}).

A natural next step is to determine representations whose orbits
parametrize geometric data of interest associated to algebraic
surfaces.  K3 surfaces form a rich class of surfaces that naturally
lend themselves to such a study, and in fact, there has already been
significant work in this direction (albeit usually over algebraically
closed fields).  For example, it is classically known that a general
polarized K3 surface of genus $g = 3$, $4$, or $5$ may be described as
a complete intersection in projective space $\Proj^g$, and such
descriptions may be easily translated into the language of orbits of a
representation of an algebraic group.  For polarized K3 surfaces of
higher genus, Mukai and others have also described them in several
cases as complete intersections in homogeneous spaces (see, e.g.,
\cite{mukai-k3-genusupto10, mukai-k3-genus1820, mukai-k3-genus1213}).
A sample of these results for polarizations of small degree appears in
Table \ref{table:examples2}.

The purpose of this paper is to generalize these ideas to study moduli
spaces of K3 surfaces with possibly multiple line bundles, namely {\em
  lattice-polarized} K3 surfaces, in terms of the orbits of suitable
representations. More precisely, the classes of line bundles in the
Picard group of a K3 surface $X$ naturally form a lattice, with the
symmetric bilinear pairing being the intersection pairing on divisors
of the surface; this is called the {\it N\'eron-Severi lattice} of
$X$. There is a coarse moduli space $\sM_\Lambda$ of K3 surfaces whose
N\'eron-Severi lattice contains a fixed lattice
$\Lambda$.\footnote{We will actually work with a slight
modification of this moduli space; see \S \ref{sec:K3s}.}  It is a
quasi-projective variety, but in general it is very difficult to
explicitly describe it by equations or to understand its geometry,
especially if $\rank(\Lambda) > 1$.  We show that there are at least
19 representations of algebraic groups whose orbits naturally
parametrize such lattice-polarized K3 surfaces.  We list them in
Table~\ref{table:examples}.  Some of these orbit parametrizations are
classical, but most of the higher rank cases appear to be new.

\begin{sidewaystable}
\renewcommand{\arraystretch}{1.1}
\begin{center}
\begin{tabular}{|c|c|c|l|c|c|}
\hline
No. & Group $G^{\textrm{ss}}$ & Representation $V$ & Rank & Generic NS & Section \\
\hline
1 & $\SL_2$ & $\Sym^8(2) \oplus \Sym^{12}(2)$ & 2 & $\mathfrak{O}_1 = U$ & \S \ref{subsec:ellsurf} \\
2 & $\SL_2^2$ & $\Sym^4(2) \tns \Sym^4(2)$ & 2 & $U(2)$& \S \ref{subsec:44doublecover} \\
3 & $\SL_2 \times \SL_3 $ & $\Sym^2(2)\tns\Sym^3(3)$ & 2 & $\mathfrak{O}_9$  & \S \ref{subsec:23form} \\
4 & $\SL_3^2$ & $3 \tns 3 \oplus \Sym^2(3) \tns \Sym^2(3)$ & 2 & $\mathfrak{O}_{12}$ & \S \ref{subsec:1122complete} \\
5 & $\SL_2^2\times \SL_4$ & $2 \tns 2 \tns \Sym^2 (4)$ & 2 & $\mathfrak{O}_{4}(2)$ &  \S \ref{sec:2x2xSym24} \\
6 & $\SL_4^3$ & $4\tns4\tns4$ & 2 & $\mathfrak{O}_{5}(2)$ & \S \ref{sec:rr}  \\
7 & $\SL_3^2$ & $3 \tns \Sym^2(3) \oplus \Sym^2(3) \tns 3$ & 2 & $\mathfrak{O}_{21}$ &  \S \ref{subsec:1221complete} \\
8 & $\SL_2^3 $ & $\Sym^2(2)\tns\Sym^2(2)\tns\Sym^2(2)$ & 3 & $U(2) \oplus \langle -4 \rangle$ &  \S \ref{subsec:222forms} \\
9 & $\SL_2^5 $ & $2\tns2\tns2\tns2\tns2$ &  4 & $U(2) \oplus A_2(2)$ &  \S \ref{sec:penteracts} \\
10 & $\SL_2\times \SL_4 $ & $\Sym^2(2)\tns\Sym^2(4)$ & 9 & $U \oplus E_7(2)$ & \S \ref{sec:sym22xsym24} \\
11 & $\SL_2^4 $ & $2\tns2\tns2\tns\Sym^2(2)$ & 9 & $U \oplus E_7(2)$ & \S \ref{sec:2sympent} \\
12 & $\SL_4^2$ & $4\tns \Sym^2(4)$ & 11 & $U \oplus E_8(2) \oplus \langle -4 \rangle$ & \S \ref{sec:2symrr} \\
13 & $\SL_2^3 $ & $2\tns\Sym^2(2)\tns\Sym^2(2)$ & 12  & $U(4) \oplus E_8 \oplus \langle -4 \rangle^{\oplus 2}$ & \S \ref{sec:22sympent} \\
14 & $\SL_2^3\times \SL_4$ & $2 \tns 2\tns 2\tns 4$ & 13 & $\langle 4 \rangle \oplus \langle -2 \rangle^{\oplus 4} \oplus D_4^{\oplus 2}$ & \S \ref{sec:2224} \\
15 & $\SL_2^3 $ & $2\tns2\tns\Sym^3(2)$ & 14 & $U(2) \oplus A_2^{\oplus 3} \oplus E_6$ &  \S \ref{sec:3sympent}\\
16 & $\SL_2^2 $ & $\Sym^2(2)\tns\Sym^3(2)$ & 15  & $U \oplus A_2^{\oplus 2} \oplus E_6 \oplus A_3$ & \S \ref{sec:23sympent}  \\
17 & $\SL_4$ & $\Sym^3(4)$ & 16 & $U(2) \oplus A_2 \oplus D_{12}$ & \S \ref{sec:3symrr} \\
18 & $\SL_2^2 $ & $2\tns\Sym^4(2)$ & 17 & $U \oplus \langle -8 \rangle \oplus D_{12} \oplus A_2$ & \S \ref{sec:4sympent} \\
19 & $\SL_2 $ & $\Sym^5(2)$ & 18  & $U \oplus A_4 \oplus D_{12}$  & \S \ref{sec:5sympent}\\
\hline
\end{tabular}
\end{center}
\caption{Representations $V$ whose $G$-orbits parametrize data related
  to K3 surfaces.  The group $G^{\textrm{ss}}$ is a semisimple
  algebraic group with a map to the group $G$, whose kernel is finite
  and cokernel is solvable.  Here $\mathfrak{O}_D$ denotes the lattice
  underlying the quadratic ring of discriminant $D$ with the quadratic
  form being twice the norm form. Root lattices are normalized to be
  negative definite, and $U \cong \mathfrak{O}_1$ denotes the
  hyperbolic plane.}
\label{table:examples}
\end{sidewaystable}

Just as $2\times 2\times 2$ cubical matrices played a key role in the
understanding of many prehomogeneous representations \cite{hcl1}, and
just as $3\times3\times3$ and $2\times 2\times 2\times 2$ matrices
played a key role in the understanding of coregular representations
associated to genus one curves \cite{coregular}, we find that
$4\times4\times4$ and $2\times 2\times 2\times 2\times 2$ matrices
appear as fundamental cases for our study of K3 surfaces.  We refer to
these cases as the ``Rubik's revenge'' and ``penteract'' cases,
respectively.  We also study orbits on symmetrized versions of these
spaces, which turn out to correspond to moduli spaces of K3 surfaces
of higher rank.  For example, we show that $\GL_2$-orbits on the
space of {quintuply symmetric penteracts}---i.e., binary quintic
forms---correspond to elements of a certain family of K3 surfaces
having rank at least 18.

We now state our main theorem more precisely.  Given a K3 surface
defined over a field $\fd$ having algebraic closure $\fdbar$, let
$\NS(X)$ denote the N\'eron-Severi group of $X$, i.e., the group of
divisors on $X$ over $F$ modulo algebraic equivalence.  Let
$\overline{\NS}(X)$ be $\NS(X_{\fdbar})$. (When $\fd$ is algebraically
closed, we have $\NS(X) = \overline{\NS}(X)$.)
Then we define a lattice-polarized K3 surface over $\fd$ as follows.

\begin{definition} \label{def:latticepolarized}
Let $\Lambda$ be an even nondegenerate lattice with signature $(1,s)$ with a choice of basis, and let $\Sigma$ be a saturated sublattice of $\Lambda$.
Then we say that a K3 surface $X$ over $\fd$ is {\em lattice-polarized
  by $(\Lambda,\Sigma)$} if there exists a primitive lattice embedding $\phi: \Lambda
\to \overline\NS(X)$
such that $\phi(\Sigma)$ is fixed pointwise by the action of ${\rm
  Gal}(\fdbar/\fd)$
and the image under $\phi$ of a certain subset $C(\Lambda)$ of $\Lambda$ contains an ample divisor class of $X$.

We now define the subset $C(\Lambda)$.  Let $Z(\Lambda) := \{z \in \Lambda : \langle z, z \rangle = -2 \}$ be the set of roots of $\Lambda$.  Fix a partitioning of $Z(\Lambda)$ into two subsets $Z(\Lambda)^+$ and $Z(\Lambda)^-$, where $Z(\Lambda)^- = \{-z: z \in Z(\Lambda)^+ \}$ and each is closed under positive finite sums; also fix a connected component $V$ of the cone $\{z \in \Lambda \otimes \R: \langle z, z \rangle > 0 \}$.  We then let $C(\Lambda)$ be the subset of $V \cap \Lambda$ consisting of elements that pair positively with all $z \in Z(\Lambda)^+$; this is the intersection of $\Lambda$ with the Weyl chamber defined by the positive roots.

Let $\sM_{\Lambda, \Sigma}$ denote the moduli space of such pairs
$(X,\phi)$, where $X$ is a K3 surface lattice-polarized by $(\Lambda,
\Sigma)$ and $\phi:\Lambda \to \overline\NS(X)$ is a primitive lattice embedding, modulo
equivalence; two pairs $(X,\phi)$ and $(Y, \psi)$ are {\em equivalent}
if there exists an isomorphism $f: Y \to X$ and an isometry $g:\Lambda \to \Lambda$
fixing $\Sigma$ pointwise such that the following diagram commutes:
\begin{equation*}
\xymatrix{
\Lambda \ar[r]^-{\phi} \ar[d]_{g} & \overline{\NS}(X) \ar[d]^{f^*} \\
\Lambda \ar[r]_-{\psi}  & \overline{\NS}(Y).
}
\end{equation*}

More generally, if $X$ is a K3 surface that is lattice-polarized by
$(\Lambda,\Sigma)$, and if $S\subset \Sigma$ is a subset that spans
the $\Q$-vector space $\Sigma\otimes\Q$, then we also say that $X$ is
{\it lattice-polarized by $(\Lambda,S)$}.  Similarly, we may speak of
the moduli space $\sM_{\Lambda,S}:=\sM_{\Lambda,\Sigma}$.
\end{definition}

\begin{remark} \label{rmk:latticeconditions}
By convention, we assume that the lattice $\Lambda$ has a class with
positive norm and has a unique embedding in the {\it K3 lattice}
$E_8^2 \oplus U^3$ (up to equivalence) in order to define the above
moduli space (see \cite{nikulin, dolgachev-mirror}); here $E_8$
denotes the unique $8$-dimensional negative definite even unimodular
lattice and $U$ the hyperbolic lattice with Gram matrix
$(\begin{smallmatrix} 0 & 1 \\ 1 & 0 \end{smallmatrix})$.  In all of
the cases we consider, the lattice $\Lambda$ will satisfy these
properties.  If $S$ equals $\Lambda$ (or contains a set of generators
of $\Lambda$) in Definition~\ref{def:latticepolarized}, then we obtain
the moduli space $\sM_\Lambda$ of {\it $\Lambda$-polarized} K3
surfaces.  Furthermore, if the $\Z$-span of $S$ contains a positive
class, then the cover $\sM_{\Lambda,S} \to \sM_\Lambda$ is finite.
\end{remark}

Then our main theorem is as follows.

\begin{theorem}\label{thm:main2}
  Let $\fd$ be a field of characteristic $0$.  For any line of
  Table~$\ref{table:examples}$, there exists an explicit finite subset
  $S$ of $\Lambda$ such that the $G(\fd)$-orbits of an open subset of
  $V(\fd)$ are in canonical bijection with the $\fd$-points of an open
  subvariety of the moduli space $\sM_{\Lambda,S}$ of K3 surfaces
  lattice-polarized by $(\Lambda, S)$.
\end{theorem}
\noindent
In each section, we will specify the relevant subset $S$ of $\Lambda$.

\begin{corollary}
  The moduli space $\sM_{\Lambda,S}$ of $(\Lambda,S)$-polarized K3
  surfaces is unirational. In particular, the moduli space
  $\sM_{\Lambda}$ of $\Lambda$-polarized K3 surfaces is unirational.
\end{corollary}

A number of the spaces in Table~\ref{table:examples} have been studied
previously, often from the point of view of invariants of group actions and not necessarily with a specific connection to K3 surfaces, and usually over an algebraically closed field.
Some of these results show that many of the moduli spaces in Table~\ref{table:examples} are actually rational over $\fdbar$. For instance, \cite{ma} proves the rationality of Nos.~2, 16, and 18; the rationality of No.~1 follows from \cite{katsylo}.  The rationality of No.~17 follows from the classical computation of invariants of cubic surfaces, and that of No.~19 from the invariant theory of the binary quintic. It is an interesting problem to determine exactly which of the spaces in Table~\ref{table:examples} are rational over $\fdbar$ or over $\fd$.

We note that Table \ref{table:examples} is not intended to be
a complete classification of all orbit spaces that are birational to
moduli spaces related to lattice-polarized K3 surfaces.
For instance, one could consider the space $3 \otimes 3 \otimes 4$,
whose elements define an unordered set of six points in the
plane, or via duality, a set of six lines; the double cover of the
plane branched along the six lines is a K3 surface of Picard number
$16$, so over $\fdbar$ (but not $\fd$), there is a correspondence between these orbits and such K3 surfaces. Such K3 surfaces have been extensively studied in the past, e.g., see \cite{matsumoto-sasaki-yoshida, ng-334, lombardo-peters-schuett}.

As mentioned earlier, in the case where $\Lambda$ $(=\Sigma)$ has rank 1,
several cases have been studied previously, and we have recorded them
in Table \ref{table:examples2}.  The first four cases are classical
and are easily adapted to give the correct parametrization over any
field, while the last five are more recent and arise in the beautiful
work of Mukai \cite{mukai-k3-genusupto10}.  It is an interesting
problem to work out the appropriate forms of Mukai's representations
so that they also parametrize polarized K3 surfaces over a general
field.

\begin{table}[ht!]
\renewcommand{\arraystretch}{1.1}
\begin{center}
\begin{tabular}{|c|c|c|c|}
\hline
Group $G^{\textrm{ss}}$ & Representation $V$ & Degree\\
\hline
$\SL_3$ & $\Sym^6 (3)$ & $2$  \\ 
\hline
$\SL_4$ & $\Sym^4 (4)$ & $4$ \\ 
\hline
$\SL_5$ & $\Sym^2(5) \oplus \Sym^3(5)$ & $6$ \\
\hline
$\SL_3\times \SL_6$ & $3\tns \Sym^2(6)$ & $8$ \\
\hline
\hline
$\SL_2$ & $\Sym^8(2)\oplus\Sym^{12}(2)$ & $10$  \\
\hline
$\SL_8\times \SO_{10}$ & $8\tns S^+(16)$  & $12$ \\
\hline
$\SL_6\times \SL_6$ & $6\tns \wedge^2(6)$ & $14$  \\
\hline
$\SL_4\times \Sp_6$ & $4\tns \wedge_0^3(6)$ & $16$  \\
\hline
$\SL_3\times G_2$ & $3\tns (14)$ & $18$ \\
\hline
\end{tabular}
\end{center}
\abovecaptionskip 0pt
\caption{Representations whose orbits parametrize polarized K3
  surfaces.\\ (The cases of degrees 10 to 18 are due to Mukai.)}
\label{table:examples2}
\end{table}

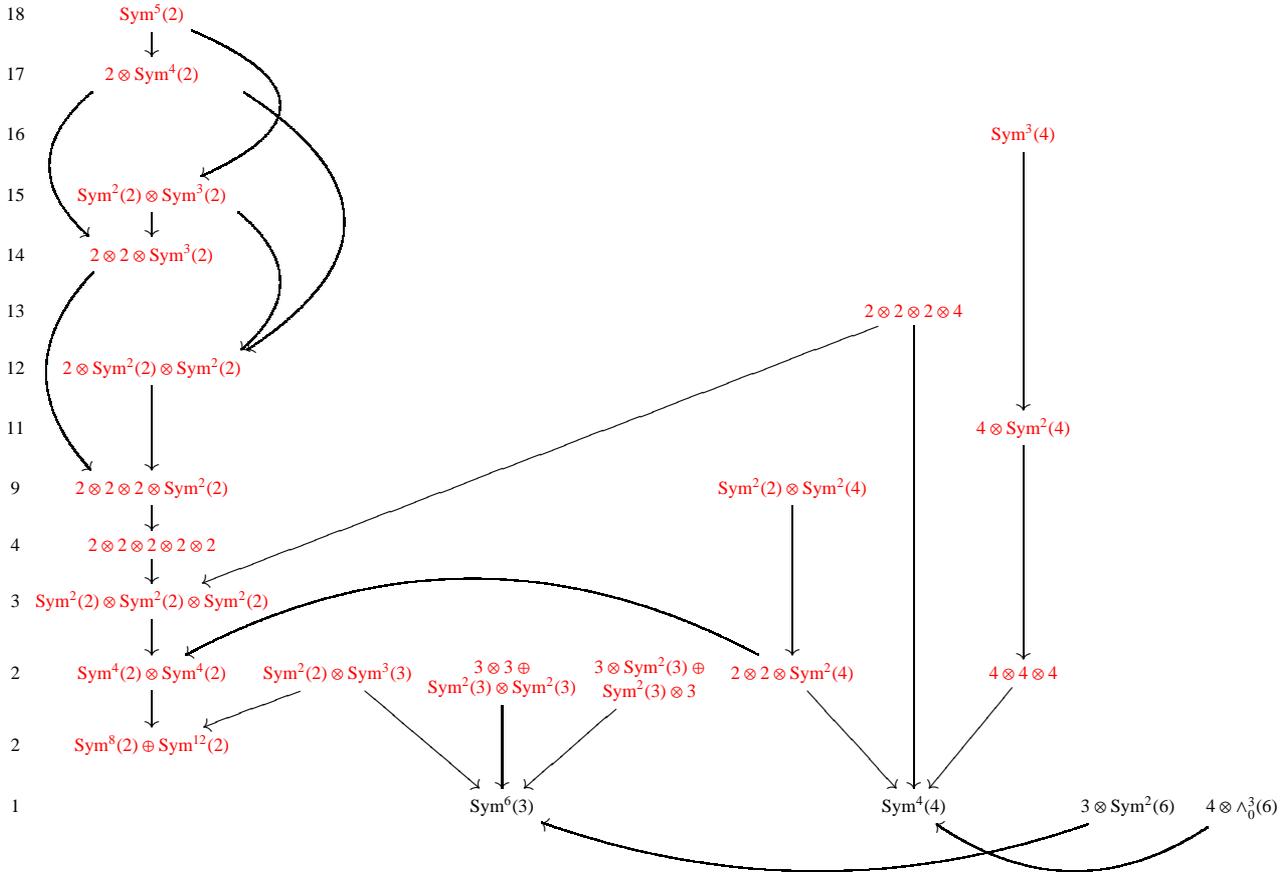
\begin{sidewaysfigure}
\scriptsize
\begin{equation*}
\xymatrix@R=9pt@C=-2pt{
18 & \hyperref[sec:5sympent]{\Sym^5(2)} \ar[d] \ar@/^4pc/[ddd]\\
17 & \hyperref[sec:4sympent]{2 \otimes \Sym^4(2)} \ar@<-1.2pc>@/_2pc/[ddd] \ar@<2pc>@/^4pc/[ddddd] \\
16 & & & & & & & \hyperref[sec:3symrr]{\Sym^3(4)} \ar[ddddd]\\
15 & \hyperref[sec:23sympent]{\Sym^2(2) \otimes \Sym^3(2)} \ar[d] \ar@<2pc>@/^2pc/[ddd]\\
14 & \hyperref[sec:3sympent]{2 \otimes 2 \otimes \Sym^3(2)} \ar@<-1.3pc>@/_2pc/[dddd] \\
13 & & & & & & \hyperref[sec:2224]{2 \otimes 2 \otimes 2 \otimes 4} \ar[dddddddd] \ar[lllllddddd] \\
12 & \hyperref[sec:22sympent]{2 \otimes \Sym^2(2) \otimes \Sym^2(2)} \ar[dd] \\
11 & & & & & & & \,\hyperref[sec:2symrr]{4 \otimes \Sym^2(4)}\, \ar[dddd]\\
9 & \hyperref[sec:2sympent]{2 \otimes 2 \otimes 2 \otimes \Sym^2(2)}  \ar[d] & & & & \!\!\!\!\!\hyperref[sec:sym22xsym24]{\Sym^2(2) \otimes \Sym^2(4)}\!\!\!\!\! \ar[ddd] \\
4 & \hyperref[sec:penteracts]{2 \otimes 2 \otimes 2 \otimes 2 \otimes 2} \ar[d] \\
3 & \hyperref[subsec:222forms]{\Sym^2(2) \otimes \Sym^2(2) \otimes \Sym^2(2)} \ar[d] \\
2 & \;\;\hyperref[subsec:44doublecover]{\Sym^4(2) \otimes \Sym^4(2)}\;\; \ar[d] 
	& \!\!\!\!\!\hyperref[subsec:23form]{\Sym^2(2) \otimes \Sym^3(3)} \ar[ld] \ar[rdd] 
	& \hyperref[subsec:1122complete]{\txt{$3 \otimes 3\  \oplus$\\$\;\Sym^2(3) \otimes \Sym^2(3)\;$}} \ar[dd] 
	& \hyperref[subsec:1221complete]{\txt{$\;\!3 \otimes \Sym^2(3) \ \oplus\;\;\;$\\$\!\Sym^2(3) \otimes 3\;\;$}} \ar[ldd] 
	& \hyperref[sec:2x2xSym24]{2 \otimes 2 \otimes \Sym^2(4)} \ar@/_3pc/[llll] \ar[rdd]
	& & \hyperref[sec:rr]{4 \otimes 4 \otimes 4} \ar[ldd]\\
2  & \hyperref[subsec:ellsurf]{\Sym^8(2) \oplus \Sym^{12}(2)} \\
1 & & & \Sym^6(3) & & & \Sym^4(4) & 
	& \!3 \otimes \Sym^2(6)\,\, \ar@/^2pc/[lllll]
	& \ \ \;4 \otimes \wedge^3_0(6) \ar@/^2pc/[lll] \\
	\\
}
\end{equation*}
\caption{Covariance relations among orbit parametrizations of
  lattice-polarized K3 surfaces.}
\label{fig:covariantrelationships}
\end{sidewaysfigure}

Figure \ref{fig:covariantrelationships} shows how many of the cases
from both Table \ref{table:examples} and Table \ref{table:examples2}
are related.  In particular, each arrow from a representation $V$ of
the group $G$ to a representation $V'$ of $G'$ indicates that there is
a group homomorphism $\tau: G \to G'$ and a map $V \to V'$ that is
$G$-equivariant with respect to~$\tau$.  For each arrow, there is a
map between the associated moduli spaces of lattice-polarized K3
surfaces as well as a reverse inclusion of the corresponding
polarization lattices $\Lambda$ and $\Sigma$. The ranks of the
polarization lattices $\Lambda$ are indicated in the first column.

Such explicit descriptions of the moduli spaces of lattice-polarized
K3 surfaces also have several other potential applications. For
example, these moduli spaces are related to Noether-Lefschetz
divisors, which are special cycles on moduli spaces of polarized K3
surfaces (see, e.g., \cite{kudla}). There has also been a great deal
of recent activity surrounding the Noether--Lefschetz conjecture
\cite{maulikpandharipande, bergeron-li-millson-moeglin}, and it would
be interesting to extend the work of Greer--Li--Tian \cite{greerlitian}
on the GIT stability of the Mukai models to these spaces of
lattice-polarized K3s.

In many of the new cases listed in Table~\ref{table:examples}, we also
obtain natural automorphisms of the corresponding K3 surfaces.  Due to
their frequently extremely interesting and rich groups of
automorphisms, K3 surfaces have provided a natural setting in recent
years on which to study questions of dynamics (see, e.g.,  \cite{cantat,
  mcmullen, oguiso-salem}; for a nice survey, see \cite{cantat-survey}).
In particular, there has been considerable interest in exhibiting positive
entropy automorphisms of projective algebraic K3 surfaces.  Recall that
the {\em entropy} of an automorphism $\phi$ of a projective surface $X$
is defined to be $\log \lambda(\phi)$, where $\lambda(\phi)$ is the spectral
radius of $\phi^*$ on $\overline{\NS}(X) \otimes \R$.  When $X$ is defined
over $\C$, this definition agrees with the topological entropy
(see \cite[\S 4.4.2]{cantat-survey}).

Recently, Oguiso~\cite{oguiso} showed that any projective algebraic
K3 with a fixed-point-free automorphism of positive entropy must have
Picard number at least 2.
He also produced a family of examples with Picard number 2 by
considering the Cayley K3 surfaces, i.e., the K3 surfaces arising from
Rubik's revenge (Line~6 of Table~\ref{table:examples}).  More
precisely, he proved that any K3 surface with N\'eron-Severi lattice
exactly $\mathfrak{O}_5(2)$ has a fixed-point-free automorphism with
entropy $\eta_{\mathrm{RR}} = 6\log\bigl(\frac{1+\sqrt5}{2}\bigr)> 0$.
(See also~\cite{festi}.)

Our perspective on the Rubik's revenge case in Section~\ref{sec:rr}
allows us to give a simpler proof of Oguiso's theorem, and in a
stronger form, in Section~\ref{sec:applications}.  More precisely, we
prove:

\begin{theorem}\label{thm:oguisoextension}
Let $X$ be a K3 surface with line bundles $L_1$ and $L_2$ satisfying
$L_1^2 = L_2^2 = 4$ and $L_1 \cdot L_2 = 6$. Assume that the
projective embedding of $X$ corresponding to $L_1$ is smooth. Then $X$
has a fixed-point-free automorphism of entropy
$\eta_{\mathrm{RR}}=6\log\bigl(\frac{1+\sqrt5}{2}\bigr)\approx2.887>0$.
\end{theorem}

The key ingredient is the use of the hyperdeterminant.  The {\it
  hyperdeterminant} is a generalization of the determinant for
multidimensional matrices, which was introduced by Cayley and studied
in depth by Gelfand, Kapranov, and Zelevinsky \cite{GKZ}.  For most
orbit parametrizations of rings and ideal classes by multidimensional
matrices (e.g., all those in \cite{hcl1, hcl2, melanie-2nn}), the
hyperdeterminant can be shown to equal the discriminant of the
corresponding ring.  For most orbit parametrizations of algebraic
curves in terms of multidimensional matrices (e.g., all those in
\cite{coregular,beauville,BGW}), the hyperdeterminant can be shown to
equal the discriminant of the corresponding algebraic curve.  Thus,
for all these parametrizations of data associated to rings and curves
by multidimensional matrices, the nonvanishing of the hyperdeterminant
corresponds to the nondegeneracy of the associated ring and the
nonsingularity of the associated curve, respectively.

The orbit parametrizations (of K3 surfaces by multidimensional
matrices) considered in this paper yield a number of examples where
the hyperdeterminant does {\it not} coincide with the discriminant.
Indeed, in these cases, we show that the hyperdeterminant only divides
the discriminant of the K3 surface, but is {not} equal to it.  This
raises the question as to the interpretation of the hyperdeterminant
in these cases.  We will prove that the nonvanishing of the
hyperdeterminant corresponds precisely to an associated automorphism
of the K3 surface being fixed-point-free.
This interpretation is what consequently allows us to prove
Theorem~\ref{thm:oguisoextension} for all nonsingular Cayley K3
surfaces.

An additional advantage of our method is that it also naturally
extends to other cases.  For example, we may use Line 9 of Table 1
together with our hyperdeterminant method to produce examples of K3
surfaces of rank 4 having many fixed-point-free automorphisms of
positive entropy:

\begin{theorem}\label{thm:penteractentropy}
  Let $X$ be a nonsingular K3 surface corresponding to a
  penteract. Then $X$ has a fixed-point-free automorphism with entropy
  $\log (\lambda_{\mathrm{pent}}) \approx 2.717>0$, where
  $\lambda_{\mathrm{pent}} = 4 + \sqrt{13} + \sqrt{7 + 2 \sqrt{13}}
  \approx 15.145$ is the unique real root of $x^4 - 16x^3 + 14x^2 -
  16x + 1$ greater than $1$.
\end{theorem}

In fact, we will show that these K3 surfaces coming from penteracts
have infinitely many automorphisms of positive entropy.  Recall that a
{\em Salem number} is a real algebraic integer $\lambda > 1$ whose
conjugates other than $\lambda^{\pm 1}$ lie on the unit circle%
\footnote{We follow the convention of McMullen \cite{mcmullen}, where the set of Salem numbers includes quadratic integers with these properties. Another convention is to call such quadratic integers {\em Pisot numbers}.}%
; its irreducible minimal polynomial is then called a {\em        
Salem polynomial} (see \cite{bams-salem} for a survey of problems involving Salem numbers). ÊThe entropy      
of an automorphism of a projective K3 surface is either $0$ or the logarithm of a {Salem number}          
\cite[\S 3]{mcmullen}, and it is an interesting question as to which Salem polynomials arise from             
automorphisms of K3 surfaces (see, e.g., \cite{cantat, mcmullen-glue, mcmullen, reschke}).                    
                                                                                                              
We obtain a plethora of quadratic and quartic Salem polynomials 
from the automorphisms of the K3 surfaces arising from penteracts.  
In fact, we will demonstrate in \S\ref{sec:pentauts} that, for a positive proportion
of natural numbers $n$, both the polynomials $x^2 - (4n^2 \pm 2)x + 1$ and
$x^2 - (12n^2 \pm 2)x + 1$ arise as Salem polynomials of automorphisms of
the general K3 surface in the penteract family.  As a consequence, it follows that
all real quadratic fields occur as the splitting fields
of Salem polynomials of automorphisms of general K3 surfaces in this family.

Other examples, similar to those mentioned in the preceding two
theorems, will be constructed in Section~\ref{sec:applications}; for
example, we will construct fixed-point-free automorphisms on K3
surfaces in certain families having entropy equal to the logarithm of
$3 + 2 \sqrt{2}$, $\frac{3 + \sqrt{5}}{2}$, and $2 + \sqrt{3}$,
respectively.

\

\noindent {\bf Outline.}  This paper is organized as follows.  In \S
\ref{sec:prelims}, we give some background on lattices and K3
surfaces, as well as notation, that will be used throughout the paper.
In \S \ref{sec:classical}, we then very briefly discuss some cases
from Table \ref{table:examples} that are classical, e.g., some of the
moduli spaces where the general K3 has Picard rank~$2$.  As seen in
Figure \ref{fig:covariantrelationships}, these classical examples
often appear as covariants of other cases that we consider in this
paper.

The bulk of this paper lies in \S\S
\ref{sec:rr}--\ref{sec:sym22xsym24}.  In each of these sections, we
prove Theorem \ref{thm:main2} for the specified group $G$,
representation $V$, lattice $\Lambda$, and subset $S$ of $\Lambda$ as
listed in Table \ref{table:examples}.  We begin each section with a
construction of the K3 surfaces, and associated line bundles/divisors,
obtained from a general $G(\fd)$-orbit of $V(\fd)$.  In many of the
sections, we also discuss various automorphisms of the relevant K3
surfaces arising from these constructions.

Finally, in \S \ref{sec:applications}, we consider some connections
between these bijections and hyperdeterminants, as well as some
applications to dynamics on K3 surfaces.

\section{Preliminaries} \label{sec:prelims}
\subsection{Lattices} \label{sec:lattices}

A {\it lattice} $\Lambda$ is a free abelian group of finite rank,
equipped with a bilinear pairing $\langle \, , \, \rangle : \Lambda
\times \Lambda \to \Z$. We will assume that that pairing is
nondegenerate. A lattice is often described by the Gram matrix $G = (
\langle v_i, v_j \rangle)_{ij}$ with respect to a basis $v_1, \dots,
v_n$ of $\Lambda$. The {\it discriminant} of the lattice is then
$\det(G)$ and the {\it discriminant group} of $\Lambda$ is $A =
\Lambda^*/\Lambda$, where $\Lambda^*$ is the dual lattice.  It is
equipped with a {\it discriminant form} $\phi: A \times A \to
\Q/\Z$. The {\it real signature} of $\Lambda$ is the pair $(r_+, r_-)$
consisting of the number of positive and negative eigenvalues of its
Gram matrix $G$. We say that $\Lambda$ is {\it positive definite}
(respectively, {\it negative definite}, or {\it indefinite}) if $r_- =
0$ (resp., $r_+ = 0$, or $r_+ r_- \neq 0$). Similarly, one can define
the {\it integral $p$-adic signatures} for every prime $p$, by
considering $\Lambda_p = \Lambda \otimes \Z_p$ and its discriminant
group $A_{\Lambda_p}$. These are invariants of the isomorphism class of
the lattice.

We now record for use in this paper some useful theorems regarding
embeddings and isomorphisms of lattices.

\begin{theorem} 
Let $\Lambda$ be an indefinite integral lattice of rank $n \geq 3$ and
discriminant $d$. Assume that there is no odd prime $p$ such that
$p^{n(n-1)/2} \mid d$, and that $2^{n(n-3)/2 + \lfloor (n+1)/2 \rfloor
} \nmid d$. Then the class number of the genus of $\Lambda$ is
$1$. That is, $\Lambda$ is determined by the collection of its real
and $p$-adic signatures.
\end{theorem}

\noindent This theorem is used in Table~\ref{table:examples} to
identify the various N\'eron-Severi lattices that arise in our orbit
problems with direct sums of familiar lattices. For the lower-rank
examples, it is easy to give a direct identification; for higher rank,
it suffices to compute the genus symbols of the N\'eron-Severi lattice
(defined by a Gram matrix in the corresponding section), and match
them with those of the lattices given in the table. For brevity, we
omit these verifications in the text.

\begin{theorem}[Nikulin {\cite[Theorem 1.14.4]{nikulin-quadforms}}]
Let $M$ be an even lattice with real signature $(t_+, t_-)$ and
discriminant form $\phi_M$, and let $\Lambda$ be an even unimodular
lattice of signature $(s_+, s_-)$. Suppose that
\begin{enumerate}
\item[{\rm (a)}]
 $t_+ < s_+$
\item[{\rm (b)}] $t_- < s_-$
\item[{\rm (c)}] $\ell(A_{M_p}) \leq \rank(\Lambda) - \rank(M) - 2$ for $p \neq 2$
\item[{\rm (d)}] One of the following condition holds at the prime $2$.
\begin{enumerate}
\item[{\rm (i)}] $\ell(A_{M_2}) \leq \rank(\Lambda) - \rank(M) - 2$, or
\item[{\rm (ii)}] $\ell(A_{M_2}) = \rank(\Lambda) - \rank(M)$ and
  $\phi_M \cong u_2^+(2) \oplus q'$ or $\phi_M \cong v_2^+(2) \oplus
  q'$ for some $q'$.
\end{enumerate}
\end{enumerate}
Then there exists a unique primitive embedding of $M$ into $\Lambda$.
\end{theorem}

\noindent Here, $\ell(A)$ denotes the smallest number of generators of
the group $A$, and $u_2^+(2)$ and $v_2^+(2)$ are specific discriminant
forms on certain finite $2$-groups (see
\cite{nikulin-quadforms}). This theorem is the key ingredient in
checking that our lattices have unique primitive embeddings in the K3
lattice. We leave the routine verification to the interested reader.

\subsection{K3 surfaces} \label{sec:K3s}

In this subsection, we first recall some basics of K3 surfaces and
explain the existence of a coarse moduli space $\sM_{\Lambda,S}$ for a
lattice $\Lambda$ as in Remark \ref{rmk:latticeconditions} and a
subset $S$ of $\Lambda$. For simplicity, we work over the complex
numbers, though it is possible also to give an algebraic description
of lattice-polarized K3 surfaces \cite{beauville-fanoK3}.

A K3 surface $X$ over $\fd$ is a projective algebraic nonsingular
surface with trivial canonical bundle and $\ch^1(X,\sO_X) = 0$.  The
cohomology group $\cH^2(X,\Z)$, equipped with the cup product form, is
a $22$-dimensional lattice that is abstractly isomorphic to the {\it
  K3 lattice} $\Lambda_{K3} := E_8^2 \oplus U^3$, where $E_8$ is the
$8$-dimensional negative definite even unimodular lattice and $U$ is
the hyperbolic lattice with Gram matrix $(\begin{smallmatrix} 0 & 1
  \\ 1 & 0 \end{smallmatrix})$.  The N\'eron-Severi group
$\overline{\NS}(X)$ as defined in the introduction is a primitive
sublattice of the K3 lattice, with signature $(1,\rho-1)$.

We start with the notion of a {\it marked $\Lambda$-polarized} K3
surface. Pick an embedding $\Lambda \hookrightarrow \Lambda_{K3}$ (it
does not matter which, since all embeddings are equivalent by
assumption). A {\it marking} is an isomorphism $\phi: \cH^2(X, \Z) \to
\Lambda_{K3}$ such that $\phi^{-1}(\Lambda) \subset \NS(X)$. For such
a marked polarized K3 surface, the class of a regular $2$-form
$\omega$ on $X$ maps under $\phi$ to an element $z \in \Lambda^\perp
\otimes \C$, since $z$ pairs to zero with the algebraic
classes. Furthermore, it is easy to see from Hodge theory that
$\langle z, z \rangle > 0$ and $\langle z, \overline{z} \rangle = 0$.
Therefore, we have $z \in \Omega$ where $\Omega$ is an open subset of
the quadric cone in $\Lambda^\perp \otimes \C$ defined by these two
conditions. Since the form is unique up to scaling, we obtain a
well-defined element of $\Proj(\Omega)$. The map taking $(X, \phi)$ to
$z$ is called the {\em period mapping}. It can be shown that it yields
an isomorphism between the moduli space of marked ample
$\Lambda$-polarized K3 surfaces and the complement $\Omega^0$ of a
union of hyperplanes in $\Omega$.

To remove the marking, let $\Gamma(\Lambda)$ be the group
\[
\Gamma(\Lambda) = \{ \sigma \in O(\Lambda_{K3}) \,:\, \sigma(v) = v
\textrm{ for all } v \in \Lambda \}.
\]
Then an element $\sigma\in\Gamma(\Lambda)$ acts on the moduli space by
sending $(X,\phi)$ to $(X, \phi \circ\sigma)$, which gives an
isomorphism of the polarized K3 surfaces. Let $\Gamma_\Lambda$ be the
image of $\Gamma(\Lambda)$ in $O(\Lambda^\perp)$. Then the moduli
space of $\Lambda$-polarized K3 surfaces is obtained by taking the
quotient by the group action of $\Gamma(\Lambda)$. It establishes an
isomorphism with the period domain, obtained by taking the quotient
$\Omega^0/\Gamma_\Lambda$. For more details, we refer the reader to
\cite{dolgachev-mirror}.

In this paper, we will require a minor modification of this
construction. Namely, we do not quotient by the pointwise stabilizer
of $\Lambda$, but only by the pointwise stabilizer of $S$.  Let
\[
\Gamma(\Lambda,S) = \{ \sigma \in O(\Lambda_{K3}) \, :\,
\sigma(\Lambda) = \Lambda \textrm{ and } \sigma(s) = s \textrm{ for
  any } s \in S \}.
\]
Then $\Gamma(\Lambda,S)$ contains $\Gamma(\Lambda)$ as a subgroup, and
is generally strictly larger.\footnote{However, $\Gamma(\Lambda)$
  is a finite index subgroup of $\Gamma(\Lambda,S)$ if $S$ (or its
  span) contains a positive/ample class.} Let $\Gamma_{\Lambda,S}$ be
its image in $O(\Lambda^\perp)$. The moduli space of
$(\Lambda,S)$-polarized K3 surfaces is obtained by taking the quotient
of the fine moduli space of ample marked $\Lambda$-polarized K3
surfaces by $\Gamma(\Lambda,S)$. From the period mapping, it follows
that the dimension of the space $\sM_{\Lambda,S}$ (when $S$ contains a
positive divisor) is $20 - \rank \Lambda$.

We note here a lemma of Nikulin \cite[Lemma 3]{nikulin-kummer}, which
will be very useful in the determination of N\'eron-Severi groups of
the K3 surfaces studied in this paper.

\begin{lemma}[Nikulin] \label{lem:nikulin}
Let $X$ be a K3 surface, and $E_1, \dots, E_n$ disjoint smooth
rational curves on $X$ such that $\frac12(E_1+\cdots+E_n)\in
\overline{\NS}(X)$. Then $n \in \{0,8,16\}$.
\end{lemma}

\noindent
The proof is a simple calculation of the Euler characteristic of the
double cover of $X$ branched along the divisor $\sum E_i$.

\subsection{Notation and conventions} \label{sec:notation}
We make some brief remarks on the notation and conventions used in
this paper.

\begin{itemize}
\item Unless otherwise stated, we will work over a field $\fd$ of
  characteristic $0$.  We expect that most of the results also hold
  for non-supersingular K3 surfaces in positive characteristic larger
  than 3.
\item For a vector space $V$, the notation $\Sym^n(V)$ denotes the
  $n$th symmetric power of $V$ as a quotient of $V^{\tns n}$.
  However, since we will be working over a field of characteristic
  $0$, this space is canonically isomorphic to the subspace
  $\Sym_n(V)$ of $V^{\tns n}$ and, in fact, it will usually be more
  natural for us to view it as the subspace.
\item If $V$ is a representation of a group $G$, we
  will sometimes consider an action of $\Gm \times G$ on $V$,
  where $\Gm$ acts by scaling.
\item We will be denoting various K3 surfaces using indices; in each
  such case, any permutation of the subscripts will denote the same
  surface, e.g., $X_{123}$ and $X_{132}$ will refer to the same
  surface.
\item We pass between line bundles and divisors on our K3 surfaces
  freely, and we will often use additive notation to denote the tensor
  product of line bundles.  When working with relations among line
  bundles, we also use $=$ to denote an isomorphism (or equivalence
  among divisors).
\item Multilinear forms play a large role in many of our
  constructions.  For example, for vector spaces $V_1$, $V_2$, and
  $V_3$, we sometimes denote an element $A$ of $V_1 \tns V_2 \tns V_3$
  as the trilinear form $A(\ccdot, \ccdot, \ccdot)$, where each
  $\cdot$ may also be replaced by an element of the appropriate dual
  vector space $V_i^\vee$.  By abuse of notation, we may also allow
  points of the projective space $\Proj(V_i^\vee)$ as entries in the
  multilinear form $A$ when we are only asking about the vanishing or
  nonvanishing of $A$.  Finally, the notation $A \subs x$ for an
  element $x \in V_2^\vee$ is just $A(\ccdot, x, \ccdot)$, for
  example.
\item When discussing the induced action of an automorphism of a K3
  surface on the N\'eron-Severi group, the matrices will act on row
  vectors. In particular, if $Q$ is the matrix of the quadratic form
  representing the N\'eron-Severi lattice, then we have $M Q M^t = Q$
  for the matrix $M$ of any automorphism.
\end{itemize}

\section{Some classical moduli spaces for K3 surfaces with low Picard number} \label{sec:classical}

We first recall some of the classical cases listed in the first few
entries of Table \ref{table:examples}. All but the last of them have
Picard number $2$, leading to a moduli space of dimension $20 - 2 =
18$. In each case below, we see directly that the moduli space is
unirational and the points in an open subset correspond bijectively to
orbits of a suitable representation of a reductive group. The excluded
locus in each case is a union of Noether--Lefschetz divisors (in the
sense of \cite{maulikpandharipande}) on the corresponding moduli space
of lattice-polarized K3 surfaces.

For the cases with Picard number $2$, the generic N\'eron-Severi
lattice is even, of signature $(1,1)$ and of (absolute) discriminant
$D$. This lattice coincides with the lattice underlying the quadratic
ring $\mathfrak{O}_D$ of discriminant $D$ equipped with twice the norm
bilinear form, i.e., the form $\langle u, v \rangle = N(u+v) - N(u) -
N(v)$.

\subsection{Elliptic surfaces with section} \label{subsec:ellsurf}
The simplest indefinite even lattice is the hyperbolic plane $U$ of
discriminant $1$.  The moduli space of K3 surfaces lattice-polarized
by $U$ is the same as that of elliptic surfaces with section. Over a
field of characteristic not $2$ or $3$, we may write the Weierstrass
equation of such a surface as
\[
y^2 = x^3 + a_4(t) x + a_6(t),
\]
with $a_4(t)$ and $a_6(t)$ polynomials of degree at most $8$ and $12$
respectively. (For such a Weierstrass equation to describe a K3 rather
a rational surface, we also need $\deg(a_4) > 4$ or $\deg(a_6) > 6$.)
Once we quotient by Weierstrass scaling $(x,y) \to (\lambda^4x,
\lambda^6 y)$ and the $\PGL_2$ action on the base $\Proj^1_t$, we
obtain a moduli space of dimension $9 + 13 - 1 - 3 = 18$, as
expected. This moduli space $\sM_U$ is clearly unirational, and
corresponds to the representation $\Sym^8(2) \oplus \Sym^{12}(2)$.

\subsection{\texorpdfstring{Double covers of $\Proj^1 \times \Proj^1$}{Double covers of P1 x P1}} \label{subsec:44doublecover}
The second discriminant we need to consider is $4$, corresponding to
the lattice $U(2)$. The corresponding K3 surfaces are double covers of
$\Proj^1 \times \Proj^1$, branched along a bidegree $(4,4)$ curve. The
pullbacks of the two hyperplane classes give us line bundles $L_1$ and
$L_2$ with $L_1^2 = L_2^2 = 0$ and $L_1 \cdot L_2 = 2$. Either of the
projections to $\Proj^1$ is a genus one fibration, and exhibits the
surface as an elliptic surface with a $2$-section. The moduli space is
birational to the space of orbits of $\Gm \times \GL_2^2$ on
$\Sym^4(2) \otimes \Sym^4(2)$.

\subsection{\texorpdfstring{Hypersurfaces of bidegree $(2,3)$ in $\Proj^1 \times \Proj^2$}{Hypersurfaces of bidegree (2,3) in P1 x P1}}\label{subsec:23form}
A smooth hypersurface of bidegree $(2,3)$ in $\Proj^1 \times \Proj^2$
is a K3 surface. It has two line bundles $L_1$ and $L_2$ which are
pullbacks of the hyperplane classes, and satisfy $L_1^2 = 0$, $L_2^2 =
2$, and $L_1 \cdot L_2 = 3$. The generic N\'eron-Severi lattice of
this family
\[
\begin{pmatrix}
0 & 3 \\
3 & 2
\end{pmatrix},
\]
has discriminant $9$. The moduli space is birational to the quotient
of $\Sym^2(2) \otimes \Sym^3(3)$ by $\Gm \times \GL_2 \times \GL_3$.

\subsection{\texorpdfstring{Complete intersection of bidegree $(1,1)$ and $(2,2)$ hypersurfaces in $\Proj^2 \times \Proj^2$}{Complete intersection of bidegree (1,1) and (2,2) hypersurfaces in P2 x P2}} \label{subsec:1122complete}
Next, we consider K3 surfaces given as the smooth complete
intersection of bidegree $(1,1)$ and $(2,2)$ forms in $\Proj^2 \times
\Proj^2$.  This time, the pullbacks $L_1$ and $L_2$ of the two line
bundles satisfy $L_i^2 = 2$ (since the intersection of two lines on
one of the $\Proj^2$'s specifies a point, whence $L_i^2$ is obtained
by computing the intersection number of a line and a
conic). Similarly, we check that $L_1 \cdot L_2 = 4$, from the
intersection number of bidegree $(1,1)$ and $(2,2)$ curves on $\Proj^1
\times \Proj^1$. Therefore, the generic N\'eron-Severi lattice has
Gram matrix
\[
\begin{pmatrix}
2 & 4 \\
4 & 2
\end{pmatrix}
\]
with discriminant $12$. The moduli space is birational to the quotient
of $3 \otimes 3 \oplus \Sym^2(3) \otimes \Sym^2(3)$ by $\Gm \times
(\Ga^9 \rtimes \GL_3^2)$, where $\Ga^9$ acts by adding to the $(2,2)$
form the product of the given bidegree $(1,1)$ form with another
bidegree $(1,1)$ form.

\subsection{\texorpdfstring{Complete intersection of bidegree $(1,2)$ and $(2,1)$ hypersurfaces in $\Proj^2 \times \Proj^2$}{Complete intersection of bidegree (1,2) and (2,1) hypersurfaces in P2 x P2}} \label{subsec:1221complete}
Finally, consider K3 surfaces given as the smooth complete
intersection of bidegree $(1,2)$ and $(2,1)$ forms in $\Proj^2 \times
\Proj^2$.  As in the case of discriminant $12$ above, we obtain the
generic N\'eron-Severi lattice
\[
\begin{pmatrix}
2 & 5 \\
5 & 2
\end{pmatrix}
\]
of discriminant $21$. The moduli space is birational to the quotient
of $3 \otimes \Sym^2(3) \oplus \Sym^2(3) \otimes 3$ by $\GL_3^2$.

\subsection{\texorpdfstring{Hypersurfaces of tridegree $(2,2,2)$ in $\Proj^1 \times \Proj^1 \times \Proj^1$}{Hypersurfaces of tridegree (2,2,2) in P1 x P1 x P1}} \label{subsec:222forms}
Finally, we consider K3 surfaces defined by the vanishing of a
tridegree $(2,2,2)$ form on $\Proj^1 \times \Proj^1 \times
\Proj^1$. The three line bundles obtained from pulling back
$\sO_{\Proj^1}(1)$ have intersection matrix
\[
\begin{pmatrix}
0 & 2 & 2 \\
2 & 0 & 2 \\
2 & 2 & 0
\end{pmatrix}.
\]
The moduli space is birational to the quotient of $\Sym^2(2) \otimes
\Sym^2(2) \otimes \Sym^2(2)$ by $\Gm \times \GL_2^3$.

\section{Rubik's revenge: \texorpdfstring{$4 \tns 4 \tns 4$}{4 (x) 4 (x) 4}}
\label{sec:rr}

We begin with a space of K3 surfaces that has been well-studied in the
classical literature in algebraic geometry \cite{cayley,
  snyder-sharpe, jessop, room}, as well as more recently
\cite{beauville, oguiso, festi}: that of determinantal quartics.  Our
perspective is slightly different, however, allowing us to unify
several existing results in the literature; in particular, we classify
orbits on the space of $4\times 4\times 4$ cubical matrices over a
general field $F$ in terms of moduli spaces of certain
lattice-polarized K3 surfaces of Picard rank $2$ over $F$, allowing
general ADE singularities.  The constructions we use here will also
help prepare us for the larger rank cases to follow in later sections.

\begin{theorem} \label{thm:rr}
Let $V_1$, $V_2$ and $V_3$ be $4$-dimensional vector spaces over
$\fd$. Let $G' = \GL(V_1) \times \GL(V_2) \times \GL(V_3)$, and let
$V$ be the representation $V_1 \otimes V_2 \otimes V_3$ of $G'$.  Let
$G$ be the quotient of $G'$ by the kernel of the multiplication map on
scalars, i.e., $\Gm \times \Gm \times \Gm \to \Gm$.  Let $\Lambda$ be
the lattice whose Gram matrix is
\[
\begin{pmatrix}
4 & 6 \\
6 & 4
\end{pmatrix},
\] 
and let $S = \{e_1, e_2\}$. Then the $G(\fd)$-orbits of an open subset
of $V(\fd)$ are in bijection with the $\fd$-points of an open
subvariety of the moduli space $\sM_{\Lambda,S}$ of nonsingular K3
surfaces lattice-polarized by $(\Lambda,S)$.
\end{theorem}

\subsection{Construction of K3 surfaces}\label{subsec:rrconst}

We first describe the construction of a K3 surface from an element
$\cube \in V = V_1 \tns V_2 \tns V_3$, where $V_1$, $V_2$, and $V_3$
are $4$-dimensional $\fd$-vector spaces.  With bases for $V_1$, $V_2$,
and $V_3$, we may view $\cube$ as a $4 \times 4 \times 4$ cubical
matrix $(a_{ijk})_{1 \leq i, j, k \leq 4}$ with entries in $\fd$.  For
any $x \in V_1^\vee$, we may obtain a $4 \times 4$ matrix $\cube \subs
x$ of linear forms in $x$. The determinant of this matrix is a form
$f$ of degree $4$ in four variables, and its vanishing locus is a
quartic surface $X_1$ in $\Proj(V_1^\vee) \cong \Proj^3$.  We restrict
our attention to the general case where $X_1$ has at most simple
isolated singularities, which are thus K3 surfaces; in this case, we
say that $A$ is {\em nondegenerate}.

Similarly, we may repeat this construction in the other two directions
(replacing $V_1$ with $V_2$ or $V_3$) to obtain two more K3 surfaces
$X_2 \subset \Proj(V_2^\vee)$ and $X_3 \subset \Proj(V_3^\vee)$. We
claim that these three K3 surfaces are birational to each other.  For
example, to exhibit the map $X_1 \dashrightarrow X_2$, we view $\cube$
as a trilinear form on $V_1^\vee \times V_2^\vee \times V_3^\vee$.
Then let
\begin{equation*}
X_{12} := \left\{ (x,y) \in \Proj(V_1^\vee) \times \Proj(V_2^\vee) : \cube(x,y,\ccdot) = 0 \right\}.
\end{equation*}
Then we observe that the projections of $X_{12}$ to $\Proj(V_1^\vee)$
and $\Proj(V_2^\vee)$ are $X_1$ and $X_2$, respectively, thereby
giving a correspondence between $X_1$ and $X_2$.  In particular, given
a point $x \in X_1$, the determinant of $\cube \subs x$ vanishes, and
the $y \in X_2$ such that $(x,y) \in X_{12}$ are exactly those $y$ (up
to scaling) in the kernel of $\cube \subs x$ in $V_2^\vee$.  We claim
that if the kernel is at least $2$-dimensional, then the point $x \in
X_1$ is a singular point.  Indeed, if all of the $3 \times 3$ minors
$\cube^*_{st}(x)$ of $\cube \subs x$ vanish, for $1 \leq s, t \leq 4$,
then so do the partial derivatives
\[
\frac{\partial f}{\partial x_i}(x) = \sum_{s,t} c_{ist} \cube^*_{st}(x),
\]
where $c_{ist} = (-1)^{s+t} a_{ist}$. Hence if $x\in X_1$ is
nonsingular, then the kernel of $\cube \subs x$ is exactly
1-dimensional.  Generically, if the kernel of $\cube \subs x$ is
2-dimensional, then $x$ gives an isolated singularity of
$X_1$,\footnote{Furthermore, if the kernel is $3$-dimensional, then
  the surface $X_1$ is a rational surface.} which we call a {\it rank
  singularity} of $X_1$.  (It is possible for $X_1$ to have isolated
singularities that are not rank singularities;
see~\S\ref{subsec:hyperdet} for a further discussion of singularities
on these surfaces.)

This describes a map $\psi_{12}: X_1 \dashrightarrow X_2$, and it is
easy to see that it is generically an isomorphism, as we may construct
the inverse map $\psi_{12}^{-1}=\psi_{21}:X_2\dashrightarrow X_1$ in
the analogous manner. Similarly, we have maps $\psi_{ij} =
\psi_{ji}^{-1}$ for all $1 \leq i \neq j \leq 3$. However, we note
that the composition $\Phi := \psi_{31} \circ \psi_{23} \circ
\psi_{12}$ is not the identity!
The resulting automorphism will be discussed further in~\S
\ref{sec:rrauts}.

The isomorphism classes of the K3 surfaces $X_i$ and maps $\psi_{ij}$
are invariant under the action of the group $G$.  As there is a finite
stabilizer group for a generic point in $V$ (in fact, the stabilizer
is trivial; see Lemma \ref{lem:444stab} below), the dimension of the
moduli space of K3 surfaces obtained in this way is $64 - 46 = 18$.

\subsection{N\'eron-Severi lattice} \label{sec:RR-NS}

We will see below that the N\'eron-Severi lattices of these K3
surfaces all contain a particular $2$-dimensional lattice with Gram
matrix
\begin{equation} \label{eq:intmatrix444}
\begin{pmatrix}
4 & 6 \\ 6 & 4
\end{pmatrix}.
\end{equation}
The space of K3 surfaces with this lattice polarization has dimension
$20 - 2 = 18$. Therefore, we see that the N\'eron-Severi lattice of a
generic K3 surface in this family will be this $2$-dimensional lattice
above.

To understand the N\'eron-Severi group of a K3 surface in our family,
say $X_1 = X_1(\cube)$ for a particular choice of $\cube$ and bases
for the vector spaces, we proceed as follows. Let $W$ be the vanishing
locus in $X_1$ of the top left $3 \times 3$ minor of
$\cube(x,\ccdot,\ccdot)$; note that $W$ contains, in particular, all
the isolated rank singularities of $X_1$. The maps $\psi_{12}$ and
$\psi_{13}$ can be expressed by the minors of the last row and column,
respectively, of $\cube(x,\ccdot,\ccdot)$ (with the appropriate
signs).  Note that each of these two sets of minors contains the top
$3\times 3$ minor of $\cube(x,\ccdot,\ccdot)$.  Hence we see that $W$
contains a divisor equivalent to $C = \psi_{12}^*(L_2)$, where $L_2$
is the hyperplane class of $X_2 \subset \Proj(V_2^\vee)$, and
similarly $W$ contains a divisor equivalent to $D =
\psi_{13}^*(L_3)$. By direct calculation, we observe that the scheme
$W$ is reducible, and generically decomposes into two components,
which must therefore be $C$ and $D$. In the N\'eron-Severi lattice, we
therefore have
\begin{equation} \label{eq:444relation}
3H = W = C + D + \sum_i E_i,
\end{equation}
where $H = L_1$ is the hyperplane class of $X_1$ and the $E_i$ are the
exceptional divisors over the isolated rank singularities. In the
generic case, there are no exceptional divisors $E_i$.

We now compute the intersection numbers involving $H$ and $C$,
assuming there are no exceptional divisors.  We have $H^2 = 4$,
and $C^2 = \langle \psi_{12}^*(L_2), \psi_{12}^*(L_2) \rangle = L_2^2
= 4$ and similarly $D^2 = 4$. So we obtain $36 = (3H)^2 = C^2 + D^2 +
2 C \cdot D = 4 + 2 C \cdot D$, giving $C \cdot D = 14$. Therefore $C
\cdot 3H = C \cdot (C + D) = 4 + 14 = 18$, leading to $C \cdot H =
6$. The divisors $H$ and $C$ thus have the intersection matrix
\eqref{eq:intmatrix444}.

\begin{proposition}
The Picard group of the K3 surface $X_1$ corresponding to a very
general point (in the moduli space of Rubik's revenge cubes) is
generated by the classes of $C$ and $H$.
\end{proposition}
\begin{proof}
The discriminant of the lattice generated by $C$ and $H$ is $20 = 2^2
\cdot 5$, so it is enough to check that it is $2$-saturated. Since
$C/2$ and $H/2$ have self-intersection $1$, which is odd, neither of
these classes are in $\NS(X)$. Similarly, $(C+H)/2$ has
self-intersection $5$. Therefore $\NS(X) = \Z C + \Z H$.
\end{proof}

\begin{lemma} \label{lem:444stab} A quartic surface $X = X_1$
  associated to a very general point in the moduli space of
  $(\Lambda,S)$-polarized K3 surfaces has no linear automorphisms
  $($i.e., induced from $\PGL_4)$ other than the identity.
\end{lemma}

\begin{proof}
From \cite{nikulin}, we have the following description of the
automorphism group. Let $O^+(\NS(X))$ be the set of isometries of the
N\'eron-Severi lattice which preserve the K\"ahler cone and
$O_\omega(\T(X))$ be the set of isometries of the transcendental
lattice which preserve the period $\omega$ of the K3 surface, up to
$\pm 1$. Then
\[
\Aut(X) \cong \{ (g,h) \in O^+(\NS(X)) \times O_\omega(\T(X)) : \bar{g} = \bar{h} \},
\]
where $\,\bar{}\,$ refers to the natural morphisms from the orthogonal
groups of the lattices $\NS(X)$ or $\T(X)$ to their discriminant
groups, which are isomorphic.

For a general element of the moduli space, the only Hodge isometries
of the transcendental lattice are $h = \pm \id$. Suppose $g$ preserves
the class of $H$ and $g(C) = mC + nH$ for some $m,n \in \Z$. Since
$g(H) \cdot g(C) = H \cdot C = 6$, we obtain $6m + 4n = 6$. Similarly,
$g(C) \cdot g(C) = C \cdot C = 4$ gives $4m^2 + 4n^2 + 12mn =
4$. Combining these, we get $(m,n)$ = $(1,0)$ or $(-1,3)$.  In the
first case, we have $g = \id$ and by the condition on the discriminant
group (which is not $2$-torsion), we see that $h = \id$ is forced,
leading to the identity automorphism of $X$. In the second case, we
see that $C$ and $D$ are switched under $g$; however, since $g$ does
not act by $\pm 1$ on the discriminant group, which is generated by
$(H+C)/10$ and $H/2$, it does not give an automorphism of $X$.
\end{proof}

\begin{corollary}
The stabilizer of the action of $G$ on $V$ is generically trivial.
\end{corollary}

\begin{proof}
  If $g = (g_1, g_2, g_3)$ stabilizes $v \in V$, then $g_i$ gives a
  linear automorphism of $X_i$ for each $i$. Therefore, generically $g
  = 1$.
\end{proof}

\subsection{Moduli problem}

This subsection contains the proof of Theorem \ref{thm:rr}.  We have
already given a construction from an element of $V_1 \otimes V_2
\otimes V_3$ to a $(\Lambda, S)$-polarized K3 surface.  The bulk of
the proof is to show the reverse construction. We start with a
well-known lemma; a simple proof may be found in \cite{mayer}
with more details in \cite{SaintDonat}. We include this proof below, 
since it is a useful template for the proofs of this section. 

\begin{lemma}
  Let $(X,L)$ be a generic point in the moduli space $\sM_4$ of K3
  surfaces equipped with a line bundle $L$ with $L^2 = 4$. Then the
  linear system $|L|$ embeds $X$ as a quartic surface in $\Proj^3$.
\end{lemma}

\begin{proof}
  By Riemann-Roch, $\ch^0(L) + \ch^2(L) \geq 4$, so $L$ or $-L$ is
  effective. We may assume the former without loss of generality. For
  a generic point in the moduli space, the linear system $|L|$
  contains an irreducible curve $C$. By Bertini's theorem, we may even
  assume $C$ is smooth. It is not difficult to show that $\ch^1(L) =
  0$, so $\ch^0(L) = 4$. Therefore, the associated morphism $\phi_L$
  maps $X$ to $\Proj^3$. Either (i) $\deg(\phi) = 1$ and the image is
  a quartic surface in $\Proj^3$, or (ii) $\deg(\phi) = 2$ and the
  image is a quadric surface, and the curve $C$ is a double cover of a
  plane conic branched at $8$ points, and therefore a hyperelliptic
  curve of genus $3$. The second case does not occur generically (see,
  for instance, the argument in \cite[Remark~2.3.8 and
    Example~2.3.9]{huybrechts} or \cite[Exp.~VI]{k3notes}), and leads
  to an increase in the Picard number.
\end{proof}

\begin{remark}
The locus of K3 surfaces for which $|L|$ does not contain an
irreducible curve (alternatively, has a base locus, necessarily a
smooth rational curve) is a Noether--Lefschetz divisor. In this {\em
  unigonal} case, the complete linear system $|L|$ describes $X$ as an
elliptic surface over a twisted cubic in $\Proj^3$. The hyperelliptic
or {\em digonal case} (ii) in the proof above also corresponds to a
Noether--Lefschetz divisor.
\end{remark}

Most of the proof of Theorem \ref{thm:rr} will be established in the
following result, which we state separately, since it will also be useful
in subsequent sections.

\begin{theorem} \label{thm:masterRR}
  Let $X$ be a K3 surface equipped with two line bundles $L_1$, $L_2$
  such that $L_1^2 = L_2^2 = 4$ and $L_1 \cdot L_2 = 6$. Assume in
  addition that $L_1$ and $L_2$ correspond to effective divisors $C_1$
  and $C_2$ on $X$ that induce maps to $\Proj^3$ whose images are
  normal quartic surfaces. Then $X$ arises from a $4 \times 4 \times
  4$ matrix via the construction of $\S\ref{subsec:rrconst}$.
\end{theorem}

\begin{proof}
  We consider $X$ as a quartic surface in $\Proj^3$, embedded through
  the linear system $|L_1|$. Then $L_2$ corresponds to a
  non-hyperelliptic curve $C$ on $X$ of genus~$3$. Equivalently, $C$
  is projectively normal. It is well-known that the sheaf $\sO_X(C)$
  is arithmetically Cohen-Macaulay (see \cite{beauville} and for more
  general hypotheses, \cite[Chapter 4]{dolgachev-CAG}). Therefore,
  there is an exact sequence
  \[
  0 \to \sO_{\Proj^3}(-1)^4 \to \sO_{\Proj^3}^4 \to j_* L_2 \to 0,
  \]
  where $j: X \to \Proj^3$ is the embedding as a quartic
  surface. Taking the long exact sequence, and using $\ch^0(\Proj^3,
  \sO_{\Proj^3}(-1)) = \ch^1(\Proj^3, \sO_{\Proj^3}(-1)) = 0$, we have
  an identification of $\ch^0(\Proj^3, \sO_{\Proj^3}^4)$ with
  $\ch^0(X, L_2)$.

  Next, tensoring with the exact $L_1$ and taking cohomology, we obtain
  \[
  0 \to \cH^0(\sO_{\Proj^3})^4 \to \cH^0(\sO_{\Proj^3}(1))^4 \to
  \cH^0(j_* L_2 \tns L_1) \to \cH^1(\sO_{\Proj^3})^4 = 0.
  \]
  Thus we obtain a surjective map
  \begin{equation}\label{surjmap}
  \mu: \cH^0(X,L_1) \tns \cH^0(X,L_2) \to \cH^0(X, L_1 \tns L_2).
  \end{equation}
  Since each $\cH^0(X,L_i)$ is $4$-dimensional, the map has a
  $4$-dimensional kernel. Thus, we obtain a $4 \times 4 \times 4$
  matrix, giving rise to a determinantal representation of $X$.
\end{proof}

\begin{proof}[Proof of Theorem $\ref{thm:rr}$]
  Given a $4 \times 4 \times 4$ tensor, we have already seen how to
  produce a K3 surface $X$ with two line bundles $L_1$ and $L_2$ with
  the required pairing matrix.

  Conversely, given a K3 surface $X$ with line bundles $L_1$ and
  $L_2$, Riemann-Roch shows that either $L_1$ or its inverse is
  effective, and similarly for $L_2$. Normalizing so that $L_1$ and
  $L_2$ are effective, we see that generically (in the moduli space of
  lattice-polarized K3 surfaces) each gives a quartic embedding to
  $\Proj^3$. Therefore, we may use the result of
  Theorem~\ref{thm:masterRR} to produce a $4 \times 4 \times 4$
  tensor.

  It remains to show that these two constructions are inverse to one
  another.  Given a K3 surface $X$ with two line bundles $L_1$ and
  $L_2$ with intersection matrix \eqref{eq:intmatrix444}, let $Y_{12}$
  be the natural image of $X$ in $\Proj(\cH^0(X,L_1)^\vee) \times
  \Proj(\cH^0(X,L_2)^\vee)$ and let $Y_1$ and $Y_2$ be the projections
  onto the respective factors.

  On the other hand, construct the element $\cube \in \cH^0(X, L_1)
  \otimes \cH^0(X,L_2) \otimes (\ker \mu)^\vee$ from $(X, L_1, L_2)$
  as above, and let $X_{12}$, $X_1$, and $X_2$ be the K3 surfaces
  constructed from $\cube$ in the usual way.  We claim that $X_{12} =
  Y_{12}$ and $X_i = Y_i$ as sets and as varieties.

  By the construction of $\cube$ from the kernel of $\mu$, we have
  $\cube(x,y,\ccdot) = 0$ for any point $(x,y) \in Y_{12}$, so $Y_{12}
  \subset X_{12}$ and $Y_i \subset X_i$.  Now the quartic polynomial
  defining $X_1$ is not identically zero, because $\cube$ must have
  nonzero tensor rank. Therefore, $X_1$ and $Y_1$ are both given by
  quartic polynomials and must be the same variety, and similarly for
  $X_{12}$ and $Y_{12}$.

  Conversely, given a nondegenerate $\cube \in V_1 \otimes V_2 \otimes
  V_3$, let $X$ be the K3 surface $X_{12}$ constructed from $\cube$,
  and let $L_1$ and $L_2$ be the line bundles on $X$.  Then the vector
  spaces $V_1$ and $\cH^0(X,L_1)$ are naturally isomorphic, as are
  $V_2$ and $\cH^0(X,L_2)$, and $V_3^\vee$ may be identified with the
  kernel of the multiplication map $\mu$ in (\ref{surjmap}).  With
  these identifications, the element of $\cH^0(X,L_1) \otimes
  \cH^0(X,L_2) \otimes (\ker \mu)^\vee$ constructed from this
  geometric data is well-defined and $G$-equivalent to the
  original~$\cube$.
\end{proof}

\begin{remark} \label{rmk:genericity}
  Strictly speaking, we have not shown that, for a generic point of
  the moduli space $\sM = \sM_{\mathfrak{O}_{20}}$ of K3 surfaces
  lattice-polarized by the two-dimensional lattice $\mathfrak{O}_{20}$
  with matrix
  \[
  \begin{pmatrix}
    4 & 6 \\ 6 & 4
  \end{pmatrix},
  \]
  the two line bundles $L_1$ and $L_2$ give quartic embeddings---we
  have only showed this for K3 surfaces lattice-polarized by $\langle
  4 \rangle$. Let $B \subset \sM_4$ be the divisor in $\sM_4$
  corresponding to K3 surfaces for which the polarization is the class
  of a hyperelliptic curve; it is 18-dimensional. There are two
  obvious maps $\phi_i: \sM \to \sM_4$, taking $(X,L_1, L_2)$ to $(X,
  L_i)$.  For any value of $i \in \{0,1\}$, since $\sM$ is
  $18$-dimensional, in principle it is possible that the ``bad''
  subvariety $\phi_i^{-1}(B)$ of $\sM$ for which the polarization
  $L_i$ gives a hyperelliptic curve coincides with all of $\sM$.
  However, this does not happen, and there are at least two ways to
  see why. First, one may see it directly in this special example, as
  follows. Suppose $|L_1|$ gives a $2$-to-$1$ map $\phi$ to a quadric
  surface. Then we have $L_1 = E + F$, where $E$ and $F$ are pullbacks
  of the generators of the Picard group of the quadric surface. They
  satisfy $E^2 = F^2 = 0$ and $E \cdot F = 2$. However, the original
  $2$-dimensional lattice has no isotropic vectors, which implies that
  the locus of ``bad'' K3 surfaces is a Noether-Lefschetz divisor in
  $\sM$.

  Another more general way to see that generically $L_1$ and $L_2$
  should give quartic embeddings is the following: the locus $Z$ of K3
  surfaces for which the corresponding map is $2$-to-$1$ to a quadric
  surface, or is composed with a pencil, is closed in the moduli space
  $\sM$. Therefore, it suffices to show that the moduli space is
  irreducible, and to show that it contains a point outside $Z$. The
  first assertion follows (over $\C$) from the description as a
  quotient of a Hermitian symmetric domain, and the second from the
  ``forward'' construction which produces such a K3 surface from a $4
  \times 4 \times 4$ cube. We will use this more general method,
  without further mention, in the doubly and triply symmetrized cases
  of Rubik's revenge. The irreducibility follows from the uniqueness
  of the embedding of $\Lambda$ into the K3 lattice.
\end{remark}

\subsection{Automorphisms} \label{sec:rrauts}

Next, let us compute the action of $\Phi^*=(\psi_{31} \circ \psi_{23}
\circ \psi_{12})^*$ on the part of the N\'eron-Severi lattice given by \eqref{eq:intmatrix444} for
K3 surfaces associated to generic orbits.
The relation \eqref{eq:444relation} holds also for the analogous
divisors on $X_2$ and $X_3$, and in the generic case, there are no
singularities.  We therefore have
\begin{align*}
	3 H &= C + D \\
	3 C &= \Phi^*(D) + H\\
	3 \Phi^*(D) &= \Phi^*(H) + C \\
	3 \Phi^*(H) &= \Phi^*(C) + \Phi^*(D),
\end{align*}
where each relation is the analogue of \eqref{eq:444relation} for
$X_1$, $X_2$, $X_3$, and then $X_1$ again, when applying the $\psi_{ij}$ in
$\Phi$ in turn.  Thus, the automorphism $\Phi^*$ acts on the
sublattice $N_0 := \Z H + \Z C \cong \mathfrak{O}_{20}$ of $\NS(X_1)$
by the matrix
\[
M = \left( 
\begin{array}{cc}
-3 & 8 \\ -8 & 21
\end{array}
\right)
\] 
in the basis $(H,C)$. It describes an automorphism of infinite order,
and in fact
\[
M^n = \left( 
\begin{array}{cc}
-F_{6n -2} & F_{6n} \\ - F_{6n} & F_{6n + 2}
\end{array}
\right)
\] 
where the $F_n$ denote the Fibonacci numbers $F_0 = 0$, $F_1 = 1$, $F_n = F_{n-1}
+ F_{n-2}$. The group generated by~$M$ has index 6 in the integral
orthogonal group $O(N_0,\Z)$ of $N_0$. For a very general such $X$
(that is, if $\NS(X) = N_0$), it can be shown that $\Phi$ generates
$\Aut(X)$.

Note that the automorphism $\Phi$ of $X$ is the same as the
automorphism considered by Cayley \cite[\S 69]{cayley}, and more
recently in the context of dynamics on K3 surfaces by Oguiso
\cite{oguiso}, who showed that for those $X$ having Picard number 2,
the automorphism $\Phi$ is fixed-point-free and has positive entropy
(see also \cite{festi} for more on this case).  In
\S\ref{subsec:hyperdet}, we will give a simple proof of this theorem,
as well as of various extensions, using hyperdeterminants. 

\section{Doubly symmetric Rubik's revenge: \texorpdfstring{$4 \tns \Sym^2(4)$}{4 (x) Sym2(4)}}
\label{sec:2symrr}

We now consider doubly symmetric $4\times 4\times 4$ cubical matrices,
namely elements of the space $V_1\otimes\Sym^2 V_2$ for 4-dimensional
$\fd$-vector spaces $V_1$ and $V_2$.  Since the natural injection of
$V_1 \otimes \Sym^2 V_2$ into $V_1 \otimes V_2 \otimes V_2$ is
equivariant for the $\GL(V_1) \times \GL(V_2)$-actions, one can
understand the $\GL(V_1) \times \GL(V_2)$-orbits of $V_1\otimes\Sym^2
V_2$ using Theorem~\ref{thm:rr}.

However, there are some important differences in the geometric data
attached to a general $4\times4\times4$ cube compared to that attached
to a symmetric one.  For a general $4\times 4\times 4$ cube, the three
resulting K3 surfaces are {\it nonsingular}.  The basic reason is that
in the $\Proj^{15}$ of $4\times 4$ matrices, the variety of matrices
having rank at most two is 11-dimensional and thus will not intersect
a general $\Proj^3\subset \Proj^{15}$ spanned by four $4\times 4$
matrices. As a result, the corresponding determinantal quartic surface
will have no rank singularities and will in fact generically be
smooth.

In the $\Proj^{9}$ of {\it symmetric} $4\times 4$ matrices, the
matrices having rank at most two form a 6-dimensional variety of
degree 10, namely, the secant variety to the image of the Veronese
embedding $\Proj^3 \hookrightarrow \Proj^9$. A general $\Proj^3\subset
\Proj^{9}$ spanned by four $4\times 4$ matrices will intersect the
variety of matrices of rank $\leq 2$ in a zero-dimensional subscheme
of degree $10$; consequently, our determinantal quartic surface will
have 10 isolated rank singularities, which are in fact nodes, and
generically, there will be no other singularities.

These K3 surfaces, cut out by determinants of a symmetric $4 \times 4$
matrix of linear forms, were also classically studied, and are called
{\em quartic symmetroids} \cite{cayley, jessop, cossec}.  We prove
that the general orbits of tensors in $V_1\otimes\Sym^2 V_2$
correspond to certain K3 surfaces with Picard rank at least $11$ over
$\fdbar$:

\begin{theorem} \label{thm:sym2rr}
Let $V_1$ and $V_2$ be $4$-dimensional vector spaces over $\fd$. Let
$G' = \GL(V_1) \times \GL(V_2)$, and let $G$ be the quotient of $G'$
by the kernel of the natural multiplication map on scalars $\Gm \times
\Gm \to \Gm$ sending $(\gamma_1, \gamma_2)$ to $\gamma_1 \gamma_2^2$.
Let $V$ be the space $V_1 \otimes \Sym^2 V_2$.  Let $\Lambda$ be the
lattice given by the Gram matrix
\[
\left(
\begin{array}{ccccccccccc}
\,4\, & \,6\, & 0 & 0 & 0 & 0 & 0 & 0 & 0 & 0 & 0 \\
 6 & 4 & 1 & 1 & 1 & 1 & 1 & 1 & 1 & 1 & 1 \\
 0 & 1 & -2 & 0 & 0 & 0 & 0 & 0 & 0 & 0 & 0 \\
 0 & 1 & 0 & -2 & 0 & 0 & 0 & 0 & 0 & 0 & 0 \\
 0 & 1 & 0 & 0 & -2 & 0 & 0 & 0 & 0 & 0 & 0 \\
 0 & 1 & 0 & 0 & 0 & -2 & 0 & 0 & 0 & 0 & 0 \\
 0 & 1 & 0 & 0 & 0 & 0 & -2 & 0 & 0 & 0 & 0 \\
 0 & 1 & 0 & 0 & 0 & 0 & 0 & -2 & 0 & 0 & 0 \\
 0 & 1 & 0 & 0 & 0 & 0 & 0 & 0 & -2 & 0 & 0 \\
 0 & 1 & 0 & 0 & 0 & 0 & 0 & 0 & 0 & -2 & 0 \\
 0 & 1 & 0 & 0 & 0 & 0 & 0 & 0 & 0 & 0 & -2
\end{array}
\right)
\] 
and let $S = \{e_1, e_2\}$. Then the $G(\fd)$-orbits of an open subset
of $V(\fd)$ are in bijection with the $\fd$-points of an open
subvariety of the moduli space $\sM_{\Lambda,S}$ of nonsingular K3
surfaces lattice-polarized by $(\Lambda,S)$.
\end{theorem}

For a generic $\cube \in V_1 \tns \Sym^2 V_2$, from the constructions
in the previous section, we obtain nonsingular quartic surfaces $X_2$
and $X_3$ by slicing the cube $\cube$ in two directions, whereas in
the third direction, we get a quartic surface $X_1$ with generically
ten $A_1$ singularities.  By symmetry, $X_2$ and $X_3$ are in fact
identical surfaces in $\Proj(V_2^\vee)$, but we will sometimes refer
to these separately in the sequel.

Because contracting the cube in the third direction gives a symmetric
matrix, we see that for a generic point $x \in X_1$, the left and
right kernels of $\cube \subs x$ are spanned by the same vector.  The
map $\psi_{12}: X_1 \dashrightarrow X_2$ has a base locus consisting
of the ten singularities on $X_1$ (since the kernel of $\cube \subs x$
at these ten points is (generically) two-dimensional), and hence it is
the minimal resolution of these singularities. Similarly, the map
$\psi_{21}: X_2 \rightarrow X_1$ is the blow-down of the exceptional
divisors, so it is just the map $\psi_{12}^{-1}$ as a rational map.
Furthermore, while $\psi_{13} \circ \psi_{21}$ is the identity map
from $X_2$ to $X_3$, the maps $\psi_{23}$ and $\psi_{32}$ are not the
identity map.

\subsection{N\'eron-Severi lattice}

We begin by describing a set of generators for the N\'eron-Severi
group for the K3 surface $X$ arising from a very general doubly
symmetric $4 \times 4 \times 4$ cubical matrix. Let $L_1$ be the
hyperplane class of $X = X_1$ and $L_2$ be the pullback of the
hyperplane class of $X_2$ via $\psi_{12}$. Finally, let $P_1, \dots,
P_{10}$ be the exceptional divisors corresponding to the ten singular
points.  While $L_1$, $L_2$, and $\sum_{i=1}^{10} P_i$ are defined
over $\fd$, the $P_i$ individually may not be.

\begin{proposition}
The Picard group of $X_{\fdbar}$ is generated by $L_1, L_2$ and the
classes of the $P_i$.
\end{proposition}

\begin{proof}
  We first observe that the dimension of the moduli space of quartic
  symmetroids is $10 \cdot 4 - 15 - 15 - 1 = 9$.  Hence the Picard
  number is at most $11$. Since the classes of $L_1$ and the ten $P_i$
  are all independent, the Picard number is exactly $11$ for a very
  general point on the moduli space, and there are exactly ten
  singular points on the associated quartic surface. We obtain the
  relation
  \begin{equation} \label{eq:2symRR-relation}
  3L_1 = 2L_2 + \sum P_i
  \end{equation}
  by specializing the relation \eqref{eq:444relation}.
  Hence a basis for the span of all these classes is given by
  $\{L_1, L_2$, $P_1, \dots, P_9 \}$. We easily compute that the
  discriminant of the lattice $\Lambda$ they span is $1024 = 2^{10}$
  and the discriminant group is $\Z/4\Z \oplus (\Z/2\Z)^8$.
  
  It is enough to show that $\Lambda$ is saturated in
  $\overline{\NS}(X)$. In fact, by computing the inverse of the Gram
  matrix, one immediately checks that any element of the dual lattice
  of $\Lambda$ must have the form
  \[
  D = \frac{c}{4} L_1 + \frac{1}{2} \Big(\sum_{i=1}^9 d_i  P_i \Big),
  \]
  for integers $c$ and $d_i$. Suppose $D$ is in the saturation
  $\Lambda'$ of $\Lambda$. If $c$ is odd, then $2D - \sum d_i P_i =
  \frac{c}2 L_1$ is also in $\Lambda'$, which is a contradiction since
  its self-intersection is odd. Therefore, we may assume that $D$ has
  the form
  \[
  D = \frac{c}{2} L_1 + \frac{1}{2} \big(\sum_{i=1}^9 d_i P_i\big).
  \]
By symmetry, it follows that $\frac{c}{2} L_1 + \frac{1}{2}
(\sum_{i=1}^8 d_i P_i) + d_9 P_{10} \in \overline{\NS}(X)$, and
therefore we have $\frac12{d_9} (P_9 - P_{10}) \in
\overline{\NS}(X)$. By Lemma \ref{lem:nikulin}, this forces $d_9$ to
be even. Similarly, all the $d_i$ are even, and then $\frac c 2 L_1
\in \overline{\NS}(X)$, which is a contradiction as above. This
concludes the proof.
\end{proof}

\subsection{Moduli problem}

\begin{proof}[Proof \,$1$ of Theorem $\ref{thm:sym2rr}$]
  As before, one direction has already been proved. Starting with a K3
  surface $X$ and line bundles $L_1$, $L_2$ and $P_1, \ldots, P_{10}$
  satisfying the desired intersection relations, we need to construct
  a doubly symmetric $4 \times 4 \times 4$ matrix, or equivalently, a
  symmetric $4 \times 4$ matrix of linear forms. This construction is
  described in, e.g., \cite[\S 4.2]{dolgachev-CAG} (see also \cite{tyurin,cossec});
  we briefly sketch the argument:
    
  Assume without loss of generality that $L_1$ and $L_2$ are
  very ample.  Let $Y$ be the image of the quartic embedding
  corresponding to the line bundle $L_1$. For each $i$, since $P_i^2 =
  -2$ and $P_i \cdot L_2 > 0$, we have that $P_i$ is effective and
  thus corresponds to a smooth rational curve on $X$. These collapse
  to singular points on $Y$, since $L_1 \cdot P_i = 0$. The surface
  $Y$ has ten singular points. Next, let $\sF$ be the pushforward of
  $L_2$ from $X$ to $Y$. We compute that $\sF^\vee$ has the divisor
  class $- L_2 - \sum P_i$. Therefore, the relation
  \[
  \sF \cong \sF^\vee(3)
  \]
  holds in the Picard group of $Y$, so the ACM sheaf $\sF$ gives a
  symmetric determinantal representation.

  Checking that these constructions are inverse to one another is a
  similar verification as in the proof of Theorem \ref{thm:rr}.
\end{proof}

We may give a second, more elementary proof of Theorem
\ref{thm:sym2rr}, using the construction in the proof of Theorem
\ref{thm:rr} together with the following lemma:

\begin{lemma}  \label{lem:kernelsmatch-sym}
  Let $B$ and $C$ be two $n \times n$ matrices over $\fd$ with $C$
  invertible.  Assume $B C^{-1}$ has distinct eigenvalues over
  $\fdbar$ and that for all $x, y \in \fdbar$, the transpose of the
  left kernel of $B x - C y$ is equal to its right kernel.  Then $B$
  and $C$ are symmetric matrices.
\end{lemma}

\begin{proof}
Since $\det(B C^{-1}-\lambda I)$ has distinct roots in $\fdbar$ by
assumption, the binary $n$-ic form $\det(B x- C y)=\det(C)\det(B
C^{-1}x - I y)$ has distinct roots $[x_i:y_i]$ ($i=1,\ldots,n$) in
$\Proj^1(\fdbar)$.  For each $i$, let $v_i$ be a nonzero vector in the
right kernel of $B x_i - C y_i$, implying that $v_i^t$ is a nonzero
vector in the left kernel.  The vectors $v_i$ are linearly
independent, because they are eigenvectors corresponding to the
distinct eigenvalues of $B C^{-1}$.
	
Consider the two symmetric bilinear forms $B(\ccdot,\ccdot)$ and
$C(\ccdot,\ccdot)$ defined by $B(w,z)=w^t B z$ and $C(w,z)=w^t C z$.
We wish to show that $B$ and $C$ are symmetric bilinear forms.
To see this, note that $v_i^t(Bx_i-Cy_i)v_j=v_i^t(Bx_j-Cy_j)v_j=0$ for
any $i\neq j$.  Since $(x_i,y_i)$ and $(x_j,y_j)$ are linearly
independent (as they yield distinct points in $\Proj^1(\fdbar)$), we
conclude that $v_i^tBv_j=v_i^tCv_j=0$ for any $i\neq j$.
	
It follows that $B$ and $C$ are diagonal bilinear forms with respect
to the basis $v_1,\ldots,v_n$.  Hence $B$ and $C$ are symmetric
bilinear forms, and thus correspond to symmetric matrices with respect
to any basis.
\end{proof}
  
\begin{proof}[Proof \,$2$ of Theorem $\ref{thm:sym2rr}$]
  Again, we only need to show that the geometric data gives rise to a
  doubly symmetric $4 \times 4 \times 4$ matrix.  Given $(X, L_1, L_2,
  P_1, \dots, P_{10})$, we use the multiplication map
  $$\mu: \cH^0(X, L_1) \otimes \cH^0(X, L_2) \to \cH^0(X, L_1 \tns
  L_2)$$ to obtain a $4 \times 4 \times 4$ matrix $A$ as before.  It
  remains to show that there exists an identification of $V_3 := (\ker
  \mu)^\vee$ and $V_2 := \cH^0(X, L_2)$ such that $A$ is doubly
  symmetric.  Let $V_1 := \cH^0(X, L_1)$.
  
  The proof of Theorem \ref{thm:rr} implies that $A$ in turn produces
  a K3 surface and line bundles isomorphic to those with which we
  started. In particular, the embedding corresponding to $L_1$ has ten
  singular points, since $L_1 \cdot P_i = 0$ implies that these
  $(-2)$-curves are contracted. Therefore, by applying equation
  \eqref{eq:444relation} of \S \ref{sec:RR-NS} and comparing with the
  relation \eqref{eq:2symRR-relation}, we deduce that the line bundles
  $L_2$ and $L_3$ are isomorphic. We therefore have an isomorphism
  $\phi: V_2 \stackrel{\sim}{\to} V_3$.  Let $X_1$ be the image of $X$
  via the quartic embedding given by $L_1$.  For any point $x \in
  X_1$, we have $\det(A(x,\ccdot,\ccdot)) = 0$ and the kernel of
  $A(x,\ccdot,\ccdot)$ in $V_2^\vee$ and in $V_3^\vee$ is the same
  (under $\phi$).  In other words, since $X_1$ spans
  $\Proj(V_1^\vee)$, the image of $V_1^\vee$ in $V_2 \otimes V_2$
  given by $(\id \otimes \phi) \circ A$ is a four-dimensional subspace
  of $V_2 \otimes V_2$ such that the ``left'' and ``right'' kernels of
  each element in $V_2^\vee$ are the same (usually empty, of course).

  We now wish to apply Lemma \ref{lem:kernelsmatch-sym} to any two
  generic matrices in this four-dimensional space.  For two
  nonsingular elements $B$ and $C$ of $V_2 \otimes V_2$, the matrix
  $BC^{-1}$ will have distinct eigenvalues over $\fdbar$ if the binary
  $n$-ic form $\det (Bx - Cy)$ has distinct roots, in which case $B x
  - C y$ has rank at least $3$ for any values of $x$ and $y$.  Recall
  that the K3 surface $X_1$ has only a finite number of isolated
  singularities, points $x \in \Proj(V_1^\vee)$ where $A(x,\ccdot,
  \ccdot)$ has rank $2$.  For any line in $\Proj(V_1^\vee)$ not
  passing through one of those singularities, the corresponding pencil
  of matrices in $\Proj(V_2 \otimes V_2)$ will thus satisfy the
  conditions of the lemma.  That is, let $B$ and $C$ be nonsingular
  elements of $V_2 \otimes V_2$ such that their span does not contain
  an element with rank less than $3$. Lemma \ref{lem:kernelsmatch-sym}
  implies that $B$ and $C$ are symmetric.  We may repeat this process
  to obtain a basis for the image of $V_1^\vee$ in $V_2 \otimes V_2$
  only consisting of symmetric elements, thereby giving an element of
  $V_1 \otimes \Sym^2 V_2$ as desired.
  
  Since these constructions are inverse to one another in the proof of
  Theorem~\ref{thm:rr}, they are also inverse to one another here.
\end{proof}

\subsection{Automorphisms} \label{sec:2symRRauts}

The map $\Phi = \psi_{31} \circ \psi_{23} \circ \psi_{12}: X_1
\dashrightarrow X_1$ considered in Section \ref{sec:rrauts} can be
constructed in this situation as well. Though it is only a birational
automorphism of $X_1$, it can be lifted to an actual automorphism of
the blown-up nonsingular model $X_{12}$. This follows from the general
fact that a birational map between two minimal nonsingular algebraic
surfaces with non-negative Kodaira dimension is an isomorphism (see,
for instance, \cite[Theorem 10.21]{badescu}).

First, we observe that $\Phi$ is an involution.  indeed, the symmetry
implies that $\psi_{12} = \psi_{13}$, $\psi_{23} = \psi_{32}$, and
$\psi_{31} = \psi_{21}$, and thus $\Phi = \Phi^{-1}$.  We now compute
its induced action on the N\'eron-Severi group.

The main idea is the same as in \S \ref{sec:rrauts}: use the relation
\eqref{eq:444relation} repeatedly, as we apply the maps $\psi_{12}$,
$\psi_{23}$, $\psi_{31}$, and $\psi_{12}$ again.  Let $L_1$, $L_2$,
and $P_i$ for $1 \leq i \leq 10$ be the classes introduced earlier.
Then we obtain the following equations among these classes (written
additively):
\begin{align*}
	3 L_1 &= 2 L_2 + \sum_i P_i, &
	3 L_2 &= L_1 + L_2, \\
	3 \Phi^*(L_2) &= \Phi^*(L_1) + L_2, & 
	3 \Phi^*(L_1) &= 2 \Phi^*(L_2) + \sum_i \Phi^*(P_i).
\end{align*}
By checking intersection numbers, we compute that $\Phi^* L_1 = -3 L_1
+ 8 L_2$, $\Phi^* L_2 = -L_1 + 3 L_2$, and $\Phi^* P_i = 2 L_1 -
\sum_{j \neq i} P_j$.
The associated transformation matrix squares to the identity, as expected.

\begin{remark}
The automorphism group of a general quartic symmetroid contans a
subgroup of the automorphism group of a general nodal Enriques
surface. The latter group is a finite-index subgroup of the reflection
group $W_{2,4,6}$ corresponding to the Coxeter diagram of type
$T_{2,4,6}$, and was computed explicitly by Cossec and Dolgachev \cite{cossec-dolgachev}.
\end{remark}

\section{Triply symmetric Rubik's revenge: \texorpdfstring{$\Sym^3(4)$}{Sym3(4)}}
\label{sec:3symrr}

We consider next the triply symmetric Rubik's revenge, in order to
understand the orbits of $\Gm \times \GL(V)$ on $\Sym^3 V$ for a 4-dimensional vector space $V$.

Such a cube is doubly symmetric in all three directions, and the three
K3's arising from such a triply symmetric Rubik's revenge are
identical.  A generic such triply symmetric Rubik's revenge will thus
give rise to a K3 that has at least 10 singularities, and a numerical
example shows that we obtain exactly 10 singularities in general.

The quartic surface $X$ has been well studied in the classical
literature \cite{cayley, hutchinson1, hutchinson2, jessop}, as a {\em
  Hessian quartic symmetroid}, since the matrix of linear forms whose
determinant defines $X$ is the Hessian (the matrix of second partial
derivatives) of a single cubic form $F$ in four variables. For more
recent references, see \cite{cossec, dolga-keum}.

Generically, over an algebraically closed field, there are five planes
tangent (along a degenerate conic) to such a Hessian surface. If
$\ell_i$ are the linear forms defining these planes $Z_i$, the
equation of the quartic may be written as
\[
\frac{1}{a_1 \ell_1} + \dots + \frac{1}{a_5 \ell_5} = 0.
\]
for some constants $a_1, \dots, a_5$. The cubic form is $F = a_1
\ell_1^3 + \dots + a_5 \ell_5^3$. The ten singular points are given by
the intersections of all ten triples of the hyperplanes $Z_i$. In
addition, the surface contains ten special lines, which come from the
pairwise intersections of the $Z_i$. Thus, the singular points may be
labelled $P_{ijk}$ and the lines $L_{lm}$, with $P_{ijk}$ lying on
$L_{lm}$ exactly when $\{l,m\} \subset \{i,j,k\}$. Therefore, there
are three singular points on each special line and three special lines
passing through each singular point.  For $1 \leq i \neq j \leq 3$,
the maps $\psi_{ij}$ defined in \S \ref{sec:rr} are all
identical. Denoting them by $\psi$, it is clear that $\psi$ is a
birational involution on the K3 surface $X$, blowing up the ten
singular points $P_{ijk}$ and blowing down the $L_{lm}$. In fact, it
exchanges $P_{ijk}$ and $L_{lm}$, where $\{i,j,k,l,m\} =
\{1,2,3,4,5\}$.

We show that the general orbits of tensors in $\Sym^3 V$ correspond to
certain K3 surfaces with Picard rank at least $16$ over $\fdbar$.

\begin{theorem} \label{thm:sym3rr}
Let $V_1$ be a $4$-dimensional vector space over $\fd$. Let $G' = \Gm
\times \GL(V_1)$ and $G$ be the quotient of $G'$ by the kernel of the
multiplication map $\Gm \times \Gm \to \Gm$ given by $(\lambda_1,
\lambda_2) \mapsto \lambda_1 \lambda_2^3$.  Let $V$ be the
representation $\Sym^3 V_1$ of $G$.  Let $\Lambda$ be the lattice
given by the Gram matrix

\begin{footnotesize}
\begin{equation}\label{symrrgram}
\begin{pmatrix}
4 & 6 & 1 & 1 & 1 & 1 & 1 & 0 & 0 & 0 & 0 & 0 & 0 & 0 & 0 & 0 \\
6 & 4 & 0 & 0 & 0 & 0 & 0 & 1 & 1 & 1 & 1 & 1 & 1 & 1 & 1 & 1 \\
1 & 0 & -2 & 0 & 0 & 0 & 0 & 1 & 1 & 1 & 0 & 0 & 0 & 0 & 0 & 0 \\
1 & 0 & 0 & -2 & 0 & 0 & 0 & 1 & 0 & 0 & 1 & 1 & 0 & 0 & 0 & 0 \\
1 & 0 & 0 & 0 & -2 & 0 & 0 & 0 & 1 & 0 & 1 & 0 & 1 & 0 & 0 & 0 \\
1 & 0 & 0 & 0 & 0 & -2 & 0 & 1 & 0 & 0 & 0 & 0 & 0 & 1 & 1 & 0 \\
1 & 0 & 0 & 0 & 0 & 0 & -2 & 0 & 0 & 0 & 1 & 0 & 0 & 1 & 0 & 0 \\
0 & 1 & 1 & 1 & 0 & 1 & 0 & -2 & 0 & 0 & 0 & 0 & 0 & 0 & 0 & 0 \\
0 & 1 & 1 & 0 & 1 & 0 & 0 & 0 & -2 & 0 & 0 & 0 & 0 & 0 & 0 & 0 \\
0 & 1 & 1 & 0 & 0 & 0 & 0 & 0 & 0 & -2 & 0 & 0 & 0 & 0 & 0 & 0 \\
0 & 1 & 0 & 1 & 1 & 0 & 1 & 0 & 0 & 0 & -2 & 0 & 0 & 0 & 0 & 0 \\
0 & 1 & 0 & 1 & 0 & 0 & 0 & 0 & 0 & 0 & 0 & -2 & 0 & 0 & 0 & 0 \\
0 & 1 & 0 & 0 & 1 & 0 & 0 & 0 & 0 & 0 & 0 & 0 & -2 & 0 & 0 & 0 \\
0 & 1 & 0 & 0 & 0 & 1 & 1 & 0 & 0 & 0 & 0 & 0 & 0 & -2 & 0 & 0 \\
0 & 1 & 0 & 0 & 0 & 1 & 0 & 0 & 0 & 0 & 0 & 0 & 0 & 0 & -2 & 0 \\
0 & 1 & 0 & 0 & 0 & 0 & 0 & 0 & 0 & 0 & 0 & 0 & 0 & 0 & 0 & -2 \\
\end{pmatrix}
\end{equation}
\end{footnotesize}%
and let $S = \{e_1, e_2\}$. Then the $G(\fd)$-orbits of an open subset
of $V(\fd)$ are in bijection with the $\fd$-points of an open
subvariety of the moduli space $\sM_{\Lambda,S}$ of nonsingular K3
surfaces lattice-polarized by $(\Lambda,S)$.
\end{theorem}

\subsection{N\'eron-Severi lattice}
The Picard group of the generic Hessian surface (base changed to
$\overline{F}$) is spanned over $\Z$ by the classes of the lines
$L_{ij}$ and the exceptional curves corresponding to the singular
points $P_{ijk}$. The lattice spanned by these has rank $16$ and
discriminant $-48$. Since its discriminant group is $\Z/3\Z \oplus
(\Z/2 \Z)^4$, a case-by-case argument shows that this lattice is the
full Picard group. We omit this proof, since the result is established
in \cite{dolga-keum} (using elliptic fibrations) and by a different
method in \cite{dardanelli-vangeemen}.

\subsection{Moduli problem}

We now proceed to the proof of Theorem \ref{thm:sym3rr}. Given a
triply symmetric $4 \times 4 \times 4$ cube and the resulting K3
surface $X$, we have seen that the Picard group of $X_{\fdbar}$ is
generated by the classes of the ten nodes and lines. Let $H_1$ and
$H_2$ be the hyperplane classes for $X_1$ and $X_2$ respectively. The
set
\[
\{ H_1, H_2, L_{12}, L_{13}, L_{14}, L_{23}, L_{34}, P_{123}, P_{124}, P_{125}, P_{134}, P_{135}, P_{145}, P_{234}, P_{235}, P_{245} \}
\]
is easily checked to be a basis for $\overline{\NS}(X)$, yielding the
Gram matrix (\ref{symrrgram}).  This data is fixed up to isomorphism
by the action of $\Gm \times \GL(V_1)$.

Conversely, given a $(\Lambda, S)$-polarized K3 surface, we construct
a triply symmetric Rubik's revenge by using the second proof of
Theorem \ref{thm:sym2rr}.  In particular, we may use that proof to
construct a $4 \times 4 \times 4$ cube $A$ in $V_1 \otimes V_2 \otimes
V_3$, for certain four-dimensional vector spaces $V_i$, where there is
an isomorphism $\phi_{32}: V_3 \to V_2$ so that $A$ is symmetric,
i.e., maps to an element of $V_1 \otimes \Sym^2 V_2$ under $\id
\otimes \id \otimes \phi_{32}$.  Here, we may also use the same proof
to show that $A$ is symmetric under an isomorphism $\phi_{21}: V_2 \to
V_1$ , i.e., gives an element of $\Sym^2 V_1 \otimes V_3$ under the
map $\id \otimes \phi_{21} \otimes \id$.  Thus, since the
transpositions $(12)$ and $(23)$ generate $S_3$, we may use
$\phi_{12}$ and $\phi_{23}$ to identify all three vector spaces and
obtain an element of $\Sym^3 V_1$.

\subsection{Automorphisms}

The automorphism group of the Hessian quartic surface is quite
large. In \cite{dolga-keum}, Dolgachev and Keum identified a set of
generators for the automorphism group. However, the relations between
these are not completely known, so a complete presentation for
the automorphism group is still unknown.

To connect with the earlier sections, we note that the maps $\psi$ are
birational involutions (recall that they all are identical).  In this
case, each $\psi$ is also the same as the $3$-cycle $\Phi$ described
in \S \ref{sec:rrauts} because of the symmetry.  We described above
the action induced by this involution on $\overline{\NS}(X)$: the
divisors $H_1$ and $H_2$ are switched, and the $P_{ijk}$ and $L_{lm}$
are switched for $\{i,j,k,l,m\} = \{1,2,3,4,5\}$.  Again, this is
studied extensively in \cite{dolga-keum}.


\section{Penteracts (or 5-cubes): \texorpdfstring{$2 \tns 2 \tns 2 \tns 2 \tns 2$}{2 (x) 2 (x) 2 (x) 2 (x) 2}} \label{sec:penteracts}

Consider the representation $V = V_1 \tns V_2 \tns V_3 \tns V_4 \tns
V_5$, where each $V_n$ is a $2$-dimensional $\fd$-vector space, of the
group $G' = \GL(V_1) \times \GL(V_2) \times \GL(V_3) \times \GL(V_4)
\times \GL(V_5)$.  With a choice of bases for each $V_n$ for
$n\in\{1,\ldots,5\}$,
an element $A \in V(\fd)$ may be visualized as a $5$-dimensional cube,
or {\em penteract}, with entries $a_{ijklm} \in \fd$ for $i, j, k, l,
m \in\{1,2\}$.  The space of penteracts is extremely rich, and indeed
the next several sections, through \S \ref{sec:5sympent}, will focus
on variations of this space of penteracts.

Let $G$ be the quotient of $G'$ by the kernel of the multiplication
map $\Gm^5 \to \Gm$.  In this section, we will study the
$G(\fd)$-orbits on $V(\fd)$, and in particular, describe the
relationship between (an open subvariety of) the orbit space
$V(\fd)/G(\fd)$ and the moduli space of certain K3 surfaces having
N\'eron-Severi rank at least 4:

\begin{theorem} \label{thm:penteractorbits}
Let $V = V_1 \tns V_2 \tns V_3 \tns V_4 \tns V_5$, where each $V_n$ is
a $2$-dimensional $\fd$-vector space.  Let $G' = \GL(V_1) \times
\GL(V_2) \times \GL(V_3) \times \GL(V_4) \times \GL(V_5)$, and let $G$
be the quotient of $G'$ by the kernel of the multiplication map $\Gm^5
\to \Gm$.  Let $\Lambda$ be the lattice whose Gram matrix is
\begin{equation}
\begin{pmatrix}
0 & 2 & 2 & 2 \\
2 & 0 & 2 & 2 \\
2 & 2 & 0 & 2 \\
2 & 2 & 2 & 0
\end{pmatrix},
\end{equation}
and let $S = \{ e_1, e_2, e_3, e_4\}$.  Then the $G(\fd)$-orbits of an
open subset of $V(\fd)$ are in bijection with the $\fd$-rational
points of an open subvariety of the moduli space $\sM_{\Lambda,S}$ of
K3 surfaces $X$ lattice-polarized by $(\Lambda,S)$.
\end{theorem}

\subsection{Constructions of K3 surfaces} \label{sec:penteractconstruct}

Given a general $A \in V(\fd)$, we construct a K3 surface with Picard
number at least $4$ as follows.  First, let
\[
X_{123} := \{ (v,w,x) \in \Proj(V_1^\vee) \times \Proj(V_2^\vee)
\times \Proj(V_3^\vee) : \det A(v,w,x,\ccdot,\ccdot) = 0 \}.
\] 
We similarly define $X_{ijk}$ for any subset $\{i,j,k \}$ in
$\{1,2,3,4,5\}$ (where permutation of the indices does not change the
variety).  The equation defining each $X_{ijk}$ is a tridegree
$(2,2,2)$ form in $\Pone \times \Pone \times \Pone$ and thus $X_{ijk}$
is generically a K3 surface; specifically, we only allow isolated
rational double point singularities.  Also, let
\[
X_{1234} := \{ (v,w,x,y) \in  \Proj(V_1^\vee) \times \Proj(V_2^\vee) \times \Proj(V_3^\vee) \times \Proj(V_4^\vee) : A(v,w,x,y,\ccdot) = 0 \},
\]
and define $X_{ijkl}$ for any subset $\{i,j,k,l \}$ in $\{1,2,3,4,5\}$
in the analogous way.  This variety $X_{ijkl}$ is the intersection of
two multidegree $(1,1,1,1)$ forms in $\Pone \times \Pone \times \Pone
\times \Pone$, which is also generically a K3 surface. In other words,
we can view the K3 surface as the base locus of a pencil of
divisors of type $(1,1,1,1)$ in $(\Proj^1)^4$. Note that the
projections from $X_{ijkl}$ to $\Proj(V_i^\vee)$ are genus one
fibrations.

For any permutation $\{i,j,k,l,m\}$ of $\{1,2,3,4,5\}$, there exists a
projection $\pi: X_{ijkl} \to X_{ijk}$, which is an isomorphism for
the generic $A \in V(\fd)$.  The fiber of each point $(v,w,x) \in
X_{ijk}$ is determined by the kernel of the singular map $A(v,w,x):
V_l^\vee \to V_m$.  If $A(v,w,x,\ccdot,\ccdot)$ is the zero matrix,
then $X_{ijk}$ is singular at $(v,w,x)$, and we then call $(v,w,x)$ a
{\it rank singularity}.  For nonsingular points $(v,w,x)$ on
$X_{ijk}$, the bilinear form $A(v,w,x,\ccdot,\ccdot)$ has a
one-dimensional kernel, and the fiber of $(v,w,x)$ under the map $\pi$
is a single point, given algebraically.  In particular, we see that if
$X_{ijk}$ is nonsingular, then it is isomorphic to $X_{ijkl}$.

More generally, if $X_{ijk}$ has an isolated rank singularity at
$(v,w,x)$, then the fiber of $\pi$ at $(v,w,x)$ is the entire line
$\{(v,w,x,y) \in \Proj(V_i^\vee) \times \Proj(V_j^\vee) \times
\Proj(V_k^\vee) \times \Proj(V_l^\vee): y \in \Proj(V_l^\vee) \}$.
Since such rational double point singularities are blown up in one
step, the surfaces $X_{ijkl}$ are nonsingular even when the $X_{ijk}$
have isolated rank singularities.  We call a penteract {\em
  nondegenerate} if the surfaces $X_{ijkl}$ are all nonsingular.

\subsection{N\'eron-Severi lattice} \label{sec:penteractNS}

We now compute the N\'eron-Severi lattice of $X = X_{1234}$.  The
surface $X$ comes equipped with the four line bundles $L_i$, given by
the pullbacks of $\sO_{\Proj(V_i^\vee)}(1)$ for $1 \leq i \leq 4$.
Their intersection numbers are easy to compute: because of the
description of $X_{ijk}$ as the vanishing of a tridegree $(2,2,2)$
form, we have $L_i \cdot L_j = 2(1-\delta_{ij})$.  As the dimension of
the orbit space is $2^5 - (3 \cdot 5 + 1) = 16$, the dimension of
$\NS(X)$ for a generic $X$ in this family must be exactly $4$.
Therefore, although we will find other natural divisors, these four
$L_i$ generate the Picard group of the generic K3 surface in this
family.

There are other line bundles given by, for example, considering the
pullback of $\sO_{\Proj(V_5^\vee)}(1)$ via the isomorphisms $X =
X_{1234} \to X_{123} \to X_{1235}$ followed by the projection to
$\Proj(V_5^\vee)$.  This particular line bundle $L_5^{(123)}$
satisfies the following:

\begin{lemma} \label{lem:pentrelation}
If $X_{123}$ is nonsingular $($and therefore isomorphic to $X)$, then
we have the relation
\begin{equation} \label{eq:penteractrelation}
L_1 + L_2 + L_3 =  L_4 + L_5^{(123)}
\end{equation}
among the above line bundles on $X$.  More generally, if $X_{123}$ has
isolated rational double point singularities, we have
\begin{equation} \label{eq:penteractrelation2}
L_1 + L_2 + L_3 =  L_4 + L_5^{(123)} + \sum_i E_i
\end{equation}
where the $E_i$ are the line bundles corresponding to the exceptional
divisors on $X$ arising from the singularities on $X_{123}$.
\end{lemma}

\begin{proof}
We first assume that $X_{123}$ is smooth.  The rational maps $\nu_4$
and $\nu_5$ from $X_{123}$ to $\Proj(V_4^\vee)$ and $\Proj(V_5^\vee)$,
which define $L_4$ and $L_5^{(123)}$, respectively, are each given by
the appropriate kernel of $A(v,w,x,\ccdot,\ccdot)$ in $V_4 \tns V_5$.
With a choice of basis vectors, let $$A(v,w,x,\ccdot,\ccdot)
= \begin{pmatrix} A_{11} & A_{12} \\ A_{21} & A_{22} \end{pmatrix},$$
where each $A_{ij}$ is a tridegree $(1,1,1)$ form on $\Proj(V_1^\vee)
\times \Proj(V_2^\vee) \times \Proj(V_3^\vee)$.  Then $\nu_4$ and
$\nu_5$ are given by, e.g., the forms $[- A_{21} : A_{11}]$ and $[-
  A_{12} : A_{11}]$, respectively.  It is easy to check that the line
bundle $\nu_4^* \sO_{\Proj(V_4^\vee)}(1)$ is isomorphic to
$\sO(Z(A_{11},A_{12}))$, where $Z(A_{11},A_{12})$ refers to the common
zero locus of those two forms; similarly, $\nu_5^*
\sO_{\Proj(V_5^\vee)}(1)$ is isomorphic to $\sO(Z(A_{11},A_{21}))$.
Thus, the right side of \eqref{eq:penteractrelation} is isomorphic to
$\sO(Z(A_{11}))$, and thus to the pullback of $\sO(1,1,1)$ from
$\Proj(V_1^\vee) \times \Proj(V_2^\vee) \times \Proj(V_3^\vee)$ to
$X$.

If $X_{123}$ has isolated rational double point singularities, then
for the singular points $(v,w,x)$, we have that
$A(v,w,x,\ccdot,\ccdot)$ is identically zero.  Thus, the divisors
$Z(A_{11})$ on $X_{123}$ contains components corresponding to $L_4$,
to $L_5^{(123)}$, and to each of the singularities, giving
\eqref{eq:penteractrelation2}.
\end{proof}

We will use versions of the relation of Lemma \ref{lem:pentrelation}
(with permuted indices, as necessary) to determine how divisor classes
interact in many of the subsequent sections.

\begin{proposition} 
For a very general $X$ in this family of K3 surfaces, $\NS(X)$ is
spanned over $\Z$ by $L_1$, $L_2$, $L_3$, and $L_4$.
\end{proposition}

\begin{proof}
Because this moduli space has dimension $32-16 = 16$, the Picard number
of such a very general $X$ is at most $4$, and
we know that $L_1, \dots, L_4$ span a finite
index subgroup of $\NS(X)$.  The discriminant of the lattice generated
by $L_1, L_2, L_3, L_4$ is $-48 = -2^4 \cdot 3$, and we only need to
check that it is $2$-saturated.

For $i \neq j$, the class $(L_i + L_j)/2$ cannot be integral, since
its self-intersection is odd, and similarly for $(L_i + L_j + L_k)/2$
for $i$, $j$, $k$ distinct. By symmetry, therefore, $L_i/2$ cannot be
in $\NS(X)$ (otherwise $L_i/2 + L_j/2$ would be). Finally, if $(L_1 +
L_2 + L_3 + L_4)/2 = L_4 + L_5^{(123)}/2$ were in $\NS(X)$, so would
$L_5^{(123)}/2$, and therefore all $L_i/2$ by symmetry, which is a
contradiction.
\end{proof}

\begin{lemma} \label{lem:22222stab} A K3 surface $X_{123}$ associated to
  a very general point in the moduli space of $(\Lambda,S)$-polarized
  K3 surfaces has no linear automorphisms $($i.e., induced from
  $\PGL_2 \times \PGL_2 \times \PGL_2)$ other than the identity.
\end{lemma}

\begin{proof}

  We proceed as in the proof of Lemma \ref{lem:444stab}. Let $g \in
  O^+(\NS(X))$ and $h \in O_\omega(\T(X))$ agree on the discriminant
  groups. As before we can assume $h = \pm \id$. On the other hand,
  $g$ fixes the classes of $L_1$, $L_2$ and $L_3$. Let
\[
g(L_4) = aL_1 + b L_2 + cL_3  + d L_4,
\]
for some integers $a,b,c,d$. Taking the intersection with $g(L_1) =
L_1$ through $g(L_3) = L_3$ and using $g(L_i) \cdot g(L_4) = L_i \cdot
L_4$, we obtain the equations
\begin{align*}
2 &= 0 + 2b + 2c + 2d \\
2 &= 2a + 0 + 2c + 2d \\
2 &= 2a + 2b + 0 + 2d,
\end{align*}
implying $a = b = c$. These equations reduce to $2 = 4a + 2d$, or $d =
1 - 2a$. Finally, $g(L_4)^2 = L_4^2 = 0$ yields $a(a+d) = 0$. Now $a =
0$ implies $d =1$ and $g = \id$, whereas $a = -d$ implies $d =
-1$. But in the latter case, $\bar{g}$ is not $\pm \id$, a
contradiction. This completes the proof.
\end{proof}

\begin{corollary}
The stabilizer of the action of $G$ on $V$ is generically trivial.
\end{corollary}

\begin{proof}
  If $g = (g_1, g_2, g_3, g_4, g_5)$ stabilizes $v \in V$, then each
  triple $(g_i, g_j, g_k)$ gives a linear automorphism of $X_{ijk}$
  for each $(i,j,k)$. Therefore, generically $g = 1$.
\end{proof}

\subsection{Reverse construction} \label{sec:penteractreverse}

We now give the proof of the reverse direction of Theorem
\ref{thm:penteractorbits}. We start from the data of a nonsingular K3
surface $X$ with non-isomorphic line bundles $L_1$, $L_2$, $L_3$ and
$L_4$ such that $L_i \cdot L_j = 2 (1-\delta_{ij}$). It follows from
Riemann-Roch that $L_i$ or $L_i^{-1}$ is effective, since $\ch^0(L_i) +
\ch^0(L_i^{-1}) \geq 2$. We assume the former without loss of
generality, noting that $L_i \dot L_j = 2$ forces a compatible choice.
Consider the multiplication map
\begin{equation}\label{construct22222}
\mu: \cH^0(X, L_1) \otimes \cH^0(X, L_2) \otimes \cH^0(X, L_3) \otimes \cH^0(X, L_4) \rightarrow \cH^0(X, L_1 \otimes L_2 \otimes L_3 \otimes L_4)
\end{equation}
on sections.  The dimension of the domain is $2^4 = 16$. Since
\[ 
(L_1 + L_2 + L_3 + L_4)^2 = 2 \cdot 2 \cdot 6 = 24, 
\]
an easy application of Riemann-Roch on $X$ then yields
\[
h^0(X,L_1 + L_2 + L_3 + L_4) = \frac{1}{2} (L_1 + L_2 + L_3 + L_4)^2 + \chi(\sO_X) = \frac{24}{2} + 2 = 14.
\]
Furthermore, we claim that the map \eqref{construct22222} is
surjective from repeated applications of the basepoint-free pencil
trick \cite[p.\ 126]{ACGH}.  We first check that a number of line
bundles have vanishing $\cH^1$ groups.

\begin{lemma}  \label{lem:pentvanish}
For generic $X$ and for distinct $i,j,k,\ell \in \{1,2,3,4\}$, the cohomology groups $\cH^1(X, L_i^{-1})$, $\cH^1(X, L_i \otimes
L_j^{-1})$, $\cH^1(X, L_i \otimes L_j \otimes L_k^{-1})$, and $\cH^1(X, L_i \otimes L_j \otimes L_k \otimes L_\ell^{-1})$ all vanish.
\end{lemma}

\begin{proof}
By symmetry, it suffices to check that $\cH^1(X, L_1^{-1})$, $\cH^1(X,
L_1 \otimes L_2^{-1})$, $\cH^1(X, L_1 \otimes L_2 \otimes L_3^{-1})$,
and $\cH^1(X, L_1 \otimes L_2 \otimes L_3 \otimes L_4^{-1})$ all
vanish. Note that Riemann-Roch and Serre duality for each of these
line bundles $L$ implies that
\begin{equation} \label{RR-penteract}
  \ch^1(L) = \ch^0(L) + \ch^0(L^{-1}) - 2 - (L \cdot L)/2.
\end{equation}

First, for $L = L_1^{-1}$, it is immediate that $\cH^0(X,L_1^{-1}) =
0$ since $L_1$ is effective and nonzero, and because $\ch^0(X,L_1) =
2$ and $L_1 \cdot L_1 = 0$, we conclude that $\ch^1(X,L_1^{-1}) =
0$. It also follows that the complete linear system described by any
of the $L_i$ is a genus one fibration on $X$.

Next, for $L = L_1 \otimes L_2^{-1}$, we have $L^2 = (L_1 - L_2)^2 =
-4$, so $\ch^0(L) - \ch^1(L) + \ch^2(L) = 0$. However, $L$ and $-L$
are not effective by genericity (since $L \cdot L_1 = -2$ and $(-L)
\cdot L_2 = -2$), so $\ch^0$ and $\ch^2$ vanish (the latter by Serre
duality). Therefore $\ch^1(L) = 0$.

Similarly, $L = L_1 \otimes L_2 \otimes L_3^{-1}$ has $L^2 = -4$, and
furthermore $(-L) \cdot L_3 = -4$, while $L \cdot L_1 = L \cdot L_2 =
0$. We can conclude again by genericity that $L$ and $-L$ are
ineffective. So $\ch^1$ vanishes for this line bundle as well.

Finally, for $L = L_1 \otimes L_2 \otimes L_3 \otimes
L_4^{-1}$, we have $L^2 = 0$, so in fact $L$ or $-L$ must be
effective. In fact, from equation~\eqref{eq:penteractrelation2}, we see
that $L = L_5^{(123)} + \sum E_i$, and it immediately follows that
$-L$ is not effective (alternatively, the latter also follows from
$(-L) \cdot L_i < 0$). From the Gram matrices for the Picard groups in
each case, one can see that $L \cdot E_i = -1$, which means that $E_i$
lie in the base locus of the linear system described by $L$. Therefore
$\ch^0(L) = \ch^0(L_5^{(123)}) = 2$. We obtain $\ch^1(L) = 0$ from
equation~\eqref{RR-penteract} above.
\end{proof}

\begin{remark}
We will treat several subvarieties of this moduli space (or rather,
finite covers of them) through the various symmetrizations of the
penteract, in subsequent sections. The comments in Remark
\ref{rmk:genericity} can be adapted to show that the genericity
assumption of the lemma does not exclude these subvarieties.
\end{remark}

The proof of surjectivity of the map \eqref{construct22222} follows
from three applications of the basepoint-free pencil trick.
Therefore, the kernel of $\mu$ in (\ref{construct22222}) has dimension
$2$, and we obtain a penteract as an element of $\cH^0(X, L_1) \otimes
\cH^0(X, L_2) \otimes \cH^0(X, L_3) \otimes \cH^0(X, L_4) \otimes
(\ker \mu)^\vee$, up to the action of $G$.

\begin{proof}[Proof of Theorem $\ref{thm:penteractorbits}$]
It only remains to show that the two constructions described above are
inverse to one another.  Given a nonsingular K3 surface $X$ with
appropriate line bundles $L_1$, $L_2$, $L_3$, and $L_4$ as in the
statement of the theorem, let $Y_{1234}$ be the natural image of $X$
in $\Proj(\cH^0(X,L_1)^\vee) \times \Proj(\cH^0(X,L_2)^\vee) \times
\Proj(\cH^0(X,L_3)^\vee) \times \Proj(\cH^0(X,L_4)^\vee)$, and let
$Y_{ijk}$ for $\{i,j,k\} \subset \{1,2,3,4\}$ be the projection onto
the $i$th, $j$th, and $k$th factors.

On the other hand, construct the penteract $A \in \cH^0(X, L_1)
\otimes \cH^0(X, L_2) \otimes \cH^0(X, L_3) \otimes \cH^0(X, L_4)
\otimes (\ker \mu)^\vee$ from $(X,L_1,L_2,L_3,L_4)$ as above, and let
$X_{1234}$ and $X_{ijk}$ be the K3 surfaces constructed from $A$ in
the usual way.  By the construction of $A$ as the kernel of $\mu$, we
have $A(v,w,x,y,\ccdot) = 0$ for any point $(v,w,x,y) \in Y_{1234}$,
so $Y_{1234} \subset X_{1234}$ and $Y_{ijk} \subset X_{ijk}$.  We
claim that $X_{1234} = Y_{1234}$ and $X_{ijk} = Y_{ijk}$ as sets and
as varieties.

At least one of the tridegree $(2,2,2)$ polynomials $f_{ijk}$ defining
$X_{ijk}$ is not identically zero; for such a variety $X_{ijk}$, we
have that $X_{ijk}$ and $Y_{ijk}$ are both given by nonzero tridegree
$(2,2,2)$ forms and thus must be the same variety, and similarly for
$X_{1234}$ and $Y_{1234}$.  Moreover, because $Y_{1234}$ is assumed to
be nonsingular, the tensor $A$ is nondegenerate.

Conversely, given a nondegenerate penteract $A \in V_1 \otimes V_2
\otimes V_3 \otimes V_4 \otimes V_5$, let $X$ be the nonsingular K3
surface $X_{1234}$ constructed from $A$ and $L_1$, $L_2$, $L_3$, $L_4$
be the line bundles on $X$.  Then the vector spaces $V_i$ and
$\cH^0(X,L_i)$ are naturally isomorphic for $1 \leq i \leq 4$, and
$V_5^\vee$ can be identified with the kernel of the multiplication map
$\mu$ from above.  With these identifications, the penteract
constructed from this geometric data is well-defined and
$G$-equivalent to the original~$A$.
\end{proof}

\subsection{Automorphisms} \label{sec:pentauts}

Given a nondegenerate penteract $A$, we may consider
the following composition of the isomorphisms from \S
\ref{sec:penteractconstruct}:
$$\alpha_{34,5}: X_{1234} \to X_{124} \to X_{1245} \to X_{125} \to
X_{1235} \to X_{123} \to X_{1234}.$$ Since each map is an isomorphism,
the entire composition is an automorphism of $X_{1234}$.  It is easy
to see that it is not the identity, however; in fact, a point
$(v_0,w_0,x_0,y_0) \in X_{1234} \subset \Proj(V_1^\vee) \times
\Proj(V_2^\vee) \times \Proj(V_3^\vee) \times \Proj(V_4^\vee)$ is sent
to $(v_0,w_0,x_1,y_1)$, where $x_0$ and $x_1$ are the two solutions
for $x$ in the equation $\det A(v_0,w_0,x,\ccdot,\ccdot) = 0$ (and
similarly for $y_0$ and $y_1$).

We may similarly define automorphisms $\alpha_{kl,m}$ of $X_{ijkl}$ for
any permutation $\{i,j,k,l,m\}$ of $\{1,2,3,4,5\}$ (where the ordering
of the indices in the subscript of $\alpha$, but not of $X$, is
relevant).  For example, the automorphisms $\alpha_{kl,m}$ and
$\alpha_{lk,m}$ of $X_{ijkl}$ are inverse to one another 
(and actually the same, as described below), but $\alpha_{km,l}$
is an automorphism of $X_{ijkm}$.

A more geometric way to describe these automorphisms is by viewing
$X_{ijkl}$ as a double cover of $\Proj(V_i^\vee) \times
\Proj(V_j^\vee)$; then $\alpha_{kl,m}$ switches the two sheets of this
double cover.  It is clear that all of these automorphisms have order
two, and thus $\alpha_{kl,m} = \alpha_{lk,m}$.

Using the relation \eqref{eq:penteractrelation} and its analogues, we
may easily compute how $\alpha_{kl,m}$ acts on the N\'eron-Severi
lattice.  For example, the automorphism $\alpha_{34,5}$ is equivalent
to the action of the matrix
\begin{equation}\label{pentaut}
\begin{pmatrix}
1 & 0 & 0 & 0 \\
0 & 1 & 0 & 0 \\
2 & 2 & -1 & 0 \\
2 & 2 & 0 & -1
\end{pmatrix}
\end{equation}
on $\NS(X)$.  Conjugating (\ref{pentaut}) by $4\times 4$ permutation
matrices yields all six automorphisms of the form $\alpha_{kl,5}$ for
$k,l\in\{1,2,3,4\}$.

For very general $X$, the group $\Gamma_{\mathrm{pent}}$ generated by
these automorphisms $\alpha_{kl,m}$ turns out to have index $60$ in
the orthogonal group $O(\NS(X),\Z)$ of $\NS(X)$,\footnote{We are
  grateful to Igor Rivin for performing this interesting
  computation.}
and therefore also finite index in $\Aut(X)$.  One
way to visualize these automorphisms is by placing each of the five
$X_{ijkl}$ on a vertex of the $5$-cell (a.k.a.~$4$-simplex) and by
viewing each edge as the isomorphism from $X_{ijkl}$ to $X_{ijkm}$
through $X_{ijk}$ (again, as defined in \S
\ref{sec:penteractconstruct}).  Then each $\alpha_{kl,m}$ is the
traversal of a triangle in the $1$-dimensional boundary of the
$5$-cell.

\begin{figure}[h]
\begin{center}
\begin{tikzpicture}
\begin{scope}[decoration={
    markings,
    mark=at position 0.5 with {\arrow[scale=2]{>}}}
    ] 
\draw [fill=black] (-2.2248595,-0.75867081) circle (0.1) node[label=below:$X_{1234}$]{};
\draw [fill=black] (1.8203396,-1.6782718) circle (0.1) node[label=below:$X_{1245}$]{};
\draw [fill=black] (-0.20225996,-0.71269076) circle (0.1) node[label={[label distance=0pt]150:$X_{1345}$}]{};
\draw [fill=black] (0.80903984,1.8392020) circle (0.1) node[label=above:$X_{2345}$]{};
\draw [fill=black] (-0.7,1.3104314) circle (0.1) node[label=above:$X_{1235}$]{};
\draw [line width=2pt, postaction={decorate},red] (-2.2248595,-0.75867081)--(1.8203396,-1.6782718);
\draw (-2.2248595,-0.75867081)--(-0.20225996,-0.71269076);
\draw (-2.2248595,-0.75867081)--(0.80903984,1.8392020);
\draw [line width=2pt, postaction={decorate},red] (-0.7,1.3104314)--(-2.2248595,-0.75867081);
\draw (1.8203396,-1.6782718)--(-0.20225996,-0.71269076);
\draw (1.8203396,-1.6782718)--(0.80903984,1.8392020);
\draw [line width=2pt, postaction={decorate},red] (1.8203396,-1.6782718)--(-0.7,1.3104314);
\draw (-0.20225996,-0.71269076)--(0.80903984,1.8392020);
\draw (-0.20225996,-0.71269076)--(-0.7,1.3104314);
\draw (0.80903984,1.8392020)--(-0.7,1.3104314);
\node at (0,-3) [red] {$\alpha_{34,5}: X_{1234} \to X_{1234}$};
\end{scope}
\end{tikzpicture}
\qquad \qquad
\begin{tikzpicture}
\begin{scope}[decoration={
    markings,
    mark=at position 0.5 with {\arrow[scale=2]{>}}}
    ] 
\draw [fill=black] (-2.2248595,-0.75867081) circle (0.1) node[label=below:$X_{1234}$]{};
\draw [fill=black] (1.8203396,-1.6782718) circle (0.1) node[label=below:$X_{1245}$]{};
\draw [fill=black] (-0.20225996,-0.71269076) circle (0.1) node[label={[label distance=0pt]150:$X_{1345}$}]{};
\draw [fill=black] (0.80903984,1.8392020) circle (0.1) node[label=above:$X_{2345}$]{};
\draw [fill=black] (-0.7,1.3104314) circle (0.1) node[label=above:$X_{1235}$]{};
\draw (1.8203396,-1.6782718)--(-2.2248595,-0.75867081);
\draw (-2.2248595,-0.75867081)--(-0.20225996,-0.71269076);
\draw [line width=2pt, postaction={decorate},red] (-2.2248595,-0.75867081)--(0.80903984,1.8392020);
\draw [line width=2pt, postaction={decorate},red] (-0.7,1.3104314)--(-2.2248595,-0.75867081);
\draw [line width=2pt, postaction={decorate},red] (-0.20225996,-0.71269076)--(1.8203396,-1.6782718);
\draw (1.8203396,-1.6782718)--(0.80903984,1.8392020);
\draw [line width=2pt, postaction={decorate},red] (1.8203396,-1.6782718)--(-0.7,1.3104314);
\draw [line width=2pt, postaction={decorate},red] (0.80903984,1.8392020)--(-0.20225996,-0.71269076);
\draw (-0.20225996,-0.71269076)--(-0.7,1.3104314);
\draw (0.80903984,1.8392020)--(-0.7,1.3104314);
\node at (0,-3) [red] {$\Phi_{51234}: X_{1234} \to X_{1234}$};
\end{scope}
\end{tikzpicture}
\end{center}
\caption{Some automorphisms of the K3 surface associated to a penteract.}
\label{sec:5cellpic}
\end{figure}
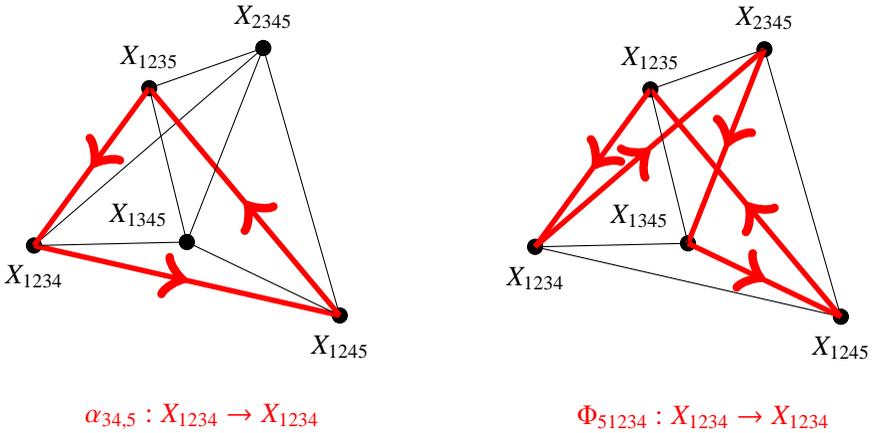

These $\alpha_{kl,m}$'s may be composed to yield nontrivial
automorphisms that are $4$- and $5$-cycles on the boundary of the
$5$-cell.  The $4$-cycles give automorphisms that preserve one of the
genus one fibrations and act by translation by a section of the
Jacobian fibration (see \cite[\S 6.2]{coregular}).  It is easy to
check that they have fixed points on reducible fibers of the
fibrations.

An example of a $5$-cycle is the automorphism
$$\Phi_{51234}: X_{1234} \to X_{2345} \to X_{1345} \to X_{1245} \to
X_{1235} \to X_{1234},$$ which is the composition $\alpha_{34,5} \circ
\alpha_{23,5} \circ \alpha_{12,5}$.  Applying $\Phi_{51234}$ to $X$
induces the action of the matrix
\begin{equation} \label{eq:pentaut}
\begin{pmatrix}
-1 & 0 & 2 & 2 \\
-2 & 1 & 2 & 4 \\
-4 & 2 & 5 & 6 \\
-6 & 2 & 8 & 11
\end{pmatrix}
\end{equation}
on $\Lambda$ in $\NS(X)$.  By symmetry, all of the $5$-cycles that
meet all $5$ vertices act in a similar way on $\NS(X)$.
As we will see in \S \ref{sec:positiveentropy}, these 5-cycle
automorphisms of $X$ turn out to be fixed-point-free in general and
have positive entropy.

As mentioned in the introduction, the elements of
$\Gamma_{\mathrm{pent}}$ are often fixed-point free and of positive entropy.  One obtains many different quadratic and quartic Salem polynomials as the
characteristic polynomials of these automorphisms.  In particular, the
N\'eron-Severi lattice $\NS(X)$ is isomorphic to $U(2) \oplus A_2(2)$
for a very general $X$ in our family. Therefore, $\Aut(X)$ has finite
index in $O(U \oplus A_2)$, which is commensurable to
$\SL_2(\Z[\omega])$ where $\omega$ denotes a third root of unity\footnote{We are grateful to Curt McMullen for pointing out this commensurability.}.
One way to see this commensurability is as follows: consider the Hermitian form over
$\Z[\omega]$ with matrix given by
\begin{equation} \label{eq:hermform}
M = \begin{pmatrix}
x & z - \omega w \\
z - \omega^2 w & y 
\end{pmatrix}.
\end{equation}
The discriminant of this form (i.e., the determinant of the matrix) is
$xy - (z^2+ zw + w^2)$, which is half the quadratic form corresponding
to the lattice $U \oplus A_2$. Therefore, it is enough to show that a
finite index subgroup of the group $\SL_2(\Z[\omega])$ acts as a group
of automorphisms of the Hermitian form \eqref{eq:hermform}. This is readily accomplished
by considering the action $g \cdot M = g M g^\dagger$, where $g \in
\SL_2(\Z[\omega])$ and $g^\dagger$ is the conjugate transpose of
$g$. We omit checking the standard details, referring the interested
reader to, e.g., \cite[\S 13.9, pg.~317]{cassels}.%
\footnote{We may also see this commensurability geometrically by
  comparing the automorphism group of the abelian surface $E \times E$
  (where $E = \mathbb{C}/\Z[\omega]$ is the elliptic curve of
  $j$-invariant $0$) with the automorphisms of its transcendental
  lattice $U(-1) \oplus A_2(-1)$.  See also Aurel Page's answer in
  \cite{mathoverflow-aurel}, which outlines a proof that $O(U\oplus
  A_2)$ is in fact isomorphic to $\PGL_2(\Z[\omega])$.}%

The Salem polynomial corresponding to the action of $g \in \SL_2(\Z[\omega])$ on the Hermitian form $M$ is
$p_g(T) = T^4- ee'T^3 + (e^2 + e'^2 - 2) T^2 - ee'T + 1,$
where $e = \mathrm{Trace}(g)$ and $e'$ is the conjugate of $e$. The splitting field
of this quartic polynomial has Galois group isomorphic to the dihedral group $D_4$ of order $8$, and it is also the splitting field of
\begin{align*}
q_g(T) &= \Norm_{\Q(\omega)[T] / \Q[T]} \det (g-T \cdot \mathrm{Id}) \\
     &= \Norm_{\Q(\omega)[T] / \Q[T]} (T^2 - eT + 1) \\
     & = T^4 - (e + e')T^3 + (2 + ee')T^2 - (e + e' )T + 1
\end{align*}
over $\Q$. In fact, the fields $K_p$ and $K_q$ obtained by adjoining a root of $p_g(T)$ and $q_g(T)$, respectively, to $\Q$ are dual $D_4$-quartic fields, i.e., they are the fixed fields of the two subgroups of order $2$ (up to conjugacy) in $D_4$ which are interchanged by the outer involution of $D_4$. In particular, the quadratic resolvent field of $K_p$ is $\Q(\omega)$.

It is then easy to deduce that the quadratic Salem polynomials of the automorphism group of $X$ generate all real quadratic fields, while the quartic Salem polynomials generate all $D_4$-quartic fields whose quadratic resolvent field is $\Q(\omega)$. Below, we indicate how to explicitly find automorphisms to prove the stronger statement in the introduction about specific quadratic Salem polynomials of the form $x^2-(4n^2 \pm 2)x+1$ and $x^2-(12n^2 \pm 2)x+1$. A similar analysis using the unit group of the quartic field $\Q(\omega)[T]/(T^2 - eT + 1)$ gives the second half of our assertion.

Let
\[
\gamma_1 = \alpha_{34,5} \circ \alpha_{24,5} = 
\begin{pmatrix}
1 & 0 & 0 & 0 \\
2 & -1 & 2 & 0 \\
6 & -2 & 3 & 0 \\
4 & -2 & 2 & 1 
\end{pmatrix} \quad \textrm{and} \quad
\gamma_2 = \alpha_{13,5} \circ \alpha_{12,5} = 
\begin{pmatrix}
1 & -2 & 2 & 4 \\
0 & -1 & 2 & 2 \\
0 & -2 & 3 & 6 \\
0 & 0 & 0 & 1
\end{pmatrix}.
\]
Then it is easily verified that for $k \in \Z$, the automorphism
$\gamma_1^k \gamma_2$ produces the Salem polynomial
\begin{equation} \label{eq:manySalempolys}
x^2 - \big( 4(2k + 1)^2 - 2 \big) x + 1.
\end{equation}
A positive proportion of the polynomials $x^2 - (4n^2 + 2) x + 1$ and $x^2 - (12n^2 \pm 2) x + 1$
may also be obtained as Salem polynomials of penteract automorphisms in a similar way:
\begin{enumerate}
\item For $\gamma_1 = \alpha_{34,5} \circ \alpha_{24,5}$ and $\gamma_2 = \alpha_{14,5}$, the Salem polynomial corresponding to the automorphism $\gamma_1^k \gamma_2$ is $x^2 - (4(2k)^2 + 2)x + 1$.
\item For $\gamma_1 = \alpha_{34,5} \circ \alpha_{24,5} \circ \alpha_{34,5}$ and $\gamma_2 = \alpha_{23,5} \circ \alpha_{12,5}$, the Salem polynomial corresponding to the automorphism $\gamma_1 \gamma_2^k$ is $x^2 - (12(2k)^2 - 2)x + 1$.
\item For $\gamma_1 = \alpha_{12,5} \circ \alpha_{34,5}$ and $\gamma_2 = \alpha_{23,5} \circ \alpha_{34,5} \circ \alpha_{24,5} \circ \alpha_{34,5}$, the Salem polynomial corresponding to the automorphism $\gamma_1 \gamma_2^k$ is $x^2 - (12(2k+1)^2 + 2)x + 1$.
\end{enumerate}

Obtaining the Salem polynomials of the form \eqref{eq:manySalempolys}
is sufficient to deduce that all real quadratic fields occur as
the splitting fields of Salem polynomials of automorphisms of general 
K3 surfaces in our penteract family.  
To see this,
for each discriminant $D$ of a real quadratic field, we wish to show 
the existence of a pair $(m,n)$ of positive integers 
with $m=2k+1$ such that $Dn^2=(4m^2-2)^2-4=16m^2(m^2-1)$, or 
equivalently, the existence of a pair $(m,n')$ of 
positive integers with $m$ odd such that $m^2-Dn'^2=1$ (for we may 
then set $n$ to be $4mn'$).  The latter Brahmagupta-Pell equation is 
well-known to have infinitely many positive integer solutions $(m,n')$ for every 
discriminant $D$, even with  the restriction that $m$ is odd, proving the claim.                                                           


\section{Doubly symmetric penteracts: \texorpdfstring{$2 \otimes 2 \otimes 2 \otimes \Sym^2(2)$}{2 (x) 2 (x) 2 (x) Sym2(2)}} \label{sec:2sympent}

We now consider doubly symmetric penteracts, namely elements of $V =
V_1\otimes V_2\otimes V_3\otimes \Sym^2 V_4$ for 2-dimensional
$\fd$-vector spaces $V_1$, $V_2$, $V_3$, and $V_4$, with an action of
the group $\GL(V_1)\times \GL(V_2)\times \GL(V_3)\times \GL(V_4)$.
Since the space of doubly symmetric penteracts maps naturally into the
space of all penteracts from \S \ref{sec:penteracts} after the
identification of $V_4$ and $V_5$, one may understand the orbits of
doubly symmetric penteracts by using constructions from
Theorem~\ref{thm:penteractorbits}.  We find that these orbits
correspond to certain K3 surfaces of Picard rank at least 9 over
$\fdbar$:

\begin{theorem} \label{thm:2sympenteractorbits} Let $V = V_1\otimes
  V_2\otimes V_3\otimes \Sym^2 V_4$ for 2-dimensional $\fd$-vector
  spaces $V_1$, $V_2$, $V_3$, and $V_4$.  Let $G' = \GL(V_1)\times
  \GL(V_2)\times \GL(V_3)\times \GL(V_4)$ and let $G$ be the quotient
  of $G'$ by the kernel of the natural multiplication map $\Gm \times
  \Gm \times \Gm \times \Gm \to \Gm$ sending $(\gamma_1, \gamma_2,
  \gamma_3, \gamma_4) \mapsto \gamma_1 \gamma_2 \gamma_3 \gamma_4^2$.
  Let $\Lambda$ be the lattice whose Gram matrix is
\begin{small}
\begin{equation} \label{eq:NS-2sym}
\begin{pmatrix}
\,0\, & \,2\, & \,2\, & \,2\, & 0 & 0 & 0 & 0 & 0 \\
2 & 0 & 2 & 2 & 0 & 0 & 0 & 0 & 0 \\
2 & 2 & 0 & 2 & 0 & 0 & 0 & 0 & 0 \\
2 & 2 & 2 & 0 & 1 & 1 & 1 & 1 & 1 \\
0 & 0 & 0 & 1 & -2 & 0 & 0 & 0 & 0 \\
0 & 0 & 0 & 1 & 0 & -2 & 0 & 0 & 0 \\
0 & 0 & 0 & 1 & 0 & 0 & -2 & 0 & 0 \\
0 & 0 & 0 & 1 & 0 & 0 & 0 & -2 & 0 \\
0 & 0 & 0 & 1 & 0 & 0 & 0 & 0 & -2
\end{pmatrix},
\end{equation}
\end{small}%
and let $S = \{e_1, e_2, e_3, e_4 \}$.  Then the $G(\fd)$-orbits of an
open subset of $V(\fd)$ are in bijection with the $\fd$-points of an
open subvariety of the moduli space $\sM_{\Lambda,S}$ of K3
surfaces $X$ lattice-polarized by $(\Lambda, S)$.
\end{theorem}

Many of the constructions in \S \ref{sec:penteracts} apply, but there
are some important differences. For a general penteract, the ten
different K3 surfaces in $\Proj^1\times \Proj^1\times\Proj^1$ are
nonsingular.  For a doubly symmetric penteract $A$, however, the locus
of $(r,s,t)$ where $A(r,s,t,\ccdot,\ccdot)$ is identically zero is
zero-dimensional and of degree 6, as it is given as the intersection
of {\em three} $(1,1,1)$ forms on $\Proj(V_1^\vee) \times
\Proj(V_2^\vee) \times \Proj(V_3^\vee)$.  Therefore, the K3 surface
$X_{123}$ defined by $\det\,A(r,s,t,\ccdot,\ccdot)=0$ will have, in
general, $6$ isolated (rank) singularities over $\fdbar$.  These
singularities of $X_{123}$ are blown up by the map $X_{1234} \to
X_{123}$, and while these singularities may not be individually
defined over $\fd$, the entire degree $6$ subscheme is defined over
$\fd$.  It is easy to check that all of the other $X_{ijk}$ coming
from $A$ are generically nonsingular.

\subsection{N\'eron-Severi lattice}

To compute the N\'eron-Severi group of the nonsingular K3 surface
$X_{1234}$ here, we observe that there still are line bundles $L_i$ on
$X_{1234}$ coming from pulling back $\sO_{\Proj(V_i^\vee)}(1)$ to
$X_{1234}$ for $1 \leq i \leq 4$.  In addition, over $\fdbar$, there
are six exceptional fibers $E_i$ for $1 \leq i \leq 6$, coming from
the blowups of the $6$ singularities in $X_{123}$; the sum of these
$E_i$ is a divisor defined over $\fd$.  It is easy to compute all of
the intersection numbers: the only nonzero ones are $L_i \cdot L_j = 2
(1 - \delta_{ij})$, $L_4 \cdot E_i = 1$ for $1 \leq i \leq 6$, and
$E_i^2 = -2$.

Recall from Lemma \ref{lem:pentrelation} that there is a relation
among the line bundles; here, it is slightly simplified because of the
symmetry (that is, $L_4$ and $L_5^{123}$ are isomorphic):
\begin{equation} \label{eq:2sympentrelation}
L_1 + L_2 + L_3 = 2 L_4 + \sum_{i=1}^6 E_i.
\end{equation}
We thus expect the N\'eron-Severi group to generically have rank $9$,
and in fact, the intersection matrix of all of these divisor classes
may be reduced to the matrix \eqref{eq:NS-2sym}, which is the
intersection matrix for $L_1, \dots, L_4, E_1,\dots, E_5$. Note that
the lattice they span has discriminant $2^{8}$.

\begin{proposition} 
  For a very general $X$ in this family of K3 surfaces, $\overline{\NS}(X)$ is
  spanned over $\Z$ by $L_1$, $L_2$, $L_4$, and the $E_i$, $i = 1,
  \dots, 6$.
\end{proposition}

\begin{proof}
  Since the moduli space here has dimension $8 \cdot 3 - 13 = 11$, the Picard
  number of a very general $X$ is at most $9$.
  It is enough to check that the lattice $L$ spanned by these divisors
  is $2$-saturated. Computation of the discriminant group shows that
  any element of the dual lattice may be written as
  \[
  D = \frac{1}{2}(c_1 L_1 + c_2 L_2 + c_4 L_4) + \frac{1}{4} \sum d_i
  E_i,
  \]
  with $c_i, d_i$ integers. First, by Lemma \ref{lem:nikulin}, no
  divisor of the form $(\sum e_i E_i)/2$, with $e_i$ integers, can be
  in $\overline{\NS}(X)$, unless all the $e_i$ are even. It follows
  that in the expression for $D$, all the $d_i$ must be even, as $2D -
  c_1 L_1 - c_2 L_2 - c_4 L_4$ would otherwise be a counterexample to
  the above observation.  So we may assume that $D$ has the form
  \begin{equation} \label{eq:2sympentdivisorcomp}
  D = \frac{1}{2}\bigl( c_1 L_1 + c_2 L_2 + c_4 L_4 + \sum_{i=1}^6 e_i
    E_i \bigr).
  \end{equation}
  Intersection with $E_i$ shows that $c_4$ is an even integer, so we
  may assume it is zero. If $c_1$ and $c_2$ are even, then we get a
  contradiction to Lemma \ref{lem:nikulin} as above. At least one of
  $c_1$ and $c_2$ is odd, and if both are not odd, we may construct a
  divisor $D' \in \overline{\NS}(X)$ by reversing the roles of $L_1$
  and $L_2$, by symmetry. Then $D + D'$ has the same shape as
  \eqref{eq:2sympentdivisorcomp}, with both coefficients $c_1$ and
  $c_2$ odd. So we may assume $c_1 = c_2 = 1$ and all $e_i \in
  \{0,1\}$ by subtracting an element of $L$.
  
  The self-intersection of $D$ is $1 - \sum e_i^2/2$ and is then even,
  so $\sum e_i^2 = 2$ or $6$. If exactly two of the $e_i$ are $1$, say
  $E_1$ and $E_2$, then by symmetry each of the divisors $(L_1 + L_2 +
  E_i + E_j)/2$ is in $\overline{\NS}(X)$. Subtracting two of these,
  we see that $(E_1 + E_2 + E_3 + E_4)/2$ is in $\overline{\NS}(X)$,
  which contradicts Lemma \ref{lem:nikulin}.  Finally, if all six
  $e_i$ are $1$, then another application of symmetry shows that $D' =
  (L_1 + L_3 + \sum E_i) \in \overline{\NS}(X)$. Therefore $D-D' =
  (L_1 - L_3)/2 \in \overline{\NS}(X)$, which is impossible since it
  has odd self-intersection.
\end{proof}

\begin{corollary} The divisors $L_1, \dots, L_4, E_1, \dots, E_5$ form
  a basis for $\overline{\NS}(X)$, for $X$ very general.
\end{corollary}

\subsection{Moduli problem}

We now complete the proof of Theorem \ref{thm:2sympenteractorbits}.

\begin{proof}[Proof of Theorem $\ref{thm:2sympenteractorbits}$]

  The constructions in both directions almost exactly follow those for
  Theorem~\ref{thm:penteractorbits}.  Given a doubly symmetric
  penteract, we obtain the K3 surfaces $X_{ijk}$ and $X_{ijkl}$ with
  the divisor classes (and intersection matrix) as described above.

  On the other hand, given such a K3 surface $X$ lattice-polarized by
  $(\Lambda,S)$, we must show that the penteract $A$ constructed by
  the reverse map of \S \ref{sec:penteractreverse} is doubly symmetric
  under an identification of two of the vector spaces.  Let $L_1$,
  $L_2$, $L_3$, and $L_4$ be line bundles corresponding to the four
  elements of $S$ (in order).  These are the line bundles used to
  construct the penteract $A \in \cH^0(X,L_1) \otimes \cH^0(X,L_2)
  \otimes \cH^0(X,L_3) \otimes \cH^0(X,L_4) \otimes (\ker \mu)^\vee,$
  where $\mu$ is the multiplication map \eqref{construct22222}.

  Note that $A$ in turn gives rise to isomorphic K3 surfaces and line
  bundles, as well as a fifth line bundle, say $L_5^{(123)}$, via the
  map from $X_{1234} \to X_{123} \dasharrow X_{1235} \to \Proj(\ker
  \mu)$.  In addition, from the intersection matrix
  \eqref{eq:NS-2sym}, we see that there are six singularities on
  $X_{123}$ which are blown up in $X_{1234}$ (whose exceptional fibers
  correspond to the last six rows/columns of the intersection
  matrix). Thus, using the relation \eqref{eq:penteractrelation2} and
  the intersection matrix \eqref{eq:NS-2sym}, we find that $L_4$ and
  $L_5^{(123)}$ are in fact isomorphic.  Therefore, we may identify
  the vector spaces $V_4 := \cH^0(X,L_4)$ and $V_5 := (\ker
  \mu)^\vee$.

  With this identification, the maps from $X_{123}$ to
  $\Proj(V_4^\vee)$ and $\Proj(V_5^\vee)$ are identical and given in
  the usual way by taking the appropriate kernels of
  $A(v,w,x,\ccdot,\ccdot) \in V_4 \otimes V_5$ for $(v,w,x) \in
  X_{123}$.  The remaining key idea is very simple to check (e.g.,
  explicitly) in this case: for a rank one element $\zeta \in V
  \otimes V$ for a $2$-dimensional $\fd$-vector space $V$, if
  $\zeta(v, \ccdot) = 0$ and $\zeta(\ccdot, v) = 0$, then $\zeta$ is
  in fact in the symmetric subspace $\Sym^2 V$ of $V \otimes V$.
  Since $X_{123}$ spans the ambient space $\Proj(\cH^0(X,L_1)^\vee)
  \times \Proj(\cH^0(X,L_2)^\vee) \times \Proj(\cH^0(X,L_3)^\vee)$,
  the penteract $A$ is in fact symmetric, i.e., an element of
  $\cH^0(X,L_1) \otimes \cH^0(X,L_2) \otimes \cH^0(X,L_3) \otimes
  \Sym^2 \cH^0(X,L_4)$, as desired.
\end{proof}

\subsection{Automorphisms} \label{sec:2sympentauts}

The automorphisms $\alpha_{ij,k}$ defined in \S \ref{sec:pentauts} are
again automorphisms of the K3 surfaces obtained from doubly symmetric
penteracts.  Of course, because of the symmetry in this case, some of
these are the same automorphism, e.g., $\alpha_{14,2} =
\alpha_{15,2}$.

Moreover, the action of the $5$-cycles on the N\'eron-Severi lattice
is different, since we now have to take the exceptional divisor
classes into account.  To compute this action for the generic K3
surface in this family, we employ the same methods as for the doubly
symmetric Rubik's revenge, namely, repeated applications of the
relation \eqref{eq:penteractrelation2} and computations of
intersection numbers.  For example, the action of
$$\Phi_{54321}: X_{1234} \to X_{1235} \to X_{1245} \to X_{1345} \to X_{2345} \to X_{1234}$$
on $\overline{\NS}(X)$ here is given by the matrix

\begin{small}
\begin{equation} \label{eq:2sympentaut}
\begin{pmatrix}
5 & 2 & -4 & 6 & 0 & 0 & 0 & 0 & 0 \\
2 & 1 & -2 & 4 & 0 & 0 & 0 & 0 & 0 \\
2 & 0 & -1 & 2 & 0 & 0 & 0 & 0 & 0 \\
1 & 1 & -1 & 1 & 0 & 0 & 0 & 0 & 0 \\
1 & 0 & -1 & 2 & 1 & 0 & 0 & 0 & 0 \\
1 & 0 & -1 & 2 & 0 & 1 & 0 & 0 & 0 \\
1 & 0 & -1 & 2 & 0 & 0 & 1 & 0 & 0 \\
1 & 0 & -1 & 2 & 0 & 0 & 0 & 1 & 0 \\
1 & 0 & -1 & 2 & 0 & 0 & 0 & 0 & 1
\end{pmatrix}.
\end{equation}
\end{small}%
We will look at this automorphism again in \S \ref{sec:positiveentropy}.

By symmetry, all of the $5$-cycles that meet all five models
$X_{ijkl}$ are either analogous to $\Phi_{54321}$ above or to
$\Phi_{53421}$.  In the latter case, the induced action on the line
bundles $L_i$ is similar to that of $\Phi$ in the penteract case (from
which the action on the divisors $E_j$ may be immediately
deduced). These two types of automorphisms will be shown in \S
\ref{sec:positiveentropy} to be fixed-point-free in general and of
positive entropy.


\section{Triply symmetric penteracts: \texorpdfstring{$2 \otimes 2 \otimes \Sym^3(2)$}{2 (x) 2 (x) Sym3(2)}} \label{sec:3sympent}

Suppose we now have a penteract that is symmetric in the last {\it
  three} coordinates.
We prove that the general orbits of such tensors correspond to certain
K3 surfaces with Picard rank at least $14$ over $\fdbar$:

\begin{theorem} \label{thm:3sympenteractorbits} Let $V = V_1\otimes
  V_2\otimes\Sym^3 V_3$ for $2$-dimensional $\fd$-vector spaces
  $V_1,V_2,V_3$. Let $G' = \GL(V_1) \times \GL(V_2) \times \GL(V_3)$
  act on $V$, and let $G$ be the quotient of $G'$ by the kernel of the
  map $\Gm \times \Gm \times \Gm \to \Gm$ sending $(\gamma_1,
  \gamma_2, \gamma_3) \mapsto \gamma_1 \gamma_2 \gamma_3^3$.  Let
  $\Lambda$ be the lattice whose Gram matrix is

{\footnotesize \begin{equation} \label{eq:NS-3sym}
\begin{pmatrix}
  \,0\,  & \, 2\,  & \,2\,  & \,2\,  & 0  & 0  & 0  & 0  & 0  & 0  & 0  & 0  & 0  &  0  \\
  2  &  0  & 2  & 2  & 0  & 0  & 0  & 0  & 0  & 0  & 0  & 0  & 0  &  0  \\
  2  &  2  & 0  & 2  & 0  & 1  & 0  & 1  & 0  & 1  & 0  & 1  & 0  &  1  \\
  2  &  2  & 2  & 0  & 1  & 0  & 1  & 0  & 1  & 0  & 1  & 0  & 1  &  0  \\
  0  &  0  & 0  & 1  & -2 & 1  & 0  & 0  & 0  & 0  & 0  & 0  & 0  &  0  \\
  0  &  0  & 1  & 0  & 1  & -2 & 0  & 0  & 0  & 0  & 0  & 0  & 0  &  0  \\
  0  &  0  & 0  & 1  & 0  & 0  & -2 & 1  & 0  & 0  & 0  & 0  & 0  &  0  \\
  0  &  0  & 1  & 0  & 0  & 0  & 1  & -2 & 0  & 0  & 0  & 0  & 0  &  0  \\
  0  &  0  & 0  & 1  & 0  & 0  & 0  & 0  & -2 & 1  & 0  & 0  & 0  &  0  \\
  0  &  0  & 1  & 0  & 0  & 0  & 0  & 0  & 1  & -2 & 0  & 0  & 0  &  0  \\
  0  &  0  & 0  & 1  & 0  & 0  & 0  & 0  & 0  & 0  & -2 & 1  & 0  &  0 \\
  0  &  0  & 1  & 0  & 0  & 0  & 0  & 0  & 0  & 0  &  1 & -2 & 0  &  0 \\
  0  &  0  & 0  & 1  & 0  & 0  & 0  & 0  & 0  & 0  &  0 & 0  & -2 &  1 \\
  0 & 0 & 1 & 0 & 0 & 0 & 0 & 0 & 0 & 0 & 0 & 0 & 1 & -2
\end{pmatrix},
\end{equation}
}%
and let $S = \{ e_1, e_2, e_3, e_4 \}.$ Then the $G(\fd)$-orbits of an
open subset of $V(\fd)$ are in bijection with the $\fd$-points of an
open subvariety of the moduli space $\sM_{\Lambda,S}$ of K3
surfaces $X$ lattice-polarized by $(\Lambda, S)$.
\end{theorem}

\subsection{N\'eron-Severi lattice}

Note that a triply symmetric penteract is also doubly symmetric in any
two of the last three coordinates.  This implies, from \S
\ref{sec:2sympent}, that the K3 surface $X_{123}$ ($= X_{124} =
X_{125}$) has at least six rank singularities (over $\fdbar$), and a
numerical example shows us that generically there are no other
singularities.  Meanwhile, the other surfaces $X_{134}$ and $X_{234}$
are generically nonsingular, as are all the $X_{ijkl}$.

The maps of the type $X_{1234} \to X_{123}$ blow up the six
singular points on $X_{123}$, and thus $X_{1234}$ contains six lines;
call the associated divisor classes $P_i$ for $1 \leq i \leq 6$.
There is also a map $X_{1234} \to X_{124}$, defined by
the identical equations after switching the $3$rd and $4$th
coordinates, so there are at least twelve lines in $X_{1234}$; call
the six lines coming from this map $Q_i$ for $1 \leq i \leq 6$.  These
twelve lines occur in pairs, say $(P_i, Q_i)$ for $1 \leq i \leq 6$,
which are flipped by the birational involution $X_{123} \dasharrow
X_{1234} \to X_{124} = X_{123}$.

Recall that there are line bundles $L_1$, $L_2$, $L_3$, and $L_4$ on
$X_{1234}$ coming from the pullback of $\sO_{\Proj(V_i^\vee)}(1)$, and
by Lemma \ref{lem:pentrelation}, we have the relations
\begin{align*}
 L_1 + L_2 +  L_3 &= 2 L_4 + \sum_{i=1}^6 P_i \qquad \textrm{and} \\
 L_1 + L_2 + L_4 &= 2 L_3 + \sum_{i=1}^6 Q_i.
\end{align*}

Each of the six pairs of lines $(P_i, Q_i)$ determines a single point
of intersection.  Explicitly, if the associated singular point on
$X_{123}$ is $(v,w,x) \in \Proj(V_1^\vee) \times \Proj(V_2^\vee)
\times \Proj(V_3^\vee)$, then the intersection point is $(v,w,x,x) \in
\Proj(V_1^\vee) \times \Proj(V_2^\vee) \times \Proj(V_3^\vee) \times
\Proj(V_3^\vee)$.  These are the only six intersection points among
all of the $P_i$ and $Q_j$, as $(\sum P_i) \cdot (\sum Q_j) = 6$.

Another way to see that each of these pairs of lines intersect once
(and do not intersect any other lines) is to view $X_{123}$ as a
double cover of $\Proj(V_1^\vee) \times \Proj(V_2^\vee)$, branched
along a bidegree $(4,4)$ curve.  One computes that there are exactly
six $A_2$ singularities on that curve.

The map from $X_{1234}$ to $\Proj(V_1^\vee)$ is a genus one fibration
whose discriminant as a binary form on $V_1$ has degree 24, and it
factors as the cube of a degree six form times an irreducible degree
six form. Therefore, the genus one fibration has six reducible fibers
of type $\mathrm I_3$ (in the sense of Kodaira \cite{kodaira1,kodaira23}).  
These reducible fibers each consist of three
lines in a ``triangle''; a distinguished pair of these lines in each
triangle together give us the six pairs of lines described previously.

As a consequence, the N\'eron-Severi lattice (over $\fdbar$) has rank
at least $2\cdot 6+2 = 14$.  It is straightforward to compute the
intersection numbers of all of the known divisor classes (the four
line bundles from pulling back $\sO_{\Proj(V_i^\vee)}(1)$ for $1 \leq
i \leq 4$ and the two distinguished lines in each of the six
triangles). The only nonzero intersection numbers are
\begin{align*}
  L_i \cdot L_j &= 2 \textrm{ for } i \neq j, & L_3 \cdot Q_i &= 1, & L_4 \cdot P_i &= 1, \\
  P_i^2 &= Q_i^2 = -2, & P_i \cdot Q_i &= 1.
\end{align*}
Taking the basis $\{L_1, \dots, L_4, P_1, Q_1, \dots, P_5, Q_5\}$, one
obtains the lattice with Gram matrix \eqref{eq:NS-3sym}.  This lattice
has discriminant $-324$.

\begin{proposition} 
  For a very general $X$ in this family of K3 surfaces, $\overline{\NS}(X)$
  is spanned over $\Z$ by $L_1$, $L_2$, $L_4$, and the exceptional
  classes $P_i, Q_i$, $i = 1, \dots, 6$.
\end{proposition}

\begin{proof}
  A dimension count shows that the moduli space in this case has dimension
  $4 \cdot 4 - 10 = 6$, so the Picard number of a very general $X$ is at most $14$.
  Let $\Lambda$ be the lattice spanned by the above classes. First,
  note that $\Lambda$ is already spanned by $L_1, \dots, L_4$ and the
  ten classes $P_i$, $Q_i$, for $i = 1, \dots, 5$, since we may solve for
  $P_6$ and $Q_6$ from the above relations. Since these remaining
  fourteen classes are linearly independent (they have a nonsingular
  intersection matrix), they form a basis for $\Lambda$.

  Let $Z_i = P_i - Q_i$. Computing the inverse of the Gram matrix
  shows that any element of the dual lattice has the form
  \[
  D = \frac{1}{2}(c_1 L_1 + c_2 L_2) + \frac{1}{3}\sum_{i = 1}^5 d_i
  Z_i
  \]
  where $c_i$ and $d_i$ are integers. Suppose $D \in
  \overline{\NS}(X)$. Then $3D \in \overline{\NS}(X)$, from which it
  follows that $D'= (c_1 L_1 + c_2L_2)/2 \in \overline{\NS}(X)$. We
  claim both $c_1$ and $c_2$ are even. If $c_1$ and $c_2$ are odd,
  then $D'^2$ is odd, a contradiction. So at least one of $c_1$ and
  $c_2$ is even. If one is odd and one is even, we can find another
  divisor (by symmetry) with the parities reversed, and adding them
  will give us an element with both coefficients odd, a
  contradiction. Therefore, we may assume
  \[
  D = \frac{1}{3}\sum_{i = 1}^5 d_i Z_i.
  \]
  We may assume each $d_i \in \{-1,0,1\}$. But note that $Z_i^2 = -6$
  and $Z_i \cdot Z_j = 0$ for $i \neq j$. Hence $D^2 = -2(\sum
  d_i^2)/3$, and since this must be an (even) integer, we see that
  exactly three of the $d_i$ must be $\pm 1$. Suppose without loss of
  generality that $(Z_1 + Z_2 + Z_3)/3 \in \overline{\NS}(X)$. Then by
  symmetry any $(Z_i + Z_j + Z_k)/3 \in \overline{\NS}(X)$. Therefore,
  \[
  \frac{1}{3}(Z_1 + Z_2 + Z_3) - \frac{1}{3}(Z_1 + Z_2 + Z_4) =
  \frac{1}{3}(Z_3 - Z_4) \in \overline{\NS}(X),
  \]
  which is impossible, since exactly three of the $d_i$ are $\pm 1$.
\end{proof}

\subsection{Moduli problem}

We now complete the proof of Theorem \ref{thm:3sympenteractorbits}.

\begin{proof}[Proof of Theorem $\ref{thm:3sympenteractorbits}$]
  The above discussion describes how to construct a K3 surface
  lattice-polarized by $(\Lambda, S)$ from a triply symmetric
  penteract.  It remains to show that from such data, the penteract
  $A$ constructed as in \S \ref{sec:penteractreverse} is in fact
  triply symmetric.  That is, starting from $X$ lattice-polarized by
  $(\Lambda, S)$, let $L_1$, $L_2$, $L_3$, and $L_4$ be the line
  bundles corresponding to the elements of $S$.  Then we obtain a
  penteract $A \in \cH^0(X,L_1) \otimes \cH^0(X,L_2) \otimes
  \cH^0(X,L_3) \otimes \cH^0(X,L_4) \otimes (\ker \mu)^\vee$, where
  $\mu$ is the usual multiplication map on sections.

  The rest of the proof builds on that of Theorem
  \ref{thm:2sympenteractorbits}.  In particular, that proof
  immediately shows that $A$ must be doubly symmetric, i.e., symmetric
  in the fourth and fifth tensor factors.  (Note that this argument
  relies on \eqref{eq:penteractrelation2} with the exceptional fibers
  $P_i$.)  By switching the roles of the indices $3$ and $4$ in that
  argument, and using the $Q_i$ for \eqref{eq:penteractrelation2}, we
  also see that $A$ is symmetric in the third and fifth factors.  In
  other words, there are simultaneous identifications of the vector
  spaces $\cH^0(X,L_3)$, $\cH^0(X,L_4)$, and $(\ker \mu)^\vee$ such
  that $A$ is triply symmetric in these three factors, i.e., under
  these identifications, we may think of $A$ as an element of
  $\cH^0(X,L_1) \otimes \cH^0(X,L_2) \otimes \Sym^3 \cH^0(X,L_3)$.
\end{proof}

\subsection{Automorphisms} \label{sec:3sympentauts}

As in the previous penteract cases, we may again consider many
automorphisms of the form $\alpha_{ij,k}$.  All the $5$-cycles
meeting all five $X_{ijkl}$ are equivalent (up to reordering) to one of the following two:
\begin{align*}
\Phi_{54123} &: X_{1234} \to X_{1235} \to X_{2345} \to X_{1345} \to X_{1245} \to X_{1234}\\
\Phi_{54132} &: X_{1234} \to X_{1235} \to X_{2345} \to X_{1245} \to X_{1345} \to X_{1234}.
\end{align*}
Using the same techniques as in previous sections, namely, applying
Lemma \ref{lem:pentrelation} and computing intersection numbers, we
obtain the action of $\Phi_{54123}$ on the
N\'eron-Severi lattice of the K3 surface $X_{\fdbar}$ arising from a
general triply symmetric penteract as the matrix

\begin{footnotesize}
\begin{equation} \label{eq:3sympentaut}
\begin{pmatrix}
	-1 & 0 & 2 & 2 & 0 & 0 & 0 & 0 & 0 & 0 & 0 & 0 & 0 & 0\\
	-2 & 1 & 2 & 4 & 0 & 0 & 0 & 0 & 0 & 0 & 0 & 0 & 0 & 0\\
	0 & 0 & 0 & 1 & 0 & 0 & 0 & 0 & 0 & 0 & 0 & 0 & 0 & 0\\
	-1 & 1 & 1 & 1 & 0 & 0 & 0 & 0 & 0 & 0 & 0 & 0 & 0 & 0\\
	0 & 0 & 0 & 1 & 0 & -1 & 0 & 0 & 0 & 0 & 0 & 0 & 0 & 0\\
	-1 & 0 & 1 & 1 & 1 & 1 & 0 & 0 & 0 & 0 & 0 & 0 & 0 & 0\\
	0 & 0 & 0 & 1 & 0 & 0 & 0 & -1 & 0 & 0 & 0 & 0 & 0 & 0\\
	-1 & 0 & 1 & 1 & 0 & 0 & 1 & 1 & 0 & 0 & 0 & 0 & 0 & 0\\
	0 & 0 & 0 & 1 & 0 & 0 & 0 & 0 & 0 & -1 & 0 & 0 & 0 & 0\\
	-1 & 0 & 1 & 1 & 0 & 0 & 0 & 0 & 1 & 1 & 0 & 0 & 0 & 0\\
	0 & 0 & 0 & 1 & 0 & 0 & 0 & 0 & 0 & 0 & 0 & -1 & 0 & 0\\
	-1 & 0 & 1 & 1 & 0 & 0 & 0 & 0 & 0 & 0 & 1 & 1 & 0 & 0\\
	0 & 0 & 0 & 1 & 0 & 0 & 0 & 0 & 0 & 0 & 0 & 0 & 0 & -1\\
	-1 & 0 & 1 & 1 & 0 & 0 & 0 & 0 & 0 & 0 & 0 & 0 & 1 & 1
\end{pmatrix}.
\end{equation}
\end{footnotesize}%
The induced action of the other automorphism $\Phi_{54132}$ on the
line bundles $L_i$ and the exceptional divisors $P_j$ is the same (up
to reordering) as the action of $\Phi_{54321}$ on $L_i$ and $E_j$ in
the doubly symmetric penteract case, and the induced action on the
$Q_j$ may also be immediately computed.


\section{Doubly-doubly symmetric penteracts: \texorpdfstring{$2 \otimes \Sym^2(2) \otimes \Sym^2(2)$}{2 (x) Sym2(2) (x) Sym2(2)}} \label{sec:22sympent}

Suppose we have a penteract $A$ that is symmetric in the second and
third coordinates, and also in the last two coordinates.  Then we may
use the theorems from \S\S \ref{sec:penteracts} and \ref{sec:2sympent}
to study the associated orbit problem. We prove that the general
orbits of such tensors correspond to certain K3 surfaces with Picard
rank at least $12$ over $\fdbar$:

\begin{theorem} \label{thm:22sympenteractorbits} Let $V = V_1 \otimes
  \Sym^2 V_2\otimes\Sym^2 V_3$ for 2-dimensional $\fd$-vector spaces
  $V_1$, $V_2$, $V_3$.  Let $G'$ be the group $\GL(V_1) \times
  \GL(V_2) \times \GL(V_3)$, and let $G$ be the quotient of $G'$ by
  the kernel of the multiplication map $\Gm \times \Gm \times \Gm \to
  \Gm$ given by $(\gamma_1, \gamma_2, \gamma_3) \mapsto \gamma_1
  \gamma_2^2 \gamma_3^2$.  Let $\Lambda$ be the lattice whose Gram
  matrix is

{ \footnotesize
\begin{equation} \label{eq:22sympent-NS}
\begin{pmatrix}
0 & 2 & 2 & 2 & 0 & 0 & 0 & 0 & 0 & 0 & 0 & 0 \\
2 & 0 & 2 & 2 & 0 & 1 & 0 & 0 & 1 & 0 & 0 & 1 \\
2 & 2 & 0 & 2 & 0 & 1 & 0 & 0 & 1 & 0 & 0 & 1 \\
2 & 2 & 2 & 0 & 1 & 0 & 1 & 1 & 0 & 1 & 1 & 0 \\
0 & 0 & 0 & 1 & -2 & 1 & 0 & 0 & 0 & 0 & 0 & 0 \\
0 & 1 & 1 & 0 & 1 & -2 & 1 & 0 & 0 & 0 & 0 & 0 \\
0 & 0 & 0 & 1 & 0 & 1 & -2 & 0 & 0 & 0 & 0 & 0 \\
0 & 0 & 0 & 1 & 0 & 0 & 0 & -2 & 1 & 0 & 0 & 0 \\
0 & 1 & 1 & 0 & 0 & 0 & 0 & 1 & -2 & 1 & 0 & 0 \\
0 & 0 & 0 & 1 & 0 & 0 & 0 & 0 & 1 & -2 & 0 & 0 \\
0 & 0 & 0 & 1 & 0 & 0 & 0 & 0 & 0 & 0 & -2 & 1 \\
0 & 1 & 1 & 0 & 0 & 0 & 0 & 0 & 0 & 0 & 1 & -2
\end{pmatrix},
\end{equation}
}%
and let $S = \{ e_1, e_2,  e_3, e_4 \}$.  Then the $G(\fd)$-orbits of
an open subset of $V(\fd)$ are in bijection with the $\fd$-points of
an open subvariety of the moduli space $\sM_{\Lambda,S}$ of K3
surfaces $X$ lattice-polarized by $(\Lambda,S)$.
\end{theorem}

\subsection{N\'eron-Severi lattice and moduli problem}

In order to study the orbits of doubly-doubly symmetric penteracts, we may use
the geometric construction from \S \ref{sec:2sympent}, since an
element of $V_1 \otimes \Sym^2 V_2 \otimes \Sym^2 V_3$ is also an
element of $V_1 \otimes V_2 \otimes V_2 \otimes \Sym^2 V_3$ and $V_1
\otimes \Sym^2 V_2 \otimes V_3 \otimes V_3$.  Thus, the K3 surfaces
$X_{123}$ and $X_{145}$ each have at least six (rank) singularities over
$\fdbar$, and a numerical example shows us that they generically have
exactly six singular points.  Meanwhile, for a generic orbit, all of
the other $X_{ijk}$ (namely, $X_{124}$, $X_{234}$ and $X_{345}$) are
nonsingular, and the maps of the type $X_{1234} \to X_{123}$
blow up the six singular points.

The nonsingular K3's---which are all naturally isomorphic---thus
contain two sets of six mutually non-intersecting lines, namely the
exceptional fibers in $X_{124}$ coming from the blow-ups $X_{124}
\stackrel{\sim}{\to} X_{1234} \to X_{123}$ and $X_{124}
\stackrel{\sim}{\to} X_{1245} \to X_{145}$. Let $P_i$ and $Q_i$ for $1
\leq i \leq 6$ denote these exceptional fibers from $X_{123}$ and
$X_{145}$, respectively.  As explained below, each of the six lines in
any one set intersects exactly two lines in the other set.

The map $X_{1234}\to\Proj(V_1^\vee)$ is a genus one fibration whose
discriminant as a binary form on $V_1$ is of degree 24 and factors as
the product of a fourth power of a cubic form times an irreducible
degree twelve form.  Thus the fibration has three reducible fibers of
type $\mathrm I_4$, i.e., each of these reducible fibers consists of
four lines forming a ``rectangle''.  Each set of six lines in the
previous paragraph contains one pair of parallel lines from each of
the three rectangles.  That is, with choices of indices, each
rectangle is made up of the lines corresponding to $P_i$, $Q_i$,
$P_{i+1},$ and $Q_{i+1}$ for $i = 1, 3, 5$.

To explicitly see this correspondence among the twelve lines, we note
that if $r_0\in \Proj(V_1^\vee)$ is a point giving a singular fiber in
the genus one fibration, then it yields two rank singularities
$(r_0,a,b)$ and $(r_0,b,a)$ on $X_{123}$.  The map $X_{124} \to
X_{123}$ blows up these singularities to the lines $(r_0,a,*)$ and
$(r_0,b,*)$, where we use $*$ to mean that the coordinate in
$\Proj(V_3^\vee)$ may vary freely.  Similarly, each such $r_0$ gives
two rank singularities $(r_0,c,d)$ and $(r_0,d,c)$ on $X_{145}$, and
under $X_{124} \to X_{145}$, these blow up to lines $(r_0,*,c)$ and
$(r_0,*,d)$ for any $* \in \Proj(V_2^\vee)$.  Therefore, the latter
two lines each intersect each of the former two lines in a single
point, giving the four intersection points $(r_0,a,c)$, $(r_0,a,d)$,
$(r_0,b,d)$, and $(r_0,b,c)$ in $X_{124}$.

The usual line bundles $L_i$ for $1 \leq i \leq 4$, given as the
pullbacks of $\sO_{\Proj(V_i^\vee)}(1)$ to $X_{1234}$, satisfy:
\begin{align}
L_1 + L_2 + L_3 &= 2 L_4 + \sum_i P_i \qquad \textrm{and}  \label{eq:22sympentrelation1} \\
2 L_1 - L_2 - L_3 + 2 L_4 & = \sum_i Q_i. \label{eq:22sympentrelation2}
\end{align}
Of course, \eqref{eq:22sympentrelation1} is clear from Lemma
\ref{lem:pentrelation}, and the second relation
\eqref{eq:22sympentrelation2} comes from repeated applications of that
lemma along each step of the composition $X_{1234} \to X_{124} \to
X_{1245} \to X_{145}$.

To determine the N\'eron-Severi lattice (over $\fdbar$) of the K3
surface associated to a generic doubly-doubly symmetric penteract, we
use the explicit geometry and the relations described above in
\eqref{eq:22sympentrelation1} and \eqref{eq:22sympentrelation2} to
compute the intersection numbers between the divisor classes.  The
only nonzero intersection numbers are as follows:
\begin{align*} 
L_i \cdot L_j &= 2 \textrm{ for } i \neq j, \qquad\qquad L_2 \cdot Q_j = L_3 \cdot Q_j = 1, \\
P_i^2 &= Q_i^2 = -2,  \qquad\qquad\qquad L_4 \cdot P_j = 1, \\
P_i \cdot Q_i &= P_i \cdot Q_{i+1} = P_{i+1} \cdot Q_{i+1} = P_{i+1} \cdot Q_i = 1 \textrm{ for } i \in \{1,3,5\}.
\end{align*}
The rank of the intersection matrix is 12, and the span of the
divisors is a sublattice of $\overline{\NS}(X)$ of discriminant 256. A
basis for the N\'eron-Severi lattice is $\{L_1, L_2, L_4, P_1, Q_1, P_2,
  P_3, Q_3, P_4, P_5, Q_5, P_6\}$, and the corresponding Gram matrix is given by 
\eqref{eq:22sympent-NS}.

\begin{proposition} 
  For a generic $X$ in this family of K3 surfaces, $\overline{\NS}(X)$
  is spanned over $\Z$ by the $L_i$ for $i = 1, \dots, 4$ and the
  exceptional classes $P_i$ and $Q_i$ for $i = 1, \dots, 6$.
\end{proposition}

\begin{proof}
  The moduli space of these K3s has dimension $2 \cdot 3 \cdot 3 - 10 = 8$,
  so the Picard number of the generic surface in this space is at most $12$.
  Let $\Lambda$ be the rank $12$ lattice spanned by the above divisors
  (equivalently, with basis $\{L_1$, $L_2$, $L_4$, $P_1$, $Q_1$, $P_2$,
  $P_3$, $Q_3$, $P_4$, $P_5$, $Q_5$, $P_6\}$). Since the discriminant of
  $\Lambda$ is $256$, it is enough to show that it is $2$-saturated.
  Suppose a divisor class
  \[
  D = \alpha_1 L_1 + \alpha_2 L_2 + \alpha_4 L_4 + \sum_{i=1}^6  \beta_i P_i + \gamma_1 Q_1 + \gamma_3 Q_3 + \gamma_5 Q_5
  \]
  is in $\overline{\NS}(X)$ for some collection of rational numbers
  $\alpha_i, \beta_j, \gamma_k$ whose denominators are powers of
  $2$. Then by symmetry, so is the divisor $D'$ obtained by replacing
  $L_2$ by $L_3$. Then $D - D' = \alpha_2( L_2 - L_3) \in
  \overline{\NS}(X)$, which forces $\alpha_2$ to be an integer, since
  the self-intersection $-4 \alpha_2^2$ of this divisor must be an
  even integer. So we may assume $\alpha_2 = 0$, and similarly,
  $\alpha_4 = 0$. Therefore, $D$ has the form
  \[
  D = \alpha_1 L_1 + \sum_{i=1}^6 \beta_i P_i + \gamma_1 Q_1 +  \gamma_3 Q_3 + \gamma_5 Q_5.
  \] 
  Again by symmetry, the divisor $D'' = \alpha_1 L_1 + \sum_{i=1}^6
  \beta_i P_i + \gamma_1 Q_2 + \gamma_3 Q_4 + \gamma_5 Q_6$ is also in
  $\overline{\NS}(X)$. Then $D - D'' \in \overline{\NS}(X)$ forces all
  the $\gamma_i$ to vansh, by Lemma~\ref{lem:nikulin} (since the $Q_i$
  are all disjoint $(-2)$-curves). Hence $D = \alpha_1 L_1 +
  \sum_{i=1}^6 \beta_i P_i$. Then $D''' = \alpha_1 L_1 + \sum_{i=1}^4
  \beta_i P_i + \beta_6 P_5 + \beta_5 P_6$ is also in
  $\overline{\NS}(X)$ by symmetry, and considering $D - D'''$ shows
  that $\beta_5$ and $\beta_6$ are equal modulo $\Z$. Similar symmetry
  arguments force all the $\beta_i$ to be equal to each other. Hence
  $D = \alpha L_1 + \beta (\sum P_i)$. Then $D^2 = -6 \beta^2$ is an
  even integer, which forces $\beta \in \Z$. Subtracting $\beta (\sum
  P_i) \in \Lambda$, we may assume $\alpha L_1 \in
  \overline{\NS}(X)$. In fact, this forces $\alpha \in \Z$ as well,
  since $L_1$ is the class of an elliptic fiber and cannot be a
  nontrivial multiple of another divisor.
\end{proof}

We now complete the proof of Theorem \ref{thm:22sympenteractorbits}.

\begin{proof}[Proof of Theorem $\ref{thm:22sympenteractorbits}$]
  The above geometric constructions explain how to obtain a
  $(\Lambda,S)$-polarized K3 surface from a doubly-doubly symmetric
  penteract.  On the other hand, given such a K3 $X$, let $L_1$,
  $L_2$, $L_3$, and $L_4$ be line bundles corresponding to the
  elements of $S$.  Then we may use these line bundles as in \S
  \ref{sec:penteractreverse} to produce a penteract $A\in \cH^0(X,L_1)
  \otimes \cH^0(X,L_2) \otimes \cH^0(X,L_3) \otimes \cH^0(X,L_4)
  \otimes (\ker \mu)^\vee$, where $\mu$ is the usual multiplication
  map; we now show that it has the appropriate symmetry.

  By the argument in the proof of Theorem
  \ref{thm:2sympenteractorbits}, and using Lemma
  \ref{lem:pentrelation}, we immediately see that there is an
  identification of $\cH^0(X,L_4)$ and $(\ker \mu)^\vee$ such that $A$
  is doubly symmetric in those coordinates.

  Similarly, we may switch the roles of the indices $2$ and $3$ with
  those of $4$ and $5$, respectively, to obtain the second symmetry.
  For example, we may use $A$ to construct K3 surfaces $X_{1245}$ and
  $X_{145}$ (with the divisor classes in $\Lambda$).  The line bundles
  $M_i$ coming from pulling back $\sO_{\Proj(V_i)^\vee}(1)$ to
  $X_{1245}$ via projection for $i = 1, 2, 4, 5$ may be used to make
  another penteract $B$, which is $\GL_2^5$-equivalent to $A$ (by the
  proof of Theorem \ref{thm:penteractorbits}).  (Note that in fact
  $M_i$ and $L_i$ are isomorphic for $i = 1, 2, 4$.)  Then the same
  argument as in Theorem \ref{thm:2sympenteractorbits} shows that the
  line bundle $M_2$ is isomorphic to the line bundle $M_3^{(145)}$,
  and in fact, the corresponding vector spaces may be identified so
  that $B$ is symmetric in those two directions.

  Therefore, via the above identifications of vector spaces, our
  penteract $A$ may be viewed as an element of the tensor space
  $\cH^0(X,L_1) \otimes \Sym^2 \cH^0(X,L_2) \otimes \Sym^2
  \cH^0(X,L_4)$, as desired.
\end{proof}

\subsection{Automorphisms} \label{sec:22sympentauts}

We may again consider the automorphisms $\alpha_{ij,k}$ for
doubly-doubly symmetric penteracts.  By symmetry, to understand the
$5$-cycles passing through all five $X_{ijkl}$, which are all
compositions of three $\alpha_{ij,k}$'s, it suffices to understand the
following three:
\begin{align*}
\Phi_{53214} &: X_{1234} \to X_{1245} \to X_{1345} \to X_{2345} \to X_{1235} \to X_{1234}, \\
\Phi_{53421} &: X_{1234} \to X_{1245} \to X_{1235} \to X_{1345} \to X_{2345} \to X_{1234},\\
\textrm{or } \Phi_{53241} &: X_{1234} \to X_{1245} \to X_{1345} \to X_{1235} \to X_{2345} \to X_{1234},
\end{align*} 
We first study the $5$-cycle $\Phi_{54321}$.  Applying Lemma
\ref{lem:pentrelation} and determining intersection numbers, we
compute the action of the automorphism $\Phi_{54321}$ on
$\overline{\NS}(X)$ arising from a general doubly doubly symmetric
penteract as the matrix

{\footnotesize \begin{equation} \label{eq:22sympentaut}
\begin{pmatrix}
1 & 4 & -2 & 2 & 0 & 0 & 0 & 0 & 0 & 0 & 0 & 0 \\
0 & 2 & -1 & 2 & 0 & 0 & 0 & 0 & 0 & 0 & 0 & 0 \\
0 & 1 & 0 & 0 & 0 & 0 & 0 & 0 & 0 & 0 & 0 & 0 \\
1 & 1 & -1 & 1 & 0 & 0 & 0 & 0 & 0 & 0 & 0 & 0 \\
0 & 1 & 0 & 0 & -1 & 0 & 0 & 0 & 0 & 0 & 0 & 0 \\
1 & 1 & -1 & 1 & 0 & -1 & 0 & 0 & 0 & 0 & 0 & 0 \\
0 & 1 & 0 & 0 & 0 & 0 & -1 & 0 & 0 & 0 & 0 & 0 \\
0 & 1 & 0 & 0 & 0 & 0 & 0 & -1 & 0 & 0 & 0 & 0 \\
1 & 1 & -1 & 1 & 0 & 0 & 0 & 0 & -1 & 0 & 0 & 0 \\
0 & 1 & 0 & 0 & 0 & 0 & 0 & 0 & 0 & -1 & 0 & 0 \\
0 & 1 & 0 & 0 & 0 & 0 & 0 & 0 & 0 & 0 & -1 & 0 \\
1 & 1 & -1 & 1 & 0 & 0 & 0 & 0 & 0 & 0 & 0 & -1
\end{pmatrix}
\end{equation}}%
with respect to the basis 
$\{L_1, L_2, L_3, L_4, P_1, Q_1, P_2, P_3, Q_3, P_4, P_5, Q_5 \}$.

The induced action of the automorphism $\Phi_{53421}$ on the $L_i$ is
the same (up to reordering) as in the usual penteract case (and it is
thus simple to compute the action on the exceptional divisors).  For
the automorphism $\Phi_{53241}$, the induced action on the $L_i$ and
the exceptional divisors $Q_j$ are the same (up to reordering) as the
induced action of $\Phi_{54321}$ in the doubly symmetric penteract
case.


\section{Doubly-triply symmetric penteracts: \texorpdfstring{$\Sym^2(2) \otimes \Sym^3(2)$}{Sym2(2) (x) Sym3(2)}} \label{sec:23sympent}

We now study penteracts that are symmetric in the first two
coordinates and also symmetric in the last three coordinates. We prove
that the orbits of the space of doubly-triply symmetric penteracts are
related to certain K3 surfaces with N\'eron-Severi rank at least 15
over $\fdbar$. Note that these penteracts also have an interpretation
as bidegree $(2,3)$ curves in $\Proj^1 \times \Proj^1$, which can be
used to connect the moduli space below with the universal Picard
scheme $\Pic^1_{\sM_2}$ over the moduli space of genus $2$ curves.

\begin{theorem} \label{thm:23sympenteractorbits} Let $V = \Sym^2
  V_1\otimes\Sym^3 V_2$ for $2$-dimensional $\fd$-vector spaces $V_1$
  and $V_2$.  Let $G'$ be the group $\Gm \times \GL(V_1) \times
  \GL(V_2)$, and let $G$ be the quotient of $G'$ by the kernel of the
  multiplication map $\Gm \times \Gm \times \Gm \to \Gm$ given by
  $(\gamma_1, \gamma_2, \gamma_3) \mapsto \gamma_1 \gamma_2^2
  \gamma_3^3$.  Let $\Lambda$ be the lattice whose Gram matrix is

{\footnotesize \begin{equation} \label{eq:23sympent-NS}
\begin{pmatrix}
0 & 2 & 2 & 2 & 0 & 0 & 0 & 0 & 0 & 0 & 0 & 0 & 0 & 0 & 1 \\
 2 & 0 & 2 & 2 & 0 & 0 & 0 & 0 & 0 & 0 & 0 & 0 & 0 & 0 & 1 \\
 2 & 2 & 0 & 2 & 0 & 1 & 0 & 1 & 0 & 1 & 0 & 1 & 0 & 1 & 0 \\
 2 & 2 & 2 & 0 & 1 & 0 & 1 & 0 & 1 & 0 & 1 & 0 & 1 & 0 & 0 \\
 0 & 0 & 0 & 1 & -2 & 1 & 0 & 0 & 0 & 0 & 0 & 0 & 0 & 0 & 1 \\
 0 & 0 & 1 & 0 & 1 & -2 & 0 & 0 & 0 & 0 & 0 & 0 & 0 & 0 & 0 \\
 0 & 0 & 0 & 1 & 0 & 0 & -2 & 1 & 0 & 0 & 0 & 0 & 0 & 0 & 1 \\
 0 & 0 & 1 & 0 & 0 & 0 & 1 & -2 & 0 & 0 & 0 & 0 & 0 & 0 & 0 \\
 0 & 0 & 0 & 1 & 0 & 0 & 0 & 0 & -2 & 1 & 0 & 0 & 0 & 0 & 0 \\
 0 & 0 & 1 & 0 & 0 & 0 & 0 & 0 & 1 & -2 & 0 & 0 & 0 & 0 & 1 \\
 0 & 0 & 0 & 1 & 0 & 0 & 0 & 0 & 0 & 0 & -2 & 1 & 0 & 0 & 0 \\
 0 & 0 & 1 & 0 & 0 & 0 & 0 & 0 & 0 & 0 & 1 & -2 & 0 & 0 & 1 \\
 0 & 0 & 0 & 1 & 0 & 0 & 0 & 0 & 0 & 0 & 0 & 0 & -2 & 1 & 0 \\
 0 & 0 & 1 & 0 & 0 & 0 & 0 & 0 & 0 & 0 & 0 & 0 & 1 & -2 & 0 \\
 1 & 1 & 0 & 0 & 1 & 0 & 1 & 0 & 0 & 1 & 0 & 1 & 0 & 0 & -2
\end{pmatrix}
\end{equation}}%
and let $S = \{ e_1, e_2, e_3, e_4 \}$. Then the $G(\fd)$-orbits of an
open subset of $V(\fd)$ are in bijection with the $\fd$-points of an
open subvariety of the moduli space $\sM_{\Lambda,S}$ of K3
surfaces $X$ lattice-polarized by $(\Lambda,S)$.
\end{theorem}

\subsection{N\'eron-Severi lattice and moduli problem}

By exploiting the double symmetry in the first two coordinates and
triple symmetry in the last three coordinates, we may apply the
constructions in all of the previous penteract sections!  That is, the
space $\Sym^2 V_1 \otimes \Sym^3 V_2$ is a subspace of doubly
symmetric penteracts $V_1 \otimes V_1 \otimes V_2 \otimes \Sym^2 V_2$
or $V_2 \otimes V_2 \otimes V_2 \otimes \Sym^2 V_1$, triply symmetric
penteracts $V_1 \otimes V_1 \otimes \Sym^3 V_2$, and doubly-doubly
symmetric penteracts $V_2 \otimes \Sym^2 V_1 \otimes \Sym^2 V_2$.

Given an element of $\Sym^2 V_1 \otimes \Sym^3 V_2$, using the usual
notation, we construct K3 surfaces $X_{123}$ ($=X_{124}=X_{125}$) and
$X_{345}$ that have at least six rank singularities (over $\fdbar$).
A numerical example shows us that generically there are exactly six
singularities on each of $X_{123}$ and $X_{345}$.  Meanwhile, the
other K3 surfaces ($X_{134}$, $X_{1234}$, and $X_{1345}$) are
generically nonsingular and isomorphic, and the maps $X_{134} \to
X_{1234} \to X_{123}$, $X_{134} \to X_{1234} \to X_{124}$, and
$X_{134} \to X_{1345} \to X_{345}$ blow down sets of six lines to each
of the corresponding sets of six singular points.  The nonsingular K3
surface $X_{1234}$ contains at least three sets of six lines; call the
divisors corresponding to these lines $P_i$, $Q_i$, and $E_i$,
respectively, for $1 \leq i \leq 6$.  Their intersection numbers are
computed below, by using the various genus one fibrations.

Either projection map from $X_{1234}$ to $\Proj(V_1^\vee)$ is a genus
one fibration whose discriminant is a degree $24$ binary form on $V_1$
that factors as the cube of a degree six form and an irreducible
degree six form (as a special case of the triply symmetric penteract).
Thus the fibration has six reducible fibers of type $\mathrm I_3$.
Each of these reducible fibers consists of three lines in a
``triangle''; in each of the six triangles, there are two
distinguished lines that correspond to $P_i$ and $Q_i$, respectively,
for $1 \leq i \leq 6$.

Either projection map from $X_{1234}$ to $\Proj(V_2^\vee)$ is a genus
one fibration with three reducible fibers of type $\mathrm{I}_4$ (as a
special case of the doubly-doubly symmetric penteract). That is, each
of the three reducible fibers is a rectangle, and two of the opposite
sides of each rectangle are the $6$ lines $E_i$ from the $18$
described above. The other two parallel sides in each rectangle are
given by $P_{2j-1}$ and $P_{2j}$ for $1 \leq j \leq 3$, for the
projection $\pi_3$ to the third factor of $\Proj^1$; for the
projection $\pi_4$, we get a similar picture, but with the $P$'s
replaced by the $Q$'s.

We explicitly list these $18$ lines in $X_{1234}$.  For $1 \leq j \leq
3$, let $(r_j, s_j) \in \Proj(V_1^\vee) \times \Proj(V_2^\vee)$ be
distinct images of singular points on $X_{123}$ under the projection
$\pi_{12}$ to $\Proj(V_1^\vee) \times \Proj(V_1^\vee)$, where $r_j
\neq s_k$ for any $1 \leq j, k \leq 3$.  Then the other singular
points will project to $(s_j, r_j)$ for $1 \leq j \leq 3$, so the six
singular points on $X_{123}$ will be, for $1 \leq j \leq 3$, given by
$(r_j, s_j, t_j)$ and $(s_j, r_j, t_j)$ for some $t_j \in
\Proj(V_2^\vee)$.  Thus, we obtain four lines in $X_{1234}$ for each
$j$:
\begin{align*}
P_{2j-1} &= (r_j, s_j, t_j, *) &\quad Q_{2j-1} &= (r_j, s_j, *, t_j), \\
P_{2j} &= (s_j, r_j, t_j, *) &\quad Q_{2j} &= (s_j, r_j, *, t_j),
\end{align*}
where we again use $*$ to mean that the coordinate varies freely in
the appropriate $\Proj^1$.  These are the six pairs of lines in the
reducible $\mathrm{I}_2$ fibers in the projection $X_{1234} \to
\Proj(V_1^\vee)$, and it is clear that they are the blowups of the
singular points from $X_{123}$ and $X_{124}$.

Using these explicit points in projective space, we also see that the
six singular points in $X_{345}$ are just the points $(t_i, t_j, t_k)$
for the permutations $\{i,j,k\}$ of $\{1,2,3\}$.  Therefore, the
surface $X_{1234}$ contains six lines of the form
\[
E_i = (*, \diamond, t_j, t_k), \qquad E_{i+3} = (*, \diamond, t_k, t_j)
\] 
for $(i,j,k)$ ranging over cyclic permutations of $(1,2,3)$, where
$*$ and $\diamond$ are varying coordinates connected by a (1,1) equation. 
These are the six lines $E_i$ described above.

The line bundles $L_i$ obtained from pulling back
$\sO_{\Proj(V_i^\vee)}(1)$ to $X_{1234}$, for $1 \leq i \leq 4$,
satisfy the following relations with the exceptional lines:
\begin{align}
L_1 + L_2 + L_3 &= 2 L_4 + \sum_{i=1}^6 P_i, \label{eq:sym23pentrelation1} \\
L_1 + L_2 + L_4 &= 2 L_3 + \sum_{i=1}^6 Q_i, \qquad \textrm{and} \label{eq:sym23pentrelation2} \\
- L_1 - L_2 + 2 L_3 + 2 L_4 &= \sum_{i=1}^6 E_i \label{eq:sym23pentrelation3}.
\end{align}
These are obtained from repeated applications of Lemma~\ref{lem:pentrelation}.
The nonzero intersections between all these divisors are as follows:
\begin{align*}
L_i \cdot L_j = 2 \textrm{ for } i \neq j, && L_1 \cdot E_i = L_2 \cdot E_i = 1, \\
L_3 \cdot Q_i = 1, && L_4 \cdot P_i  = 1, &&  P_i \cdot Q_i = 1, \\
P_1, P_2  \textrm{ intersect } E_3, E_5, && P_3, P_4  \textrm{ intersect } E_1, E_6 , && P_5, P_6 \textrm{ intersect } E_2, E_4, \\
Q_1, Q_2  \textrm{ intersect } E_2, E_6, && Q_3, Q_4 \textrm{ intersect } E_3, E_4 , && Q_5, Q_6 \textrm{ intersect } E_1, E_5,
\end{align*}
where ``intersect'' means has intersection number $1$.  As a
consequence, the N\'eron-Severi lattice of $(X_{1234})_{\fdbar}$ has
rank 15 and discriminant 108. A basis for the lattice consists of the
divisors $L_1$, $L_2$, $L_3$, $L_4$, $P_1$, $Q_1$, $\dots$, $P_5$,
$Q_5$, $E_3$, and the corresponding Gram matrix is
\eqref{eq:23sympent-NS}.

\begin{proposition} For a very general $X$ in this family of K3 surfaces,
  $\overline{\NS}(X)$ is spanned over $\Z$ by $L_i$, $i = 1, \dots,
  4$, and the exceptional classes $P_i, Q_i, E_i$, $i = 1, \dots, 6$.
\end{proposition}

\begin{proof}
  Let $\Lambda$ be the rank $15$ lattice spanned by the $L_i, P_i, Q_i, E_i$
  as in the statement of the proposition.
  Consider the elliptic fibration $\pi: X \to \Proj(V_1^\vee)$. Taking
  the class of the zero section to be
  \[
  O = -L_1 - L_2 + 2P_1 + Q_1 + 2P_2+ Q_2 + P_3 + 2Q_3 + P_4 + 2Q_4 +
  3E_3,
  \]
  we see that $\{P_1, Q_1\}, \dots, \{P_6, Q_6 \}$ give the
  non-identity components of the six $\mathrm{I}_3$ fibers. The curve
  $E_1$ gives a section $Q$ of height $4/3$, whereas the class of $L_3
  - L_2 + L_1 + O$ is a $3$-torsion section $T$. The discriminant of
  the lattice spanned by these sections and the components of the
  fibers is $4/3 \cdot 3^6/9 = 108$, so it is all of $\Lambda$. We
  must now show that generically, $\Lambda$ is all of $\overline{\NS}(X)$.

  For a very general $X$ in this family of K3 surfaces, the rank of $\overline{\NS}(X)$
  is at most $15$, since the dimension of the moduli space is $3 \cdot 4 - 7 = 5$.
  Since the discriminant of $\Lambda$ is $108 = 2^2 \cdot 3^3$, 
  we just need to check that $\Lambda$ is
  $2$- and $3$-saturated. We observe from the configuration of fibers
  that there cannot be a $2$-torsion section, and that $Q$ cannot be
  twice a section $Q'$, since this would force the height of $Q'$ to
  be $1/3$, which is impossible. This checks $2$-saturation. For
  similar reasons, $Q$ cannot be thrice a point, and there cannot be a
  $9$-torsion point. So we just need to check that the elliptic
  surface does not have full $3$-torsion. Observe that if $T'$ were
  another $3$-torsion section, independent of $T$ over $\F_3$, then to
  have height $0 = 4 - 6(2/3)$, both $T$ and $T'$ must intersect
  non-identity components of each of the six $\mathrm{I}_3$
  fibers. But then at least one $T+T'$ or $T-T'$ cannot satisfy the
  same property, and yet it is a $3$-torsion point. This gives a
  contradiction.
\end{proof}

We now complete the proof of Theorem \ref{thm:23sympenteractorbits}.

\begin{proof}[Proof of Theorem $\ref{thm:23sympenteractorbits}$]
  The above discussion explains how to construct a K3 surface
  lattice-polarized by $(\Lambda, S)$ from a doubly-triply symmetric
  penteract. Given such a K3 surface $X$, let $L_1$, $L_2$, $L_3$, and $L_4$
  be line bundles on $X$ corresponding to the elements of $S$.  As in
  the previous cases, we use the construction from \S
  \ref{sec:penteractreverse} to build a penteract $A \in \cH^0(X, L_1)
  \otimes \cH^0(X, L_2) \otimes \cH^0(X, L_3) \otimes \cH^0(X, L_4)
  \otimes (\ker \mu)^\vee$.  The proof of Theorem
  \ref{thm:3sympenteractorbits} shows that there is a simultaneous
  identification of $\cH^0(X, L_3)$, $\cH^0(X,L_4)$, and $(\ker
  \mu)^\vee$ that shows that $A$ is triply symmetric with respect to
  those coordinates, i.e., we can think of $A$ as an element of
  $\cH^0(X, L_1) \otimes \cH^0(X, L_2) \otimes \Sym^3 \cH^0(X, L_3)$.

  To show the last symmetry, we use an argument similar to that in the
  proof of Theorem~\ref{thm:22sympenteractorbits}.  That is, we switch
  the roles of the indices $1$ and $2$ with $4$ and $5$ and apply the
  proof of Theorem~\ref{thm:2sympenteractorbits}.  Using the
  intersection matrix \eqref{eq:23sympent-NS}, combined with Lemma
  \ref{lem:pentrelation}, shows that the vector spaces $\cH^0(X,L_1)$
  and $\cH^0(X,L_2)$ may be identified and $A$ is symmetric in those
  two coordinates.

  Thus, we obtain a penteract $A$ in the space $\Sym^2
  \cH^0(X,L_1) \otimes \Sym^3 \cH^0(X,L_3)$, as desired.
\end{proof}

\subsection{Automorphisms} \label{sec:23sympentauts}

The automorphisms $\alpha_{ij,k}$ again arise for doubly-triply
symmetric penteracts.  The $5$-cycles through all five $X_{ijkl}$ are
all equivalent to either
\begin{align*}
\Phi_{53214} &: X_{1234} \to X_{1245} \to X_{1345} \to X_{2345} \to X_{1235} \to X_{1234} \\
\textrm{or } \Phi_{52413} &: X_{1234} \to X_{1345} \to X_{1235} \to X_{2345} \to X_{1245} \to X_{1234}.
\end{align*}
Applying Lemma~\ref{lem:pentrelation} and computing intersection
numbers, we find that the action of the first automorphism
$\Phi_{53214}$ on $\overline{\NS}(X)$ arising from a general
doubly-triply symmetric penteract is given by the matrix

{\footnotesize \begin{equation} \label{eq:23sympentaut}
\begin{pmatrix}
2 & -1 & 0 & 2 & 0 & 0 & 0 & 0 & 0 & 0 & 0 & 0 & 0 & 0 & 0 \\
1 & 0 & 0 & 0 & 0 & 0 & 0 & 0 & 0 & 0 & 0 & 0 & 0 & 0 & 0 \\
0 & 0 & 0 & 1 & 0 & 0 & 0 & 0 & 0 & 0 & 0 & 0 & 0 & 0 & 0 \\
1 & -1 & 1 & 1 & 0 & 0 & 0 & 0 & 0 & 0 & 0 & 0 & 0 & 0 & 0 \\
0 & 0 & 0 & 0 & 0 & 1 & 0 & 0 & 0 & 0 & 0 & 0 & 0 & 0 & 0 \\
1 & 0 & 0 & 0 & -1 & -1 & 0 & 0 & 0 & 0 & 0 & 0 & 0 & 0 & 0 \\
0 & 0 & 0 & 0 & 0 & 0 & 0 & 1 & 0 & 0 & 0 & 0 & 0 & 0 & 0 \\
1 & 0 & 0 & 0 & 0 & 0 & -1 & -1 & 0 & 0 & 0 & 0 & 0 & 0 & 0 \\
0 & 0 & 0 & 0 & 0 & 0 & 0 & 0 & 0 & 1 & 0 & 0 & 0 & 0 & 0 \\
1 & 0 & 0 & 0 & 0 & 0 & 0 & 0 & -1 & -1 & 0 & 0 & 0 & 0 & 0 \\
0 & 0 & 0 & 0 & 0 & 0 & 0 & 0 & 0 & 0 & 0 & 1 & 0 & 0 & 0 \\
1 & 0 & 0 & 0 & 0 & 0 & 0 & 0 & 0 & 0 & -1 & -1 & 0 & 0 & 0 \\
0 & 0 & 0 & 0 & 0 & 0 & 0 & 0 & 0 & 0 & 0 & 0 & 0 & 1 & 0 \\
1 & 0 & 0 & 0 & 0 & 0 & 0 & 0 & 0 & 0 & 0 & 0 & -1 & -1 & 0 \\
-1 & -1 & 0 & 1 & 1 & 0 & 1 & 0 & 1 & 1 & 1 & 1 & 0 & 0 & 1
\end{pmatrix}
\end{equation}}%
with respect to the basis $\{L_1, L_2, L_3, L_4, P_1, Q_1, \ldots,
P_5, Q_5, E_3 \}$.  Also, the induced action of the $5$-cycle
$\Phi_{52413}$ on the $L_i$ and the $Q_j$ is the same (up to
reordering) as the induced action of $\Phi_{54321}$ in the doubly
symmetric penteract case, and as in previous cases, expressions for
the divisors $\Phi_{52413}^* P_j$ and $\Phi_{52413}^* E_j$ follow
immediately.  These two types of $5$-cycles will be shown in \S
\ref{sec:positiveentropy} to be fixed-point-free in general and have
positive entropy.


\section{Quadruply symmetric penteracts: \texorpdfstring{$2 \otimes \Sym^4(2)$}{2 (x) Sym4(2)}} \label{sec:4sympent}

Suppose we now have a penteract that is symmetric in the last {four}
coordinates; we will show that such penteracts give rise to K3
surfaces with N\'eron-Severi rank at least $17$ over $\fdbar$. Note
that such tensors may also be viewed as pencils of binary quartic
forms, whose invariant theory was worked out in \cite{wall}.

\begin{theorem} \label{thm:4sympenteractorbits} Let $V =
  V_1\otimes\Sym^4 V_2$ for $2$-dimensional $\fd$-vector spaces $V_1$
  and $V_2$.  Let $G'$ be the group $\GL(V_1) \times \GL(V_2)$, and
  let $G$ be the quotient of $G'$ by the kernel of the multiplication
  map $\Gm \times \Gm \to \Gm$ given by $(\gamma_1, \gamma_2) \mapsto
  \gamma_1 \gamma_2^4$.  Let $\Lambda$ be the lattice whose Gram
  matrix is

{\footnotesize \begin{equation} \label{eq:4sympent-NS}
\begin{pmatrix}
0 & 2 & 2 & 2 & 0 & 0 & 0 & 0 & 0 & 0 & 0 & 0 & 0 & 0 & 0 & 0 & 0\\
 2 & 0 & 2 & 2 & 1 & 0 & 0 & 1 & 0 & 1 & 0 & 0 & 1 & 0 & 1 & 0 & 0\\
 2 & 2 & 0 & 2 & 0 & 1 & 0 & 0 & 1 & 0 & 1 & 0 & 0 & 1 & 0 & 1 & 0\\
 2 & 2 & 2 & 0 & 0 & 0 & 1 & 0 & 0 & 0 & 0 & 1 & 0 & 0 & 0 & 0 & 1\\
 0 & 1 & 0 & 0 & \!-2\! & 1 & 0 & 0 & 0 & 0 & 0 & 0 & 0 & 0 & 0 & 0 & 0\\
 0 & 0 & 1 & 0 & 1 & \!-2\! & 1 & 0 & 0 & 0 & 0 & 0 & 0 & 0 & 0 & 0 & 0\\
 0 & 0 & 0 & 1 & 0 & 1 & \!-2\! & 1 & 0 & 0 & 0 & 0 & 0 & 0 & 0 & 0 & 0\\
 0 & 1 & 0 & 0 & 0 & 0 & 1 & \!-2\! & 1 & 0 & 0 & 0 & 0 & 0 & 0 & 0 & 0\\
 0 & 0 & 1 & 0 & 0 & 0 & 0 & 1 & \!-2\! & 0 & 0 & 0 & 0 & 0 & 0 & 0 & 0\\
 0 & 1 & 0 & 0 & 0 & 0 & 0 & 0 & 0 & \!-2\! & 1 & 0 & 0 & 0 & 0 & 0 & 0\\
 0 & 0 & 1 & 0 & 0 & 0 & 0 & 0 & 0 & 1 & \!-2\! & 1 & 0 & 0 & 0 & 0 & 0\\
 0 & 0 & 0 & 1 & 0 & 0 & 0 & 0 & 0 & 0 & 1 & \!-2\! & 1 & 0 & 0 & 0 & 0\\
 0 & 1 & 0 & 0 & 0 & 0 & 0 & 0 & 0 & 0 & 0 & 1 & \!-2\! & 1 & 0 & 0 & 0\\
 0 & 0 & 1 & 0 & 0 & 0 & 0 & 0 & 0 & 0 & 0 & 0 & 1 & \!-2\! & 0 & 0 & 0\\
 0 & 1 & 0 & 0 & 0 & 0 & 0 & 0 & 0 & 0 & 0 & 0 & 0 & 0 & \!-2\! & 1 & 0\\
 0 & 0 & 1 & 0 & 0 & 0 & 0 & 0 & 0 & 0 & 0 & 0 & 0 & 0 & 1 & \!-2\! & 1\\
 0 & 0 & 0 & 1 & 0 & 0 & 0 & 0 & 0 & 0 & 0 & 0 & 0 & 0 & 0 & 1 & \!-2\!
\end{pmatrix}
\end{equation}}%
and let $S = \{ e_1, e_2,e_3, e_4 \}$.  Then the $G(\fd)$-orbits of an
open subset of $V(\fd)$ are in bijection with the $\fd$-points of an
open subvariety of the moduli space $\sM_{\Lambda,S}$ of K3
surfaces $X$ lattice-polarized by $(\Lambda,S)$.
\end{theorem}

\subsection{N\'eron-Severi lattice and moduli problem}

Since quadruply symmetric penteracts are also triply symmetric in,
say, the last three coordinates, we may use our constructions from \S
\ref{sec:3sympent} to help us analyze these.  In particular, the K3
surface $X_{123}$ ($=X_{1ij}$ for any $2\leq i < j$) has at least six
rank singularities (over $\fdbar$), and a numerical example shows that
it then generically has exactly six singular points.  Meanwhile, the
K3 surfaces $X_{ijk}$ for $2\leq i < j < k\leq 5$ and all $X_{ijkl}$
are generically nonsingular, and the maps from $X_{1234}$ to $X_{1ij}$
for $2 \leq i < j \leq 4$ blow down lines to the six singular points
on $X_{1ij}$.  We thus have at least $18$ lines on $X_{1234}$.  Denote
the lines coming from $X_{123}$, $X_{124}$, and $X_{134}$ by $E_{\ell
  4}$, $E_{\ell 3}$, and $E_{\ell 2}$, respectively, for $1 \leq \ell
\leq 6$.

The map $X_{1234}\to\Proj(V_1^\vee)$ is a genus one fibration whose
discriminant as a binary form on $V_1$ is of degree 24 and factors as
the sixth power of a degree three form times an irreducible degree six
form.  An argument similar to the doubly-doubly symmetric case shows
that this fibration indeed has three reducible fibers of type $\mathrm
I_6$, i.e., these reducible fibers each consist of six lines in a
``hexagon''.  These three sets of six lines, yielding a total of 18
lines, correspond exactly to the 18 lines in the previous paragraph
(each set of six lines in the previous paragraph contains one pair of
parallel lines from each of the three hexagons).

These lines can be written down very explicitly.  Let $r_i$ for $1
\leq i \leq 3$ be the three points in $\Proj(V_1^\vee)$ over which the
fibration has reducible fibers.  Then the six singular points in
$X_{123}$ are of the form $(r_i, s_i, t_i)$ and $(r_i, t_i, s_i)$ for
$1 \leq i \leq 3$, and each hexagon in $X_{1234}$ consists of the
lines (in cyclic order)
\begin{align*}
E_{i,4} &= (r_i, s_i, t_i, *), & E_{i,3} &= (r_i, s_i, *, t_i), &  E_{i,2} &=  (r_i, *, s_i,t_i), \\
E_{i+3,4} &=  (r_i, t_i, s_i, *), &  E_{i+3,3} &= (r_i,t_i, *, s_i), &  E_{i+3,2} &= (r_i, *, t_i, s_i).
\end{align*}

The three projections $X_{1234} \to \Proj(V_2^\vee)$ are genus one
fibrations with six $\mathrm{I}_3$ fibers, that is, triangles of
lines.  Two of the lines in each triangle come from our $18$ lines.
For example, for the projection to the second factor, the triangles
contain the pair $(r_i, s_i, t_i, *)$ over $s_i$ and $(r_i, s_i, *,
t_i)$, or the pair $(r_i, t_i, s_i, *)$ and $(r_i, t_i, *, s_i)$ over
$t_i$.

If $L_1$, $L_2$, $L_3$, and $L_4$ denote the pullbacks of
$\sO_{\Proj(V_i)^\vee}(1)$ to $X_{1234}$ via the projection maps, we
obtain relations like in Lemma \ref{lem:pentrelation}:
\begin{align*}
L_1 + L_2 + L_3 &= 2 L_4 + \sum_{\ell = 1}^6 E_{\ell 4}, \\
L_1 + L_2 + L_4 &= 2 L_3 + \sum_{\ell = 1}^6 E_{\ell 3},  \\
\textrm{and} \quad L_1 + L_3 + L_4 &= 2 L_2 + \sum_{\ell = 1}^6 E_{\ell 2}.
\end{align*}
The nonzero intersections of these divisors are as follows:
\begin{align*}
  L_i \cdot L_j &= 2 \textrm{ if } i \neq j, \\
  L_i \cdot E_{\ell i} &= 1, \\
  E_{i,j} \cdot E_{i',j'} &= 1 \textrm{ if } i=i' \textrm{ and } |j -
  j'| = 1, \\
 & \qquad  \textrm{ or if } |i-i'| = 3 \textrm{ and } |j-j'| = 2.
\end{align*}

Thus, the N\'eron-Severi lattice of $(X_{1234})_{\fdbar}$ has rank $17
= 3\cdot 5+2$ and discriminant $96$. A basis for the lattice is given
by the divisor classes $L_1$, $L_2$, $L_3$, $L_4$, $E_{12}$, $E_{13}$,
$E_{14}$, $E_{42}$, $E_{43}$, $E_{22}$, $E_{23}$, $E_{24}$, $E_{52}$,
$E_{53}$, $E_{31}$, $E_{32}$, $E_{33}$, and the corresponding Gram
matrix is in \eqref{eq:4sympent-NS}.

\begin{proposition} 
  For a very general $X$ in this family of K3 surfaces,
  $\overline{\NS}(X)$ is spanned by the above divisors.
\end{proposition}

\begin{proof}
  We give a proof using elliptic fibrations. Let $e_1, \dots, e_{22}$
  be the divisor classes $L_1$, \dots, $L_4, E_{12}, \ldots, E_{62},
  E_{13}, \ldots, E_{63}, E_{14}, \ldots, E_{64}$, and $f_1, \dots,
  f_{17}$ the basis chosen. Consider the elliptic fibration with fiber
  class $e_2 = f_2$ (i.e., projection to $\Proj(V_2^\vee)$). A section
  is given by $Z = 2 f_2 - f_5 + f_6 + 2f_8 + f_9$, which we take to
  be our zero section. Then $\{e_6, e_7\}$, $\{e_9, e_{10}\}$,
  $\{e_{12}, e_{13}\}$, $\{e_{15}, e_{16}\}$, $\{e_{18}, e_{19}\}$,
  $\{e_{20}, e_{21}\}$ are the non-identity components of the six
  $\mathrm{I}_3$ fibers. A $3$-torsion section is given by $-f_1 + f_2
  + f_4 + F$. The sections $e_8, e_{11}, e_{17}$ have N\'eron-Tate
  height pairing
  \[
  \begin{pmatrix}
    \frac{8}{3} & 4 & 4 \\
    4 & \frac{20}{3} & 6 \\
    4 & 6 & \frac{20}{3}
  \end{pmatrix}.
  \]
  The determinant of this matrix is $32/27$, and therefore the
  discriminant of the lattice spanned by these sections and the fibers
  is $(32/27) \cdot 3^6 /3^2 = 96$.  Therefore, it is equal to the
  lattice $\Lambda$ spanned by $e_1, \dots, e_{22}$.
  
  We now have to show that $\Lambda$ is equal to $\overline{\NS}(X)$
  for a very general such K3 surface. As usual, the ranks agree
  (because the dimension of the moduli space here is $2 \cdot 5 - 7 =
  3$), so we only need to show $\Lambda$ is saturated. Since the
  discriminant is $96 = 2^5 \cdot 3$, we merely need to show that it
  is $2$-saturated. This can only fail to happen if there is a
  $2$-torsion section (which is not possible because of the fiber
  configuration), or if the Mordell-Weil lattice is larger. A direct
  calculation (by checking all $2^3- 1$ representatives) shows that a
  representative of a nonzero class in $\Lambda^*/\Lambda$ would have
  height $m/3$, for $m$ an odd number. But the height of a section $P$
  equals $4 +2 (P \cdot O)$ minus the sum of the contributions from
  the fibers, which are $0$ or $2/3$, and hence is an even number
  divided by $3$. It follows that $\Lambda$ is saturated.
\end{proof}

\begin{proof}[Proof of Theorem $\ref{thm:4sympenteractorbits}$]
  The above discussion shows that a quadruply symmetric penteract
  gives rise to a K3 surface lattice-polarized by $(\Lambda,S)$.  For
  the reverse, the proof is an obvious generalization of that of
  Theorem \ref{thm:3sympenteractorbits}. The main argument within
  (coming from the proof of Theorem \ref{thm:2sympenteractorbits})
  needs to be repeated three times to show that the constructed
  penteract is symmetric with respect to three transpositions, e.g.,
  $(25)$, $(35)$, and $(45)$, of the tensor factors (under
  identifications of the corresponding vector spaces).
\end{proof}

\subsection{Automorphisms} \label{sec:4sympentauts}

The automorphisms $\alpha_{kl,m}$ and $\Phi_{ijklm}$ also apply in
this case.  All of the $5$-cycles, by the symmetry, act in equivalent
ways, and in particular, $\Phi_{54321}$ induces the action of the
matrix

{\footnotesize
$$\begin{pmatrix}
-1 & 2 & 2 & 0 & 0 & 0 & 0 & 0 & 0 & 0 & 0 & 0 & 0 & 0 & 0 & 0 & 0 \\
-1 & 1 & 1 & 1 & 0 & 0 & 0 & 0 & 0 & 0 & 0 & 0 & 0 & 0 & 0 & 0 & 0 \\
0 & 1 & 0 & 0 & 0 & 0 & 0 & 0 & 0 & 0 & 0 & 0 & 0 & 0 & 0 & 0 & 0 \\
0 & 0 & 1 & 0 & 0 & 0 & 0 & 0 & 0 & 0 & 0 & 0 & 0 & 0 & 0 & 0 & 0 \\
1 & 0 & 0 & 0 & -1 & -1 & -1 & -1 & -1 & 0 & 0 & 0 & 0 & 0 & 0 & 0 & 0 \\
-1 & 0 & 1 & 0 & 0 & 1 & 1 & 1 & 1 & 0 & 0 & 0 & 0 & 0 & 0 & 0 & 0 \\
0 & 1 & 0 & 0 & 0 & -1 & -1 & 0 & 0 & 0 & 0 & 0 & 0 & 0 & 0 & 0 & 0 \\
0 & 0 & 0 & 0 & 0 & 0 & 1 & 0 & 0 & 0 & 0 & 0 & 0 & 0 & 0 & 0 & 0 \\
0 & 0 & 1 & 0 & 0 & 0 & -1 & -1 & 0 & 0 & 0 & 0 & 0 & 0 & 0 & 0 & 0 \\
1 & 0 & 0 & 0 & 0 & 0 & 0 & 0 & 0 & -1 & -1 & -1 & -1 & -1 & 0 & 0 & 0 \\
-1 & 0 & 1 & 0 & 0 & 0 & 0 & 0 & 0 & 0 & 1 & 1 & 1 & 1 & 0 & 0 & 0 \\
0 & 1 & 0 & 0 & 0 & 0 & 0 & 0 & 0 & 0 & -1 & -1 & 0 & 0 & 0 & 0 & 0 \\
0 & 0 & 0 & 0 & 0 & 0 & 0 & 0 & 0 & 0 & 0 & 1 & 0 & 0 & 0 & 0 & 0 \\
0 & 0 & 1 & 0 & 0 & 0 & 0 & 0 & 0 & 0 & 0 & -1 & -1 & 0 & 0 & 0 & 0 \\
-1 & 1 & 1 & -2 & 1 & 1 & 0 & 1 & 1 & 1 & 1 & 0 & 1 & 1 & 0 & 0 & -1 \\
1 & -1 & 0 & 2 & -1 & -1 & 0 & -1 & -1 & -1 & -1 & 0 & -1 & -1 & -1 & 0 & 1 \\
0 & 1 & 0 & 0 & 0 & 0 & 0 & 0 & 0 & 0 & 0 & 0 & 0 & 0 & 0 & -1 & -1
 \end{pmatrix}
$$}%
on the divisors $\{L_1, L_2, L_3, L_4, E_{12}, \ldots, E_{62}, E_{13},
\ldots, E_{63}, E_{14}, \ldots, E_{64}\}$ in $\overline{\NS}(X)$.
This automorphism has order $4$ and is fixed-point-free; it will be
discussed further in \S \ref{sec:finiteorderauts}.  Note, however,
that the square of $\Phi_{54321}$ is an involution but not
fixed-point-free over $\fdbar$ (see \S \ref{sec:5sympentauts} for a
detailed explanation, which also applies to this case).


\section{Quintuply symmetric penteracts: \texorpdfstring{$\Sym^5(2)$}{Sym5(2)}} \label{sec:5sympent}

Let us now consider a quintuply symmetric penteract. We will prove
that the general orbits of such tensors correspond to certain K3
surfaces with Picard rank at least $18$ over $\fdbar$. Note that such
penteracts also have an interpretation as binary quintic forms, i.e.,
degree $5$ subschemes of the projective line.

\begin{theorem} \label{thm:5sympenteractorbits} Let $V = \Sym^5 V_1$
  for a $2$-dimensional $\fd$-vector space $V_1$.  Let $G'$ be the
  group $\Gm \times \GL(V_1)$ and let $G$ be the quotient of $G'$ by
  the kernel of the multiplication map $\Gm \times \Gm \to \Gm$ given
  by $(\gamma_1, \gamma_2) \mapsto \gamma_1 \gamma_2^5$.  Let
  $\Lambda$ be the lattice whose Gram matrix is

{\footnotesize
\begin{equation} \label{eq:5sympent-NS}
\begin{pmatrix}
\!-2\! & 0 & 0 & 0 & 0 & 0 & 1 & 0 & 0 & 0 & 0 & 0 & 0 & 0 & 0 & 0 & 0 & 1 \\
 0 & \!-2\! & 0 & 0 & 0 & 0 & 0 & 1 & 0 & 0 & 0 & 1 & 0 & 0 & 0 & 0 & 0 & 0 \\
 0 & 0 & \!-2\! & 0 & 0 & 0 & 0 & 0 & 1 & 0 & 0 & 0 & 1 & 0 & 0 & 0 & 0 & 0 \\
 0 & 0 & 0 & \!-2\! & 0 & 0 & 0 & 0 & 0 & 1 & 0 & 0 & 0 & 0 & 0 & 1 & 0 & 0 \\
 0 & 0 & 0 & 0 & \!-2\! & 0 & 0 & 0 & 0 & 0 & 1 & 0 & 0 & 0 & 1 & 0 & 1 & 0 \\
 0 & 0 & 0 & 0 & 0 & \!-2\! & 0 & 0 & 0 & 0 & 0 & 1 & 0 & 0 & 0 & 0 & 1 & 0 \\
 1 & 0 & 0 & 0 & 0 & 0 & \!-2\! & 0 & 0 & 0 & 0 & 0 & 1 & 0 & 0 & 1 & 0 & 0 \\
 0 & 1 & 0 & 0 & 0 & 0 & 0 & \!-2\! & 0 & 0 & 0 & 0 & 0 & 1 & 0 & 0 & 0 & 0 \\
 0 & 0 & 1 & 0 & 0 & 0 & 0 & 0 & \!-2\! & 0 & 0 & 0 & 0 & 0 & 0 & 0 & 0 & 0 \\
 0 & 0 & 0 & 1 & 0 & 0 & 0 & 0 & 0 & \!-2\! & 0 & 0 & 0 & 0 & 1 & 0 & 0 & 1 \\
 0 & 0 & 0 & 0 & 1 & 0 & 0 & 0 & 0 & 0 & \!-2\! & 0 & 0 & 0 & 0 & 0 & 0 & 0 \\
 0 & 1 & 0 & 0 & 0 & 1 & 0 & 0 & 0 & 0 & 0 & \!-2\! & 0 & 0 & 0 & 1 & 0 & 0 \\
 0 & 0 & 1 & 0 & 0 & 0 & 1 & 0 & 0 & 0 & 0 & 0 & \!-2\! & 0 & 0 & 0 & 1 & 0 \\
 0 & 0 & 0 & 0 & 0 & 0 & 0 & 1 & 0 & 0 & 0 & 0 & 0 & \!-2\! & 0 & 0 & 0 & 1 \\
 0 & 0 & 0 & 0 & 1 & 0 & 0 & 0 & 0 & 1 & 0 & 0 & 0 & 0 & \!-2\! & 0 & 0 & 0 \\
 0 & 0 & 0 & 1 & 0 & 0 & 1 & 0 & 0 & 0 & 0 & 1 & 0 & 0 & 0 & \!-2\! & 0 & 0 \\
 0 & 0 & 0 & 0 & 1 & 1 & 0 & 0 & 0 & 0 & 0 & 0 & 1 & 0 & 0 & 0 & \!-2\! & 0 \\
 1 & 0 & 0 & 0 & 0 & 0 & 0 & 0 & 0 & 1 & 0 & 0 & 0 & 1 & 0 & 0 & 0 & \!-2\!
\end{pmatrix}
\end{equation}
}%
and let $S = \{ e_1, e_2,e_3, e_4 \}$.  Then the $G(\fd)$-orbits of
an open subset of $V(\fd)$ are in bijection with the $\fd$-points of
an open subvariety of the moduli space $\sM_{\Lambda,S}$ of K3
surfaces $X$ lattice-polarized by $(\Lambda,S)$.
\end{theorem}

\subsection{N\'eron-Severi lattice and moduli problem}

As before, we may use the constructions from previous sections; in
particular, a quintuply symmetric penteract is also quadruply
symmetric in any four coordinates.  By the results of \S
\ref{sec:4sympent}, all the K3 surfaces $X_{ijk}$ (for any $1 \leq i <
j < k \leq 5$) has at least six rank singularities (over $\fdbar$),
and it is easy to check numerically that $X_{ijk}$ generically has
exactly six singular points.  Meanwhile, the K3 surfaces $X_{ijkl}$
for $1\leq i < j< k < l \leq 5$ are generically nonsingular, and for
$1 \leq i < j < k \leq 4$, the projections $X_{1234}\to X_{ijk}$ blow
up the six singular points lying on each of the four surfaces
$X_{ijk}$, yielding 24 lines on $X_{1234}$.

For $\{i,j,k,m\} = \{1,2,3,4\}$, we denote the six lines coming from
the blowup $X_{1234} \to X_{ijk}$ by $E_{\sigma}$, for the
permutations $\sigma$ in the symmetric group $S_4$ with
$\sigma^{-1}(1) = m$. Let $L_i$ denote the pullback of the line bundle
$\sO_{\Proj(V_i^\vee)}(1)$ to $X_{1234}$ via the projection. By Lemma
\ref{lem:pentrelation}, we have
\begin{equation} \label{eq:5sympent-relation}
L_i + L_j + L_k = 2 L_m + \sum_{\substack{\sigma \in S_4 \\ \sigma^{-1}(1) = m}} E_{\sigma}.
\end{equation}

Because of the symmetry, if $(a,b,c)$ is one of the singular points in
$X_{123}$, then the other five singular points are just permutations
of the three coordinates.  Therefore, the $6$ lines $E_{\sigma}$ in
$X_{1234}$ obtained from blowing up the six singular points in
$X_{123}$ are given by $\{(\tau(a),\tau(b),\tau(c),*)\} \subset
X_{1234}$, for each permutation $\tau \in S_3$, where $*$ means that
any point of $\Pone$ may be used. More generally, the $24$ lines are
given by the permutations of $(a,b,c,*)$.  Each line intersects
exactly one of the lines in the other three sets of $6$, namely when
two of their non-$*$ coordinates coincide.  If we view these $24$
lines as vertices of a graph, with the edges corresponding to the $36$
intersection points, this graph is the generalized Petersen graph on
$12$ vertices.

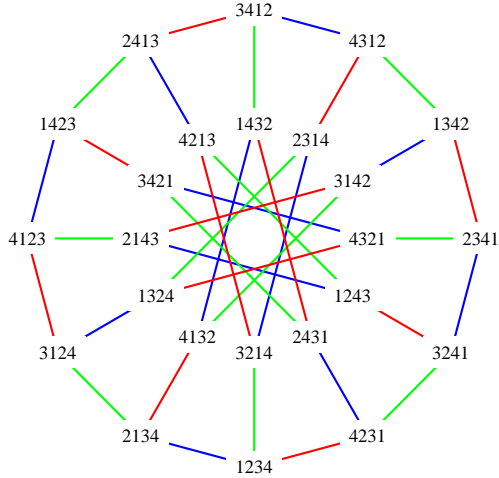
\begin{figure}
\begin{center}
\begin{tikzpicture}
\draw[blue,thick] (4.50,0)--(1.70,0.750);
\draw[blue,thick] (4.30,0.750)--(5.60,1.50);
\draw[blue,thick] (3.75,1.30)--(3.00,-1.50);
\draw[blue,thick] (3.00,1.50)--(2.25,-1.30);
\draw[blue,thick] (2.25,1.30)--(1.50,2.60);
\draw[blue,thick] (1.50,0)--(4.30,-0.750);
\draw[blue,thick] (1.70,-0.750)--(0.402,-1.50);
\draw[blue,thick] (3.75,-1.30)--(4.50,-2.60);
\draw[blue,thick] (6.00,0)--(5.60,-1.50);
\draw[blue,thick] (4.50,2.60)--(3.00,3.00);
\draw[blue,thick] (0.402,1.50)--(0,0);
\draw[blue,thick] (1.50,-2.60)--(3.00,-3.00);
\draw[green,thick] (4.50,0)--(6.00,0);
\draw[green,thick] (4.30,0.750)--(2.25,-1.30);
\draw[green,thick] (3.75,1.30)--(1.70,-0.750);
\draw[green,thick] (3.00,1.50)--(3.00,3.00);
\draw[green,thick] (2.25,1.30)--(4.30,-0.750);
\draw[green,thick] (1.70,0.750)--(3.75,-1.30);
\draw[green,thick] (1.50,0)--(0,0);
\draw[green,thick] (3.00,-1.50)--(3.00,-3.00);
\draw[green,thick] (5.60,1.50)--(4.50,2.60);
\draw[green,thick] (1.50,2.60)--(0.402,1.50);
\draw[green,thick] (0.402,-1.50)--(1.50,-2.60);
\draw[green,thick] (4.50,-2.60)--(5.60,-1.50);
\draw[red,thick] (4.50,0)--(1.70,-0.750);
\draw[red,thick] (4.30,0.750)--(1.50,0);
\draw[red,thick] (3.75,1.30)--(4.50,2.60);
\draw[red,thick] (3.00,1.50)--(3.75,-1.30);
\draw[red,thick] (2.25,1.30)--(3.00,-1.50);
\draw[red,thick] (1.70,0.750)--(0.402,1.50);
\draw[red,thick] (2.25,-1.30)--(1.50,-2.60);
\draw[red,thick] (4.30,-0.750)--(5.60,-1.50);
\draw[red,thick] (6.00,0)--(5.60,1.50);
\draw[red,thick] (3.00,3.00)--(1.50,2.60);
\draw[red,thick] (0,0)--(0.402,-1.50);
\draw[red,thick] (3.00,-3.00)--(4.50,-2.60);
\draw (4.50,0) node[fill=white]{\scriptsize{{$4321$}}};
\draw (4.30,0.750) node[fill=white]{\scriptsize{{$3142$}}};
\draw (3.75,1.30) node[fill=white]{\scriptsize{{$2314$}}};
\draw (3.00,1.50) node[fill=white]{\scriptsize{{$1432$}}};
\draw (2.25,1.30) node[fill=white]{\scriptsize{{$4213$}}};
\draw (1.70,0.750) node[fill=white]{\scriptsize{{$3421$}}};
\draw (1.50,0) node[fill=white]{\scriptsize{{$2143$}}};
\draw (1.70,-0.750) node[fill=white]{\scriptsize{{$1324$}}};
\draw (2.25,-1.30) node[fill=white]{\scriptsize{{$4132$}}};
\draw (3.00,-1.50) node[fill=white]{\scriptsize{{$3214$}}};
\draw (3.75,-1.30) node[fill=white]{\scriptsize{{$2431$}}};
\draw (4.30,-0.750) node[fill=white]{\scriptsize{{$1243$}}};
\draw (6.00,0) node[fill=white]{\scriptsize{{$2341$}}};
\draw (5.60,1.50) node[fill=white]{\scriptsize{{$1342$}}};
\draw (4.50,2.60) node[fill=white]{\scriptsize{{$4312$}}};
\draw (3.00,3.00) node[fill=white]{\scriptsize{{$3412$}}};
\draw (1.50,2.60) node[fill=white]{\scriptsize{{$2413$}}};
\draw (0.402,1.50) node[fill=white]{\scriptsize{{$1423$}}};
\draw (0,0) node[fill=white]{\scriptsize{{$4123$}}};
\draw (0.402,-1.50) node[fill=white]{\scriptsize{{$3124$}}};
\draw (1.50,-2.60) node[fill=white]{\scriptsize{{$2134$}}};
\draw (3.00,-3.00) node[fill=white]{\scriptsize{{$1234$}}};
\draw (4.50,-2.60) node[fill=white]{\scriptsize{{$4231$}}};
\draw (5.60,-1.50) node[fill=white]{\scriptsize{{$3241$}}};

\end{tikzpicture}
\end{center}
\caption{The intersection graph as a Cayley graph of $S_4$ (also known
  as the {Nauru graph}~\cite{eppstein}).  Each vertex is given by an
  element $\sigma \in S_4$ (represented by the string $\sigma(1)\dots
  \sigma(4)$), and the blue, green, and red edges correspond to the
  actions of the transposition $(12)$, $(13)$, or $(14)$,
  respectively.}
\label{fig:nauru}
\end{figure}

To relate the graph in Figure \ref{fig:nauru} to the lines in our K3
surface $X_{1234}$, note that the vertex corresponding to $\sigma \in
S_4$ represents the line given by the action of $\sigma$ on the
ordered set $(*, a, b, c)$.  For example, the bottom vertex is the
line $(*, a, b, c)$ and intersects the lines $(a, *, b, c)$, $(b, a,
*, c)$, and $(c, a, b, *)$.

Each of the projections $\pi_i: X_{1234} \to \Proj(V_1^\vee)$ to the
$i$th factor, for $1 \leq i \leq 4$, is a genus one fibration whose
discriminant as a binary form on $V_1$ is again of degree 24 and
factors as the sixth power of a degree three form times an irreducible
degree six form.  An argument similar to the quadruply symmetric case
shows that this fibration indeed has three reducible fibers of type
$\mathrm I_6$, i.e., these reducible fibers each consist of six lines
in a ``hexagon''.  These account for $18$ of the $24$ lines
encountered earlier,
namely the $E_{\sigma}$ for $\sigma^{-1}(1) \neq i$; the other $6$
lines, via this projection $\pi_i$, in fact cover the entire
$\Proj(V_1^\vee)$.  Thus, the intersection of the $24$ lines with one
of the reducible fibers is exactly a hexagon of lines and $6$ distinct
points.

The intersections among all these divisors can be described as
follows:
\begin{align*}
L_i \cdot L_j &= 2(1 - \delta_{ij}), \qquad\qquad L_i \cdot E_{\sigma} = \delta_{i,\sigma^{-1}(1)}, \qquad\qquad E_{\sigma}^2 = -2, \\
E_{\sigma} \cdot E_{\sigma'} &= \begin{cases} 1 \textrm{ if there is an edge between the corresponding vertices} \\ 0 \textrm{ otherwise. } 
\end{cases}
\end{align*}
The intersection matrix has rank $18$, and the lattice generated by
these divisors has discriminant $20$, so the N\'eron-Severi lattice of
$(X_{1234})_{\fdbar}$ has Gram matrix \eqref{eq:5sympent-NS} and has
rank $18$. A basis for the lattice is given by $L_1$, $L_2$, $L_3$,
$L_4$, $E_{4321}$, $E_{4312}$, $E_{4231}$, $E_{4132}$, $E_{4213}$,
$E_{3421}$, $E_{3412}$, $E_{2431}$, $E_{1432}$, $E_{2413}$,
$E_{3241}$, $E_{3142}$, $E_{3214}$ and $E_{2134}$.

\begin{proposition}
  Let $\Lambda$ be the lattice spanned by the classes of the above
  divisors $E_{\sigma}$. Then for the K3 surface $X$ arising from a very
  general quintuply symmetric penteract, we have $\overline{\NS}(X) =
  \Lambda$.
\end{proposition}

\begin{proof}
  First, note that the dimension of the moduli space of quintuply
  symmetric penteracts is $6 - 4 = 2$, so the rank of $\overline{\NS}(X)$
  for a very general $X$ is $18$, which is the rank of $\Lambda$.
  
  Consider the elliptic fibration $X \to
  \Proj(V_1^\vee)$. Generically, the root lattice formed by the
  non-identity components of the reducible fibers is $A_5^3$. Since
  $E_{(1)} = E_{1234}$ intersects the fiber class in $1$, it follows
  that the elliptic fibration has a section. The root sublattice has
  rank $15$ and discriminant $6^3$. We check that the (Jacobian of)
  the elliptic fibration has a $3$-torsion section, in fact defined
  over the ground field, and since the Picard number is $18$, the only
  possibility is to have a non-torsion section of height $20/(216/9) =
  5/6 = 4 - 5/6 - 5/6 - 9/6$. We can also check directly that there
  are no $2$-torsion sections, even over the algebraic closure. For
  the Picard group to be any larger, it would have to have
  discriminant $5$, and a Mordell-Weil generator of height $5/24$,
  which is impossible with the fiber configuration.
\end{proof}

\begin{proof}[Proof of Theorem $\ref{thm:5sympenteractorbits}$]
  The above discussion shows that a quintuply symmetric penteract
  produces a K3 surface that is lattice-polarized by $(\Lambda, S)$.
  For the other direction, like for Theorem
  \ref{thm:4sympenteractorbits}, the proof is a straightforward
  generalization of the proof of Theorem
  \ref{thm:3sympenteractorbits}.  We construct a penteract from this
  data, and applying the argument from Theorem
  \ref{thm:2sympenteractorbits} four times shows that all five vector
  spaces related to the penteract may be identified and that the
  penteract is symmetric with respect to any two.
\end{proof}

\subsection{Automorphisms} \label{sec:5sympentauts}

Given a K3 surface $X$ coming from a quintuply symmetric penteract,
the visible automorphisms of $X$ may be described quite easily.

Because of the symmetry in this case, all the automorphisms of the
form $\alpha_{kl,m}$ on $X_{ijkl}$ act in similar ways.  As always,
each $\alpha_{kl,m}$ is an involution of the K3 surface $X_{ijkl}$
switching the $k$th and $l$th coordinates, e.g., $\alpha_{34,5}$ sends
$(a,b,c,d) \in X_{1234}$ to $(a,b,d,c)$.  Thus, these generate a group
of automorphisms isomorphic to $S_4$.  Note that the automorphism
$\Phi_{ijklm}$ introduced in \S \ref{sec:pentauts} is an order $4$
element, as the composition of three of these transpositions.

While $\Phi_{ijklm}$ is fixed-point-free for the general $X$ in this
family (see \S \ref{sec:hypdetfixedpt} or simply observe that the
diagonal $\Proj^1$ in $(\Proj^1)^4$ does not generally intersect $X$),
note that its square is an involution but not fixed-point-free (over
$\fdbar$).  For example, the $\Proj^1\times\Proj^1$ of points
$(a,b,a,b) \in (\Proj^1)^4$ on $X_{1234}$ will be fixed under
$\Phi_{ijklm}^2$; in particular, for the general $X$ in this family,
there will be $8$ fixed points over $\fdbar$, namely the intersection
of $X$ with this diagonal $\Proj^1 \times \Proj^1$ in $(\Proj^1 \times
\Proj^1) \times (\Proj^1 \times \Proj^1)$.


\section{\texorpdfstring{$2 \tns 2 \tns 2 \tns 4$}{2 (x) 2 (x) 2 (x) 4}}
\label{sec:2224}

In this section, we study the space of $2 \times 2 \times 2 \times 4$
matrices and classify their orbits in terms of certain K3 surfaces of
rank at least $13$ over $\fdbar$:

\begin{theorem} \label{thm:2224orbits} Let $V = V_1 \tns V_2 \tns V_3
  \tns U$, where $V_1$, $V_2$, and $V_3$ are $2$-dimensional
  $\fd$-vector spaces and $U$ is a $4$-dimensional $\fd$-vector space.
  Let $G' = \GL(V_1) \times \GL(V_2) \times \GL(V_3) \times \GL(U)$,
  and let $G$ be the quotient of $G'$ by the kernel of the
  multiplication map $\Gm^4 \to \Gm$.  Let $\Lambda$ be the lattice
  whose Gram matrix is

{\footnotesize
\begin{equation} \label{eq:2224NS}
\begin{pmatrix}
4 & 4 & 4 & 4 & 0 & 0 & 0 & 0 & 0 & 0 & 0 & 0 & 0 \\
4 & 0 & 2 & 2 & 0 & 0 & 0 & 1 & 1 & 1 & 1 & 1 & 1 \\
4 & 2 & 0 & 2 & 1 & 1 & 1 & 0 & 0 & 0 & 1 & 1 & 1 \\
4 & 2 & 2 & 0 & 1 & 1 & 1 & 1 & 1 & 1 & 0 & 0 & 0 \\
0 & 0 & 1 & 1 & -2 & 0 & 0 & 0 & 0 & 0 & 0 & 0 & 0 \\
0 & 0 & 1 & 1 & 0 & -2 & 0 & 0 & 0 & 0 & 0 & 0 & 0 \\
0 & 0 & 1 & 1 & 0 & 0 & -2 & 0 & 0 & 0 & 0 & 0 & 0 \\
0 & 1 & 0 & 1 & 0 & 0 & 0 & -2 & 0 & 0 & 0 & 0 & 0 \\
0 & 1 & 0 & 1 & 0 & 0 & 0 & 0 & -2 & 0 & 0 & 0 & 0 \\
0 & 1 & 0 & 1 & 0 & 0 & 0 & 0 & 0 & -2 & 0 & 0 & 0 \\
0 & 1 & 1 & 0 & 0 & 0 & 0 & 0 & 0 & 0 & -2 & 0 & 0 \\
0 & 1 & 1 & 0 & 0 & 0 & 0 & 0 & 0 & 0 & 0 & -2 & 0 \\
0 & 1 & 1 & 0 & 0 & 0 & 0 & 0 & 0 & 0 & 0 & 0 & -2
\end{pmatrix},
\end{equation}
}%
and let $S = \{ e_1, e_2, e_3, e_4\}$.  Then the $G(\fd)$-orbits of an
open subset of $V(\fd)$ are in bijection with the $\fd$-rational
points of an open subvariety of the moduli space
$\sM_{\Lambda,S}$ of K3 surfaces $X$ lattice-polarized by
$(\Lambda,S)$.
\end{theorem}

\subsection{Construction of K3 surfaces}

Given an element $A \in V(K)$, we construct a K3 surface $X$ with
Picard number at least 13 as follows.  We will show that the
intersection of the varieties defined by the equations
\begin{eqnarray} \label{eq:2224eqns1}
A(v_1, v_2, \ccdot, u) &\equiv& 0 \\ \label{eq:2224eqns2}
A(v_1, \ccdot, v_3, u) &\equiv& 0 \\ \label{eq:2224eqns3}
A(\ccdot, v_2, v_3, u) &\equiv& 0
\end{eqnarray}
in $\Proj(V_1^\vee) \times \Proj(V_2^\vee) \times \Proj(V_3^\vee)
\times \Proj(U^\vee)$ is a K3 surface $X$.  The projection $X_U$ of
$X$ to $\Proj(U^\vee)$ is then a quartic surface with $12$
singularities over $\fdbar$ (to be described below).

The projection $X_{123}$ of $X$ to $\Proj(V_1^\vee) \times
\Proj(V_2^\vee) \times \Proj(V_3^\vee)$ is cut out by a single
tridegree $(2,2,2)$ form $f(v_1,v_2,v_3)$.  In order to explicitly
describe this form, let us write $A$ as a quadruple
$(A_1,A_2,A_3,A_4)$ of trilinear forms on $V_1^\vee\times
V_2^\vee\times V_3^\vee$ (by choosing a basis for $U$), and consider
the determinant
{\small
\begin{equation}\label{eq:fmatrix}
D(v_1,v_1',v_2,v_2',v_3,v_3'):=\left| \begin{array}{llll}
A_1(v_1,v_2,v_3) & A_2(v_1,v_2,v_3) & A_3(v_1,v_2,v_3) & A_4(v_1,v_2,v_3) \\[.025in]
A_1(v_1',v_2,v_3) & A_2(v_1',v_2,v_3) & A_3(v_1',v_2,v_3) & A_4(v_1',v_2,v_3) \\[.025in]
A_1(v_1,v_2',v_3) & A_2(v_1,v_2',v_3) & A_3(v_1,v_2',v_3) & A_4(v_1,v_2',v_3) \\[.025in]
A_1(v_1,v_2,v_3') & A_2(v_1,v_2,v_3') & A_3(v_1,v_2,v_3') & A_4(v_1,v_2,v_3') \\
\end{array}
\right|
\end{equation}
}%
for vectors $v_1,v_1'\in V_1^\vee$, $v_2,v_2'\in V_2^\vee$,
$v_3,v_3'\in V_3^\vee$.  Then we observe that if $(v_1,v_2,v_3,v_4)\in
V_1^\vee \times V_2^\vee \times V_3^\vee \times U$ satisfies equations
(\ref{eq:2224eqns1})--(\ref{eq:2224eqns3}), then $u\in U^\vee$ lies in
the (right) kernel of the matrix in (\ref{eq:fmatrix}).  Furthermore,
since the determinant $D$ vanishes if $v_1=cv_1'$, $v_2=cv_2'$, or
$v_3=cv_3'$ for any constant $c\in K$, we see that the polynomial
$D(v_1,v_1',v_2,v_2',v_3,v_3')$ is a multiple of
$\det(v_1,v_1')\det(v_2,v_2')\det(v_3,v_3')$.  The tridegree $(2,2,2)$
form
\begin{equation}\label{eq:fformula}
f(v_1,v_2,v_3):=\frac{D(v_1,v_1',v_2,v_2',v_3,v_3')}{\det(v_1,v_1')\det(v_2,v_2')\det(v_3,v_3')}
\end{equation}
is then easily checked to be irreducble and thus defines the
projection $X_{123}$ of $X$ onto $\Proj(V_1^\vee) \times
\Proj(V_2^\vee) \times \Proj(V_3^\vee)$.

One checks that generically $X_{123}$ is smooth, and thus $X$ and
$X_{123}$ are isomorphic.  Moreover, for $\{i,j,k\} = \{1,2,3\}$, we
have that $X_{123}$ is a double cover of $\Proj(V_i^\vee) \times
\Proj(V_j)^\vee$ branched over a genus~$9$ curve, given by a bidegree
$(4,4)$ equation (namely, the discriminant of $f$ viewed as a
quadratic form on $\Proj(V_k^\vee)$).

Let $X_U$ be the image of $X$ under the fourth projection to
$\Proj(U^\vee)$.  Then $u \in X_U$ if and only if there exists $(v_1,
v_2, v_3) \in \Proj(V_1^\vee) \times \Proj(V_2^\vee) \times
\Proj(V_3^\vee)$ such that the equations
(\ref{eq:2224eqns1})--(\ref{eq:2224eqns3}) are satisfied, which occurs
if and only if the $2 \times 2 \times 2$ cube
$A(\ccdot,\ccdot,\ccdot,u) = A(u)$ has discriminant $0$.  (Recall that
$2 \times 2 \times 2$ cubes have a single $\SL_2 \times \SL_2 \times
\SL_2$-invariant of degree four called the {\em discriminant}, which
is the discriminant of each of the three binary quadratics that arise
from the determinant construction on the cube.  If this discriminant
vanishes, the cube is called {\em singular}, and in this case, all
three of the binary quadratics are multiples of squares of linear
forms, i.e., have double roots in $\Proj^1$.) We conclude that $X_U$
is given by the vanishing of the quartic polynomial $\disc A(u)$.

We may also give the following alternative description of $X_U$.  Let
$Y_{12}$, $Y_{13}$, and $Y_{23}$ denote the threefolds in
$\Proj(V_1^\vee)\times\Proj(V_2^\vee)\times\Proj(U^\vee)$,
$\Proj(V_1^\vee)\times\Proj(V_3^\vee)\times\Proj(U^\vee)$, and
$\Proj(V_2^\vee)\times\Proj(V_3^\vee)\times\Proj(U^\vee)$ defined by
(\ref{eq:2224eqns1}), (\ref{eq:2224eqns2}), and (\ref{eq:2224eqns3}),
respectively.  Then for $(i,j)\in\{(1,2),(1,3),(2,3)\}$, we have that
{$X_U$ is the ramification locus of the double cover $Y_{ij}\to
  \Proj(U^\vee)$ given by projection}.  To see this, fix $u\in
U^\vee$; then $A(v_1,v_2,\ccdot,u)=0$ has generically two solutions
$(v_1,v_2)\in \Proj(V_1^\vee)\times\Proj(V_2^\vee)$, for there are
generically two choices for $v_1$ as the root of the associated binary
quadratic on $V_1^\vee$, and then a uniquely determined choice for
$v_2$ given $v_1$ (namely, $v_2$ is the left kernel of the bilinear
form $A(v_1,\ccdot,\ccdot,u)$).  If this binary quadratic form on
$V_1^\vee$ has only one root (which occurs when $\disc A(u)=0$), then
there will thus be only one $(v_1,v_2)$ giving $A(v_1,v_2,\ccdot,u)$.
Hence $X_U$ is the ramification locus of the double cover $Y_{12}\to
\Proj(U^\vee)$, and similarly is the ramification locus of the double
covers $Y_{13}\to \Proj(U^\vee)$ and $Y_{23}\to \Proj(U^\vee)$.  It
follows, in particular, that the preimage $X_{ijU}$ of $X_U$ in
$Y_{ij}$ is the projection of $X$ onto
$\Proj(V_i^\vee)\times\Proj(V_j^\vee)\times \Proj(U^\vee)$ and is
isomorphic to $X_U$.

\subsection{Singularities and exceptional divisors}

We claim that $X_U$ generically has $12$ singularities; these are
closely related to certain special sets $S_i$ of four points in each
$\Proj(V_i^\vee)$.  To construct this set $S_1$ of four points in
$\Proj(V_1^\vee)$, we consider the $2 \times 2 \times 4$ box
$A(v_1,\ccdot,\ccdot,\ccdot) = A(v_1)$ attached to a given point $v_1 \in
\Proj(V_1^\vee)$.  It has a natural $\SL(V_2) \times \SL(V_3) \times
\SL(U)$-invariant---in fact, an $\SL(V_2 \otimes V_3) \times
\SL(U)$-invariant---of degree four, namely the determinant of $A(v_1)$
when viewed as an element of $(V_2 \otimes V_3) \tns U$.  That is,
with a choice of basis, this is simply the determinant of $A(v_1)$
viewed as a $4 \times 4$ matrix.  This invariant gives a degree 4 form
on $\Proj(V_1^\vee)$, which then cuts out our set $S_1$ of four points
in $\Proj(V_1^\vee)$ over $\fdbar$. The sets $S_i$ for $i=2,3$
are constructed in the analogous manner.

The sets $S_i$ have a further significance.  Consider the projection
$\pi_1: X_{123}\to\Proj(V_1^\vee)$.  The fiber over any point $v_1 \in
V_1^\vee$ is then the curve in $\Proj^1(V_2^\vee)\times
\Proj^1(V_3^\vee)$ defined by the bidegree $(2,2)$-form
$f(v_1,\ccdot,\ccdot)$.  We thus see that $X_{123}\cong X$ is a genus
one fibration over $\Proj(V_1^\vee)$, where each genus one fiber is
described as a bidegree $(2,2)$ curve in $\Proj^1(V_2^\vee)\times
\Proj^1(V_3^\vee)$.  A fiber in this fibration is singular precisely
when the discriminant of this bidegree $(2,2)$ form, which is a binary
form of degree 24 on $\Proj^1(V_1^\vee)$, is zero.  Using
indeterminate entries for $A$, one checks that this degree 24 binary
form factors as the square of a binary quartic form times an
irreducible binary form of degree 16.  Thus, generically, we have 16
nodal fibers, while the remaining four fibers turn out to be {\it
  banana curves}, i.e., they have as components two rational curves
intersecting in two points, as we now show.

In fact, we claim that $S_1$ gives precisely the set of four points
over which the fibers for the map $\pi_1: X\to\Proj(V_1^\vee)$ are
{banana curves}.
Indeed, for $v_1\in S_1$, by construction there exists a (generically
unique) point $u \in \Proj(U^\vee)$ such that $A(v_1,\ccdot,\ccdot,u)
\equiv 0$.  Then the points $(v_2,v_3) \in \Proj(V_2^\vee) \times
\Proj(V_3^\vee)$ such that $(v_1,v_2,v_3,u) \in X$ are cut out by the
single equation $A(\ccdot,v_2,v_3,u) \equiv 0$, which is a bidegree
$(1,1)$-form on $\Proj(V_2^\vee) \times \Proj(V_3^\vee)$.  It follows
that the bidegree $(2,2)$-form that defines the fiber of $\pi_1$ over
$v_1$ factors into two $(1,1)$ forms, as claimed.

Let $E_j$ for $1 \leq j \leq 12$ denote these twelve rational $(1,1)$
curves on $X$ as constructed in the previous paragraph ($1 \leq j \leq
4$ for $i = 1$, $5 \leq j \leq 8$ for $i = 2$, and $9\leq j \leq 12$
for $i =3$).  Recall that each of these $E_j$ must intersect the other
rational curve in its fiber in two points.  To obtain the other
component, we note that for $r\in S_1$, the space of $2 \times 2$
matrices spanned by $A_i(r, \ccdot, \ccdot)$, $i = 1, \ldots, 4$ is
(generically) three-dimensional, by the definition of $S_1$. So there
is a plane conic which describes the linear combinations which are of
rank $1$. With choices of bases for the vector spaces $V_i$ and their
duals, suppose such a rank $1$ matrix is $Z_{s,t} =
(\begin{smallmatrix} s_2t_2 & -s_2t_1 \\ -s_1t_2 &
  s_1t_1 \end{smallmatrix})$; then $(s_1, s_2) \in V_2^\vee$ and
$(t_1, t_2) \in V_3^\vee$ give a point $(r,s,t)$ on the fiber over $r$
(the $U$ component may be computed uniquely, and is the linear
combination above). The locus of these $(s,t) \in \Proj(V_2^\vee)
\times \Proj(V_3^\vee)$ is given by a determinantal condition which
says that $Z_{s,t}$ is linearly dependent with the $A_i$, hence a
$(1,1)$-form.

As we have already noted, the projection $X \to X_{123}$ is an
isomorphism.  The map $X \to X_U$ is, not, however: the 12 rational
curves $E_j$ are blown down to 12 singularities (recall that, for each
$j$, the elements of $E_j$ all have the same $U$-coordinate $u$).
Meanwhile, the other rational curves map to nodal curves.

\subsection{N\'eron-Severi lattice}

We thus have a number of divisors on $X_{\fdbar}$: the $E_j$ for $j =
1, \dots, 12$, as well as $H$ (the pullback of
$\sO_{\Proj(U^\vee)}(1)$) and the $L_i$ (the pullbacks of
$\sO_{\Proj(V_i^\vee)}(1)$) for $i = 1, 2, 3$.  We now compute their
intersection numbers.

First, note that $L_1 \cdot L_2 = 2$ since $X \hookrightarrow
\Proj(V_1^\vee) \times \Proj(V_2^\vee) \times \Proj(V_3^\vee) $ is cut
out by a $(2,2,2)$-form. Intersecting $X$ with the zero loci of
$(1,0,0)$ and $(0,1,0)$ forms, we get an intersection number of $2$,
corresponding to the fact that $X \rightarrow \Proj(V_1^\vee) \times
\Proj(V_2^\vee)$ is a double cover. By symmetry, $L_i \cdot L_j = 2$
for all $i \neq j$.

Next, we show that $L_i \cdot H = 4$. The geometric meaning of the
intersection number $L_3 \cdot H$ is as follows. We fix a point $(r_1,
r_2) \in \Proj(V_3^\vee)$, i.e.\ take a fixed (generic) linear
combination of the front and back faces of our $2 \times 2 \times 2$
cube of linear forms, yielding a $2 \times 2$ matrix of linear
forms. Intersection with $H$ means that we restrict the forms to a
generic hyperplane in $\Proj(U^\vee)$. We look for the number of
points in this plane for which the matrix is singular, and such that
$(r_1, r_2)$ is the unique linear combination of the faces which is
singular. For simplicity, assume that $(1,0)$ is not one of the four
special points in $\Proj(V_3^\vee)$ over which $X$ has a reducible
fiber. Then, due to the $\GL(V_3)$ action, we may compute the
intersection number when $(r_1, r_2)$ is $(1,0)$. The constraint that
the front face
\[
\left(\begin{array}{cc}
a & b \\ c & d
\end{array}\right)
\] 
of our cube is singular describes a conic in the plane $V(H) \subset
\Proj(U^\vee)$. On the other hand, for $(1,0)$ to be the unique linear
combination of the faces which makes the matrix singular, if
\[
\left(\begin{array}{cc}
e & f \\ g & h
\end{array}\right)
\]
denotes the back face of our cube, then we also need the mixed determinant 
\[
ah + ed - fc - bg,
\]
to vanish, and this also describes a plane conic. Generically, these
two conics intersect in four points, proving our assertion.

The nonzero intersections among the divisors $H, L_1, \dots, L_3, E_1, \dots,
E_{12}$ are given by 
\begin{align*}
H^2 = H \cdot L_i &= 4, \qquad\qquad H \cdot E_i = 0, \qquad \qquad L_i \cdot L_j = 2, \\
L_1 \cdot E_i &= 1 \textrm{ for } i \in \{5, \dots, 12\}, \\
L_2 \cdot E_i &= 1 \textrm{ for } i \in \{1,2,3,4, 9,10,11, 12\}, \\
L_3 \cdot E_i &= 1 \textrm{ for } i \in \{1, \dots, 8\}.
\end{align*}
Next, we determine the N\'eron-Severi group of the generic K3 surface
in this family.

\begin{proposition}
  For a generic $X$ in this family of K3 surfaces, $\overline{\NS}(X)$
  is spanned over $\Z$ by $H, E_1, \dots, E_{12}$ and $L_1, L_2, L_3$.
\end{proposition}
\begin{proof}
  We have already demonstrated that $\overline{\NS}(X)$ contains $H,
  E_1, \dots, E_{12}$, with $H^2 = 4$ , $H \cdot E_i = 0$ and $E_i
  \cdot E_j = - 2 \delta_{ij}$. Therefore the rank is at least
  $13$. On the other hand, the moduli space has dimension $2 \cdot 2
  \cdot 2 \cdot 4 - (2^2 - 1) \cdot 3 - 4^2 = 7$. Hence the dimension
  of $\overline{\NS}(X)$ for generic $X$ must be exactly $13$.

  Since $L_i \cdot H = 4$ and $L_i \cdot E_j = 0$ if $j \in
  \{4i-3,4i-2,4i-1,4i\}$, by comparing intersection numbers, we obtain
  \begin{align*}
    L_1 &= H - (E_5 + \dots E_{12})/2, \\
    L_2 &= H - (E_1 + \dots + E_4 + E_9 + \dots + E_{12})/2, \\
    L_3 &= H - (E_1 + \dots + E_8)/2.
  \end{align*}
  Note that $L_1 + L_2 + L_3$ is already in $\Z H + \Z E_1 + \dots +
  \Z E_{12}$.

  The discriminant of the lattice $M$ spanned by $H$, the three $L_i$,
  and the twelve $E_j$ is $4 \cdot 2^{12}/ 2^4 = 2^{10} = 1024$. Since
  this discriminant is a power of $2$, if the N\'eron-Severi lattice
  is larger than $M$, there exists an element $D \in \Q H + \Q E_1 +
  \dots \Q E_{12}$ where all the denominators are powers of $2$.

  In that case, we claim that $2D \in \Z H + \dots \Z E_{12}$.  Let
  $2^e$ be the largest power of $2$ in a denominator of a coefficient
  of $E_i$ in $D$. If $e \geq 2$, then $2^{e-2} D \cdot
  E_i$ is not an integer, since $E_i^2 = -2$. Also, we cannot have $D
  = mH/2^e + (c_1 E_1 + \dots c_{12} E_{12})/2$ with $c_i$ integers,
  $e \geq 2$ and $m$ odd, for then $2^{e-1}D = mH/2 + 2^{e-2}(c_1 E_1
  + \dots c_{12} E_{12})$ is in $\overline{\NS}(X)$, and so is $mH/2$. But
  this is impossible, since $(mH/2)^2 = m^2$ is odd while the
  intersection pairing is even. Without loss of generality (by
  subtracting integer multiples of $H$ and $E_i$), we may thus assume
  that $D = (c_1 E_1 + \dots + c_{12} E_{12})/2$ or $D = (H + c_1 E_1 +
  \dots + c_{12} E_{12})/2$, where $c_i \in \{0, 1\}$.

  In the first case, we note that $\sum c_i \in \{0,8,16\}$ by Lemma
  \ref{lem:nikulin}. Now $\sum c_i = 16$ is impossible, while $\sum
  c_i = 0$ is trivial. Therefore, we need to show that if $\sum c_i =
  8$, then $D$ is one of $H - L_i$. If not, then $D \cdot L_3 \in \Z$
  shows that $c_1 + \cdots +c_8$ is even. It must be at least $4$
  (otherwise $\sum c_i = (c_1 + \dots + c_8) + c_9 + \dots + c_{12}$
  would be less than $8$) and cannot be $8$ (otherwise $D = H -
  L_3$). Finally, it cannot be $6$ (otherwise, subtracting $H - L_3$
  would lead to a divisor $y = (\sum d_i E_i)/2$ with $d_i \in
  \{0,1\}$ and $\sum d_i = 4$, which is impossible). We conclude that
  $c_1 + \cdots +c_8 = 4$. Similarly $c_1 + \cdots + c_4 + c_5 +
  \cdots c_{12} = 4$ and $c_5 + \cdots+ c_{12} = 4$. Adding yields
  $\sum c_i = 6$, which is a contradiction.
\end{proof}

An easy discriminant calculation shows:
\begin{corollary}
  For a generic $X$ in this family of K3 surfacs, $\overline{\NS}(X)$
  has a basis given by $H$, $L_1$, $L_2$, $L_3$, $E_1$, $E_2$, $E_3$,
  $E_5$, $E_6$, $E_7$, $E_9$, $E_{10}$, and $E_{11}$.
\end{corollary}

\subsection{Reverse map}

Starting from the data of a K3 surface $X$ with line bundles $L_1,
L_2, L_3$ and $H$ coming from a $2\times 2\times 2\times 4$ box $A$,
we show how to recover the box.  Consider the map
\begin{equation}\label{construct}
\cH^0(L_1) \otimes \cH^0(L_2) \otimes \cH^0(H) \rightarrow \cH^0(L_1
\otimes L_2 \otimes H).
\end{equation}
The dimension of the domain is $2 \cdot 2 \cdot 4 = 16$. The dimension
of the image can be computed by the Riemann-Roch formula, after noting
that
$$
(L_1 + L_2 + H)^2 = 0 + 0 + 4 + 2 \cdot 2 + 2\cdot 4 + 2 \cdot 4 = 24.
$$ 
Since $L_1 + L_2 + H$ is the class of a big and nef divisor, an
easy application of Riemann-Roch on the K3 surface $X$ yields
$$
\cH^0(L_1 + L_2 + H) = \frac{1}{2} (L_1 + L^2 + H)^2 + \chi(\sO_X) = \frac{24}{2} + 2 = 14.
$$ 
Therefore the kernel (which we will soon identify with $V_3^\vee$)
of (\ref{construct}) has dimension $2$, and we obtain a $2 \times 2
\times 2 \times 4$ box $B\in V_1\otimes V_2\otimes V_3\otimes U$,
where $V_1=H^0(L_1)$, $V_2=H^0(L_2)$, and $U=H^0(H)$.

Let $X(B)\in \Proj(V_1^\vee)\times \Proj(V_2^\vee) \times
\Proj(V_3^\vee) \times \Proj(U^\vee)$ denote the K3 surface associated
to $B$.  To see that $B$ is in fact the desired box $A$ (once
$V_3^\vee$ is correctly identified with the kernel of
(\ref{construct})), it suffices to show that $X(B)_{12U}$ is in fact
equal to $X_{12U}$ as sets in $\Proj(V_1^\vee)\times \Proj(V_2^\vee)
\times\Proj(U^\vee)$.
It is equivalent to show that the threefold $Y(B)_{12}$ associated to
$B$ is the same as the threefold $Y_{12}$ in $\Proj(V_1^\vee)\times
\Proj(V_2^\vee) \times\Proj(U^\vee)$, since $X(B)_{12U}$ (and
$X_{12U}$) is then recovered as the ramification locus of
$Y(B)_{12}=Y_{12}\to\Proj(U^\vee)$.  (In other words, if two $2\times
2\times 2\times 4$ boxes yield the same threefold, then they must be
the same box!)  Now the equality $Y_{12}\subset Y(B)_{12}$ is true by
the very construction of $B$, yielding $X_{12U}\subset X(B)_{12U}$.
Then $X_U\subset X(B)_U$, but since both are defined by quartics, we
have $X_U=X(B)_U$, and then $X_{12U} = X(B)_{12U}$ and also
$Y_{12}\subset Y(B)_{12}$, as desired.

We have proved Theorem \ref{thm:2224orbits}.


\section{\texorpdfstring{$2 \tns 2 \tns \Sym^2 (4)$}{2 (x) 2 (x) Sym2(4)}}
\label{sec:2x2xSym24}

In this section, we study the orbits of $V_1 \tns V_2 \tns \Sym^2
V_3$, where $V_1$, $V_2$, and $V_3$ are $\fd$-vector spaces of
dimensions $2$, $2$, and $4$, respectively.  We show that these orbits
correspond to K3 surfaces lattice-polarized by a rank $2$ lattice:

\begin{theorem} \label{thm:22sym24} Let $V_1$, $V_2$, and $V_3$ be
  $\fd$-vector spaces of dimensions $2$, $2$, and $4$, respectively.
  Let $G' = \GL(V_1) \times \GL(V_2) \times \GL(V_3)$, and let $G$ be
  the quotient of $G'$ by the kernel of the multiplication map $\Gm
  \times \Gm \times \Gm \to \Gm$ sending $(\gamma_1, \gamma_2,
  \gamma_3)$ to $\gamma_1 \gamma_2 \gamma_3^2$.  Let $\Lambda$ be the
  lattice whose Gram matrix is
\begin{equation} \label{eq:intmatrix22sym24}
\begin{pmatrix}
0 & 4 \\
4 & 4
\end{pmatrix}
\end{equation}
and let $S = \{e_1, e_2\}$.  Then the $G(\fd)$-orbits of an open
subset of $V(\fd)$ are in bijection with the $\fd$-rational points of
an open subvariety of the moduli space $\sM_{\Lambda,S}$ of K3
surfaces $X$ lattice-polarized by $(\Lambda, S)$.
\end{theorem}

\subsection{Construction of K3 surfaces}

From a general element $A \in V_1 \tns V_2 \tns \Sym^2(V_3)$, we
obtain several natural surfaces.  We view $A$ as a tridegree $(1,1,2)$
form, denoted by $A(\ccdot,\ccdot,\ccdot)$, on $V_1^\vee \times
V_2^\vee \times V_3^\vee$.  First, define the quartic surface $X_3 :=
\{ z \in \Proj(V_3^\vee) : \det A(\ccdot,\ccdot, z) = 0 \}$.  If $X_3$
is nonsingular or has only rational double point singularities, then
$X_3$ is a K3 surface.  We call such $A$ {\em nondegenerate}, and we
will only consider such $A$.  Now let
\begin{align*}
X_{13} &:= \{(x,z) \in \Proj(V_1^\vee) \times \Proj(V_3^\vee) : A(x, \ccdot, z) = 0 \} \\
X_{23} &:= \{(y,z) \in \Proj(V_2^\vee) \times \Proj(V_3^\vee) : A(\ccdot, y, z) = 0 \}.
\end{align*}
These are each cut out by two bidegree $(1,2)$ forms in $\Proj^1
\times \Proj^3$.  Note that there are natural projections $X_{i3} \to
X_3$ for $i = 1$ or $2$, and any isolated singularities on $X_3$ will
be blown up by these maps.  Finally, we let
\[
X_{123} := \{(x,y,z) \in \Proj(V_1^\vee) \times \Proj(V_2^\vee) \times \Proj(V_3^\vee) : A(x, \ccdot, z) = A(\ccdot, y, z) = 0 \}.
\]
The surface $X_{123}$ projects to $X_{i3}$ for $i = 1$ or $2$, and we
see that all of these surfaces are birational.

The projection of $X_{123}$ to $\Proj(V_1^\vee) \times
\Proj(V_2^\vee)$ has degree $8$.  For any point $(x, y) \in
\Proj(V_1^\vee) \times \Proj(V_2^\vee)$, the preimage in $X_{123}$ is
the intersection of a $\Proj^2$ of quadrics in $\Proj(V_3^\vee)$.

For $\{i,j\} = \{1,2\}$, another way to construct $X_{i3}$ is to view
a general element of $V_j \tns \Sym^2(V_3)$ as giving a genus one
curve of degree $4$ in $\Proj(V_3^\vee)$, namely the base locus of the
pencil of quadrics in $\Proj(V_3^\vee)$.  Then an element $v$ of $V_i
\tns V_j \tns \Sym^2(V_3)$ gives a pencil over $\Proj(V_i^\vee)$ of
genus one curves, and the discriminant has degree $24$ and is
irreducible.  This gives $X_{i3}$ as a genus one fibration over
$\Proj(V_i^\vee)$ with generically only nodal reducible fibers.

Although it will not be directly relevant to the moduli problem, yet
another K3 surface $Y$ may be obtained by viewing $v$ as a symmetric
matrix of bilinear forms on $V_1^\vee \times V_2^\vee$.  The
determinant of this matrix is thus a bidegree $(4,4)$ curve in
$\Proj(V_1^\vee) \times \Proj(V_2^\vee)$, and the double cover of
$\Proj(V_1^\vee) \times \Proj(V_2^\vee)$ ramified at this bidegree
$(4,4)$ curve is a K3 surface.  It is also a genus one fibration over
both $\Proj(V_1^\vee)$ and over $\Proj(V_2^\vee)$.  Indeed, as a
fibration over $\Proj(V_i^\vee)$, the smooth irreducible fibers are
genus one curves of degree $2$ (namely, double covers of
$\Proj(V_j^\vee)$ ramified at a degree $4$ subscheme of
$\Proj(V_j^\vee)$).

\subsection{N\'eron-Severi lattice}

The K3 surface $X_{123}$ corresponding to a very general point in the
moduli space has rank $2$, since it is a genus one fibration without
any extra divisors. The N\'eron-Severi lattice is spanned by $L_1$,
$L_2$ and $L_3$ (the pullback of hyperplane divisors from
$\Proj(V_1^\vee)$, $\Proj(V_2^\vee)$ and $\Proj(V_3^\vee)$) which have
the intersection numbers $L_i^2 = 0$ and $L_i \cdot L_3 = 4$ for $i =
1$ or $2$, $L_1 \cdot L_2 = 8$, and $L_3^2 = 4$.  It is easily seen
that $2L_3 = L_1 + L_2$, so the lattice spanned by their classes in
the N\'eron-Severi group has a basis $\{ L_1, L_3 \}$, with
intersection matrix
\[
\left( \begin{array}{cc}
0 & 4 \\ 4 & 4
\end{array} \right).
\] 
In fact, it must be the entire N\'eron-Severi lattice. To prove this,
it suffices to show that the three classes $L_1/2$, $L_3/2$ and $(L_1
+ L_3)/2$ do not arise from divisors on the surface. The first
assertion is immediate, since $L_1$ is the class of an elliptic fiber,
and therefore not multiple. The other two $\Q$-divisor classes have
odd self-intersection, so they cannot come from divisors, either.

\subsection{Moduli problem}

To prove Theorem \ref{thm:22sym24}, we need to construct an element of
$V_1 \tns V_2 \tns \Sym^2(V_3)$ from a K3 surface $X$ with divisors
$L_1$ and $L_3$ that have intersection matrix
\eqref{eq:intmatrix22sym24}.

\begin{proof}[Proof of Theorem $\ref{thm:22sym24}$]
  First, the natural multiplication map $\Sym^2 \cH^0(X, L_3) \to
  \cH^0(X, 2 L_3)$ is an isomorphism by a dimension count (each has
  dimension $10$).  We consider the multiplication map
  \begin{equation} \label{eq:multmap22sym24}
    \mu : \cH^0(X,L_1) \tns \Sym^2 \cH^0(X,L_3) \stackrel{\cong}{\to} 
    \cH^0(X,L_1) \tns \cH^0(X, 2 L_3) \to \cH^0(X,L_1 + 2 L_3).
  \end{equation}
  The dimension of the domain is $20$, and $\ch^0(X,L_1 + 2 L_3) =
  \frac{1}{2}(L_1 + 2 L_3)^2 + \chi(\mathcal{O}_X) = 16 + 2 = 18$.  We
  claim that \eqref{eq:multmap22sym24} is surjective, in which case
  the kernel is $2$-dimensional and will give the desired tensor.

  The surjectivity of $\mu$ follows directly from the basepoint-free
  pencil trick and the fact that $\cH^1(X, L_1^{-1} \tns L_3^{\tns
    2})$ is $0$.  This last vanishing may be obtained by computing
  $\chi(L_1^{-1} \tns L_3^{\tns 2}) = 2$, $\ch^0(X,L_1^{-1} \tns
  L_3^{\tns 2}) = 2$ (because the bundle is nef and semiample), and
  $\ch^2(X,L_1^{-1} \tns L_3^{\tns 2}) = 0$ by Serre duality.  Note
  that the basepoint-free pencil trick also gives an isomorphism of
  the kernel of $\mu$ with $\cH^0(X,L_1^{-1} \tns L_3^{\tns 2})$.

  The usual argument (e.g., see the proofs of Theorems \ref{thm:rr}
  and \ref{thm:penteractorbits}) shows that these two constructions
  are inverse to one another.
\end{proof}

\section{\texorpdfstring{$\Sym^2(2) \tns \Sym^2 (4)$}{Sym2(2) (x) Sym2(4)}} \label{sec:sym22xsym24}

Finally, just as in the Rubik's revenge and penteract cases, we may
consider a symmetric linear subspace of the previous case of $2 \tns 2
\tns \Sym^2(4)$.  Specifically, let $V = \Sym^2(V_1) \tns
\Sym^2(V_2)$, where $V_1$ and $V_2$ are $\fd$-vector spaces of
dimensions $2$ and $4$, respectively.  Then the general orbits of $V$
under linear transformations on $V_1$ and $V_2$ correspond to certain
K3 surfaces of rank at least $9$ over~$\fdbar$:

\begin{theorem} \label{thm:sym22sym24} Let $V_1$ and $V_2$ be
  $\fd$-vector spaces of dimensions $2$ and $4$, respectively.  Let
  $G' = \Gm \times \GL(V_1) \times \GL(V_2)$ and let $G$ be its
  quotient by the kernel of the multiplication map $\Gm \times \Gm
  \times \Gm \to \Gm$ sending $(\gamma_1, \gamma_2, \gamma_3)$ to
  $\gamma_1 \gamma_2^2 \gamma_3^2$.  Let $\Lambda$ be the lattice
  whose Gram matrix is {\small
	\begin{equation} \label{eq:intmatrixsym22sym24}
	\begin{pmatrix}
	 0 & 4 & 1 & 1 & 1 & 1 & 1 & 1 & 1 \\
	 4 & 4 & 0 & 0 & 0 & 0 & 0 & 0 & 0 \\
	 1 & 0 & -2 & 0 & 0 & 0 & 0 & 0 & 0 \\
	 1 & 0 & 0 & -2 & 0 & 0 & 0 & 0 & 0 \\
	 1 & 0 & 0 & 0 & -2 & 0 & 0 & 0 & 0 \\
	 1 & 0 & 0 & 0 & 0 & -2 & 0 & 0 & 0 \\
	 1 & 0 & 0 & 0 & 0 & 0 & -2 & 0 & 0 \\
	 1 & 0 & 0 & 0 & 0 & 0 & 0 & -2 & 0 \\
	 1 & 0 & 0 & 0 & 0 & 0 & 0 & 0 & -2
	\end{pmatrix}
	\end{equation}
	}%
  and let $S = \{e_1, e_2\}$.  Then the $G(\fd)$-orbits of an open
  subset of $V(\fd)$ are in bijection with the $\fd$-rational points of
  an open subvariety of the moduli space $\sM_{\Lambda,S}$ of K3
  surfaces $X$ lattice-polarized by $(\Lambda, S)$.
\end{theorem}

We view this vector space as a subspace of $V_1 \tns V_1 \tns
\Sym^2(V_2)$.  Then the K3 surfaces in this case are constructed in
the same way as in \S \ref{sec:2x2xSym24}.  However, for a general
element $A \in \Sym^2(V_1) \tns \Sym^2(V_2)$, we obtain eight rank
singularities (over $\fdbar$) on the surface $X_3 \subset
\Proj(V_2^\vee)$; they are exactly the points where the symmetric $2
\times 2$ matrix of quadratic forms $A(\ccdot, \ccdot, z)$ is
identically zero, namely the intersection of three quadrics in
$\Proj^3$.  These singularities are blown up in the other surfaces
$X_{12}$, $X_{13}$, and $X_{123}$ described in \S \ref{sec:2x2xSym24},
all of which are isomorphic nonsingular K3 surfaces for the
general~$A$.  Let $E_i$ for $1 \leq i \leq 8$ denote these exceptional
divisors on $X := X_{123}$.

The symmetry also shows that the line bundles $L_1$ and $L_2$, defined
as pullbacks of $\sO_{\Proj(V_1^\vee)}(1)$ to $X$, are
the same.  We thus have the relation 
\begin{equation}
	2 L_3 = 2 L_1 + \sum_i E_i
\end{equation}
where $L_3$ is the pullback of $\sO_{\Proj(V_2^\vee)}(1)$ to $X$.
Computing the intersection numbers of $L_1$, $L_3$, and the $E_i$ in
the usual way gives the intersection matrix
\eqref{eq:intmatrixsym22sym24} (with respect to the basis $L_1, L_3,
E_1, \ldots, E_7$).  The lattice spanned by these divisors has rank
$9$ and discriminant $256$.  To check that it is all of
$\overline{\NS}(X)$, we observe by direct calculation that any element
of the dual lattice has the form
\[
D = \frac{cL_3}{4}  + \frac{1}{2} \sum_{i=1}^7 d_i E_i.
\]
First, observe that $c$ cannot be odd; otherwise, $2D$ and hence
$cL_3/2$ would be in $\overline{\NS}(X)$, which is impossible since it
has odd self-intersection. We can therefore write
\[
D = \frac{cL_3}{2}  + \frac{1}{2} \sum_{i=1}^7 d_i E_i.
\]
Then by symmetry,
\[
D'= \frac{cL_3}{2}  + \frac{1}{2} \sum_{i=1}^6 d_i E_i + d_7 E_8
\]
is also in $\overline{\NS}(X)$. Subtracting, we get $d_7(E_7 - E_8)/2 \in
\overline{\NS}(X)$, which is impossible by Lemma \ref{lem:nikulin}, unless $d_7$ is
even. Similarly, all the $d_i$ are even, resulting in $D = cL_3/2$,
which is impossible by the argument above.

To complete the proof of Theorem \ref{thm:sym22sym24}, we check that a
K3 surface $X$ whose N\'eron-Severi lattice contains the lattice
\eqref{eq:intmatrixsym22sym24} may be obtained from an element of our
vector space $V$.  The construction in the proof of Theorem
\ref{thm:22sym24} applies here, and we only need to check that the
resulting element $A$ is symmetric in the two $2$-dimensional vector
spaces. This is by the same argument as in the symmetric penteract
cases: by Theorem \ref{thm:22sym24}, the tridegree $(1,1,2)$ form $A$
in $U_1 \tns U_2 \tns \Sym^2(U_3)$ gives a K3 surface whose two
projections to $\Proj(U_1^\vee)$ and $\Proj(U_2^\vee)$ are identical
(under some identification $\phi : U_2 \to U_1$).  Therefore, if $A$
is viewed as a $2 \times 2$ matrix $B = (b_{ij})$ of quadratic forms
on the $4$-dimensional space $U_3$, we must have that $b_{12} =
b_{21}$ identically, or in other words, the image of $A$ under $\id
\tns \phi \tns \id$ is an element of $\Sym^2(U_1) \tns \Sym^2(U_3)$.

\section{Applications and connections} \label{sec:applications}

In this section, we prove Theorems \ref{thm:oguisoextension} and \ref{thm:penteractentropy}, as well as several related results, by using hyperdeterminants and the automorphisms of the K3 surfaces discussed in earlier sections.  

\subsection{Definition of hyperdeterminant} \label{subsec:hyperdet}

The hyperdeterminant of a multidimensional matrix is a natural
analogue of the determinant of a square matrix.  It was first
introduced by Cayley~\cite[pp.\ 80--94]{cayley1}, while a detailed
study was carried out in the important work of Gelfand, Kapranov, and
Zelevinsky~\cite{GKZ}.

We may define the hyperdeterminant as follows (see~\cite{GKZ} for more
details).  Let $F$ be a field, and let $T:V_1\otimes\cdots\otimes
V_r\rightarrow F$ be a linear map, where $V_1,\ldots,V_r$ are
$F$-vector spaces having dimensions $k_1+1,\ldots,k_r+1$,
respectively.  By choosing bases for $V_1,\ldots,V_r$, we may view $T$
as a $(k_1+1)\times \cdots \times (k_r+1)$ matrix.
The {\it kernel} $\ker(T)$ of $T$ is defined 
to be
$$\{ v=v_1\otimes\cdots\otimes v_r\in
V_1\otimes\cdots\otimes V_r:
T(v_1,\ldots,v_{i-1},\cdot,v_{i+1},\ldots,v_r)=0 \textrm{ for all } i\}.$$ By definition, a {\it hyperdeterminant} $\det(T)$ of the
multidimensional matrix $T$ is a polynomial of minimal degree in the
entries of $T$ whose vanishing is equivalent to $T$ having a
nontrivial kernel.  If it exists, the hyperdeterminant is then
well-defined up to a scalar multiple.

The necessary conditions on the dimensions of the matrix $T$ for the
existence of hyperdeterminants was determined by Gelfand, Kapranov,
and Zelevinsky~\cite[Ch.~14]{GKZ}:

\begin{theorem}[\cite{GKZ}]\label{GKZdim}
Assume without loss of generality that $k_r\geq k_1,\ldots,k_{r-1}$.
Then hyperdeterminants exist for $(k_1+1)\times\cdots\times(k_r+1)$
matrices if and only if $k_r \leq k_1+\cdots+k_{r-1}$.
\end{theorem}
For example, when $r=2$, hyperdeterminants exist if and only if
$k_1=k_2$, i.e., the matrix is square.  By definition, we see that a
square matrix $T$ has vanishing hyperdeterminant if and only if $T$
has a nontrivial left (equivalently, right) kernel.  Thus the
hyperdeterminant in this case coincides with the usual determinant.

\subsection{Interpretations in terms of fixed-point-free automorphisms} \label{sec:hypdetfixedpt}

Although interpretations of the determinant of a square matrix (e.g.,
as a volume) have been known for centuries, interpretations for the
hyperdeterminant for higher dimensional matrices have been less
forthcoming.

In \cite{hcl2}, an interpretation of the hyperdeterminant in the case
of a $2\times2\times2$ matrix was given, namely, as the discriminant
of an associated quadratic algebra.  Analogous interpretations for the
hyperdeterminant of a $2\times3\times 3$ matrix---namely, as the
discriminant of an associated {\it cubic} algebra---were given in
\cite{hcl3}.  In the works \cite{beauville, coregular, BGW}, orbits on
multidimensional matrices of various dimensions were shown to be in
bijection with certain data involving algebraic curves, and in these
cases the hyperdeterminants are equal to the discriminants of the
corresponding curves.  Thus the nonvanishing of the hyperdeterminant
in these cases corresponds to the nondegeneracy of the associated
rings and the nonsingularity of the associated curves, respectively.

For the orbit parametrizations of K3 surfaces by multidimensional
matrices that we have studied in this paper, we find that the
hyperdeterminant does {\it not} coincide with the discriminant, but
only divides it.  This raises the question as to the interpretation of
the hyperdeterminant in these cases.  In the cases of $4\times 4\times
4$ and $2\times 2\times 2\times 2\times 2$ matrices, we showed that
the generic orbits of such matrices correspond to K3 surfaces with at
most isolated double point singularities that are
$(\Lambda,S)$-polarized for some pair $(\Lambda,S)$.  Moreover, these
K3 surfaces are naturally equipped with birational automorphisms
$\Phi$, which lift to automorphisms of the nonsingular models; these
birational automorphisms are in fact automorphisms whenever the
associated K3 surfaces have no rank singularities.

The interpretation of the hyperdeterminant locus that we obtain in
this case is then as follows:

\begin{theorem}\label{hypint}
Let $T$ be a $4\times4\times4$ matrix, and suppose that the associated
K3 surfaces $X_1,X_2,X_3$ via Theorem~$\ref{thm:rr}$ have no rank
singularities.  Then the hyperdeterminant of $T$ vanishes if and only
if the associated automorphism $\Phi$ of $X=X_1$ has a fixed point.
\end{theorem}

\begin{proof}
Suppose the hyperdeterminant of $T$ vanishes, and let $v_1\otimes
v_2\otimes v_3\in\ker(T)$.  For each $i\in\{1,2,3\}$, let $\bar v_i$
denote the image of $v_i$ in $\Proj(V_i)$. Then $\bar v_i$ is a point
on $X_i$.  By the definition of $\Phi$ and the fact that there are no
rank singularities on the $X_i$, we see that $\psi_{12}(\bar v_1)=\bar
v_2$, $\psi(\bar v_2)=\bar v_3$, and $\psi(\bar v_3)=\bar v_1$; hence
$\Phi(\bar v_1)=\bar v_1$, yielding a fixed point of $\Phi$ on $X_1$,
as desired.
\end{proof}

If some of the $X_i$ have isolated rank singularities, then the maps
$\psi_{ij}:X_i\dashrightarrow X_j$ are not isomorphisms but birational
maps.  These maps lift uniquely to isomorphisms
$\tilde\psi_{ij}:\tilde X_i\to \tilde X_j$ between the nonsingular
models $\tilde X_i$ and $\tilde X_j$ of $X_i$ and $X_j$, respectively
(see, e.g., \cite[Theorem 10.21]{badescu}). We thus obtain an
automorphism
$\tilde\Phi=\tilde\psi_{31}\circ\tilde\psi_{23}\circ\tilde\psi_{12}$
of $X=X_1$.  In this case too, we may still use the hyperdeterminant
to detect fixed points of $\tilde\Phi$:

\begin{theorem}\label{hypint11}
Let $T$ be a $4\times4\times4$ matrix, and suppose that the associated
K3 surfaces $X_1,X_2,X_3$ via Theorem~$\ref{thm:rr}$ have only
isolated double point singularities. If the associated automorphism
$\tilde \Phi$ of the nonsingular model $\tilde X=\tilde X_1$ of $X_1$
has a fixed point, then the hyperdeterminant of $T$ vanishes.
\end{theorem}
\begin{proof}
Suppose $\tilde v_1$ on $\tilde X_1$ is a fixed point of $\tilde\Phi$.
Let $\tilde v_2=\tilde\psi_{12}(\tilde v_1)$ and $\tilde
v_3=\tilde\psi_{23}(\tilde v_2)$, so that $\tilde
v_1=\tilde\psi_{31}(\tilde v_3)$.  Let $v_1,v_2,v_3$ denote the images
of $\tilde v_1, \tilde v_2,\tilde v_3$ in $X_1,X_2,X_3$, respectively.
We claim that $v_1\otimes v_2\otimes v_3\in\ker(T)$.  Indeed, the
nonsingularization map $\tilde X_1\to X_1$ factors through
$$
X_{12} =\left\{ (x,y) \in \Proj(V_1^\vee) \times \Proj(V_2^\vee) :
\cube(x,y,\ccdot) = 0 \right\}.
$$ 
(In fact, $X_{12}$ is isomorphic to
$\tilde X_1$ when $X_1$ only has simple isolated rank singularities.)
It follows that $T(v_1,v_2,\ccdot)=0$.  Similarly,
$T(\ccdot,v_2,v_3)=T(v_1,\ccdot,v_3)=0$.  This is the desired
conclusion.
\end{proof}
\noindent
In particular, if the hyperdeterminant is nonzero, then the
automorphism $\tilde\Phi$ of $X$ has no fixed points.

Similarly, we have: 

\begin{theorem}\label{hypint2}
Let $T$ be a $2\times2\times2\times2\times2$ matrix, and suppose that
the associated K3 surfaces $X_{ijk}$ via
Theorem~\ref{thm:penteractorbits} have no rank singularities.  Then
the hyperdeterminant of $T$ vanishes if and only if one
$($equivalently, every one$)$ of the associated automorphisms
$\Phi_{abcde}$ of $X_{ijk}$ has a fixed point.
\end{theorem}

\begin{theorem}\label{hypint22}
Let $T$ be a $2\times2\times2\times2\times2$ matrix, and suppose that
the associated K3 surfaces $X_{ijk}$ via
Theorem~$\ref{thm:penteractorbits}$ have only isolated double point
singularities. If, for any $i,j,k$, the associated automorphism
$\tilde \Phi_{abcde}$ of the nonsingular model $\tilde X=\tilde
X_{ijk}$ of $X_{ijk}$ has a fixed point, then the hyperdeterminant of
$T$ vanishes.
\end{theorem}
\noindent
The proofs are similar to those of Theorems~\ref{hypint} and \ref{hypint11}.

In \S \ref{sec:finiteorderauts} and \S \ref{sec:positiveentropy}, we
use these theorems about hyperdeterminants vanishing to exhibit
fixed-point-free automorphisms of finite order and of positive
entropy, respectively, for most of the K3 surfaces in some of the
families we consider (namely, those where the hyperdeterminant does
not vanish).

\subsection{Fixed-point-free automorphisms of finite order} \label{sec:finiteorderauts}

We may use Theorem \ref{hypint11} and Theorem \ref{hypint22} to find
fixed-point-free automorphisms of finite order for most of the K3
surfaces in some of the symmetric Rubik's revenge and penteract
families.

For the doubly symmetric Rubik's revenge case, the automorphism $\Phi$
of a general member $X$ of the family of K3 surfaces gives an
involution of $X$ (as described in \S \ref{sec:2symRRauts}).  When the
hyperdeterminant does not vanish, by Theorem \ref{hypint11}, this
involution is fixed-point-free.  Such an involution produces an
Enriques surface, so the moduli space of the K3s in this family also
correspond to (an open part of the) moduli space for certain Enriques
surfaces.

Similarly, for the triply symmetric Rubik's revenge, the automorphism
$\Phi$ is a fixed-point-free involution for the general Hessian
quartic surface; this involution is studied in \cite{dolga-keum}.

For quadruply and quintuply symmetric penteracts, recall from \S
\ref{sec:4sympentauts} and \S \ref{sec:5sympentauts} that the
$5$-cycles $\Phi_{ijklm}$ are order $4$ automorphisms.  By Theorem
\ref{hypint22}, these automorphisms are fixed-point-free.  As
discussed in \S \ref{sec:4sympentauts} and \S \ref{sec:5sympentauts},
the square of each of these automorphisms is an involution but no
longer fixed-point-free.

Note that we have previously constructed other automorphisms for the
penteract (and symmetric penteract) cases with finite order but which
are not fixed-point-free.  For example, for the triply symmetric
penteracts, the four-cycles
$$\alpha_{23,5} \circ \alpha_{34,5}: X_{1234} \to X_{1235} \to
X_{1245} \to X_{1345}$$ and $\alpha_{13,5} \circ \alpha_{34,5}$ are
order $3$.  Viewing the K3 surface $X_{1234}$ as a genus one fibration
over $\Proj(V_1^\vee)$ (respectively, $\Proj(V_2^\vee)$), the
automorphism $\alpha_{23,5} \circ \alpha_{34,5}$ (resp.,
$\alpha_{13,5} \circ \alpha_{34,5}$) is given by translation by a
$3$-torsion section of the Jacobian fibration (see \cite[\S
  6.3.2]{coregular}).  The reducible fibers of the genus one fibration
have fixed points, however.  Similar automorphisms (corresponding to
translations by $3$-torsion sections of the Jacobian fibrations)
appear for the doubly-triply, quadruply, and quintuply symmetric
penteracts as well.

\subsection{Fixed-point-free automorphisms of positive entropy} \label{sec:positiveentropy}

We show that many of the automorphisms that we have constructed in
earlier sections have positive entropy and are fixed-point-free for
the general member of the corresponding family.  Specifically, we
obtain such fixed-point-free automorphisms with positive entropy for
the cases from lines 6, 9, 11, 13, 15, and 16 of
Table~\ref{table:examples}.

In each of these cases, by the parametrization theorems in this paper,
the N\'eron-Severi lattice of the K3 surfaces $X$ (over ${\overline{F}}$)
in the families contain a given lattice $\Lambda$; the N\'eron-Severi lattice of the very
general member of the family will be exactly $\Lambda$.  We will
describe the action of a particular automorphism $\Phi$ on $X$ (defined over $F$);
we find that $\Phi^*$ acts on $\Lambda$ by a
matrix $M$, which has an eigenvalue $\lambda$ of norm larger than $1$.
Since the action of $\Phi^*$ on $\overline{\NS}(X) \otimes \R$ has at most one eigenvalue of 
modulus larger than $1$ (see \cite[\S 2.3.2]{cantat-survey}) and
$M$ fixes the subspace $\Lambda \otimes \R$ in $\overline{\NS}(X) \otimes \R$, the
spectral radius of $\Phi^*$ is exactly $\lambda$.

In other words, for each of these cases,
we find that the entropy of the automorphism for each K3 surface is
the logarithm of the norm of the largest eigenvalue $\lambda$ of
$M$.  The theorems from \S \ref{sec:hypdetfixedpt} imply that these
automorphisms are fixed-point-free for the general member of the
family, specifically when the hyperdeterminant of the corresponding
element does not vanish.

\subsubsection*{Rubik's revenge}

As mentioned in \S \ref{sec:rrauts}, for each K3 surface $X$ arising
from a Rubik's revenge, there exists an automorphism $\Phi$ whose
induced action on (the known part of) $\NS(X)$ is given by the matrix
$$\begin{pmatrix}
-3 & -8 \\
8 & 21
\end{pmatrix}.$$
The characteristic polynomial of this matrix is $\lambda^2 - 18
\lambda + 1$, and the largest eigenvalue is $\lambda_{\mathrm{RR}} = 9
+ 4 \sqrt{5}$.  The entropy of the automorphism $\Phi$ is thus 
$6 \log \frac{1+\sqrt{5}}{2}$, and by Theorem \ref{hypint}, this
automorphism is also fixed-point-free if the hyperdeterminant of the
Rubik's revenge does not vanish.  This gives the proof of Theorem
\ref{thm:oguisoextension}, an extension of Oguiso's result from
\cite{oguiso}.

\subsubsection*{Penteracts}

Recall from \S \ref{sec:pentauts} that we defined an automorphism
$\Phi_{51234}$ (as a certain $5$-cycle along the $1$-dimensional
boundary of a $5$-cell), and $\Phi_{51234}$ induces the action of the
matrix
\[
\begin{pmatrix}
-1 & 0 & 2 & 2 \\
-2 & 1 & 2 & 4 \\
-4 & 2 & 5 & 6 \\
-6 & 2 & 8 & 11
\end{pmatrix}
\]
on (the known part of) $\NS(X)$.  The characteristic polynomial of
this matrix is $\lambda^4 - 16 \lambda^3 + 14 \lambda^2 - 16 \lambda +
1$, and the maximum eigenvalue $\lambda_{\mathrm{pent}}$ is
approximately $15.1450744834468$.  Therefore, the entropy of
$\Phi_{51234}$ is $\log \lambda_{\mathrm{pent}} \approx
2.717675362$.  The same numerics occur for all of the other
$5$-cycles, by symmetry.

By Theorem \ref{hypint2}, the automorphism $\Phi_{51234}$ will be
fixed-point-free if the hyperdeterminant of the penteract does not
vanish.  Thus, we have produced a family of K3 surfaces whose general
member has several fixed-point-free automorphisms with positive
entropy, giving the proof of Theorem \ref{thm:penteractentropy}.

Recall that for the penteract (and the symmetric) cases, one may also
consider automorphisms that are $3$- or $4$-cycles along the boundary
of the $5$-cell; these all have zero entropy. There are also
infinitely many other automorphisms (for instance, by taking arbitrary
words in the generators) with positive entropy (see \S
\ref{sec:pentauts}).  This is also true for the symmetric penteract
cases below.

\subsubsection*{Doubly symmetric penteracts} 

Recall from \S \ref{sec:2sympentauts} that the automorphism
$\Phi_{54321}$ on a K3 surface $X$ here induces the action of the
matrix \eqref{eq:2sympentaut} on (the known part of)
$\overline{\NS}(X)$, which has characteristic polynomial
$(\lambda-1)^6(\lambda+1)(\lambda^2 - 6\lambda + 1)$ and largest
eigenvalue $3 + 2 \sqrt{2}$.  The entropy of $\Phi_{54321}$ here is
thus $\log{(3 + 2 \sqrt{2})} = 2 \log{(\sqrt{2} + 1)}$.

In addition, the entropy of $\Phi_{53421}$ is $\log
\lambda_{\mathrm{pent}}$, since its action on $\overline{\NS}(X)$ is
similar to the action of the penteract case.  The $5$-cycles thus have
entropy either $2\log{(\sqrt{2} + 1)}$ or $\log
\lambda_{\mathrm{pent}}$.

By Theorem \ref{hypint22}, we thus find that the general member of
this family of K3 surfaces has many fixed-point-free automorphisms
with positive entropy.

\subsubsection*{Triply symmetric penteracts} 

In \S \ref{sec:3sympentauts}, for the general member $X$ of the family
of K3 surfaces related to triply symmetric penteracts, we found an
automorphism $\Phi_{54123}$ whose action on (the known part of)
$\overline{\NS}(X)$ is given by the matrix \eqref{eq:3sympentaut}.
This matrix has characteristic polynomial $(\lambda + 1)^2 (\lambda^2
- \lambda + 1)^5 (\lambda^2 - 3 \lambda+1)$ and largest eigenvalue
$\frac{3 + \sqrt{5}}{2} = (\frac{\sqrt{5}+1}{2})^2$.

As before, some of the other $5$-cycles have the same numerics (by
symmetry).  The other $5$-cycles (like $\Phi_{54132}$) have entropy at
least $\log (3 + 2 \sqrt{2})$, as their action on N\'eron-Severi is
much like that of $\Phi_{54321}$ in the doubly symmetric penteract
case.

Thus, applying Theorem \ref{hypint22}, we have that the general member
of this family of K3 surfaces has fixed-point-free automorphisms with
entropy $2\log{(\frac{\sqrt{5}+1}{2})}$ and 
$2\log{(\sqrt{2}+1)}$.

\subsubsection*{Doubly-doubly symmetric penteracts} 

In \S \ref{sec:22sympentauts}, we described a $5$-cycle automorphism
$\Phi_{53214}$ on the general member of the family of K3 surfaces
arising from doubly doubly symmetric penteracts.  Its action on (the
known part of) $\overline{\NS}(X)$ is given by
\eqref{eq:22sympentaut}, with characteristic polynomial
$(\lambda+1)^{12} (\lambda^2+1) (\lambda^2 - 4 \lambda + 1)$ and
largest eigenvalue $2 + \sqrt{3}$.

Therefore, Theorem \ref{hypint22} implies that the general member of
this family of K3 surfaces has a fixed-point-free automorphism with
entropy $\log{(2 + \sqrt{3})}$.  In addition, from the automorphisms
$\Phi_{53421}$ and $\Phi_{53241}$ (and other analogous $5$-cycles), we
also obtain fixed-point-free automorphisms with entropy equal to $\log
\lambda_{\mathrm{pent}}$ and $2\log{(\sqrt{2}+1)}$, respectively.

\subsubsection*{Doubly-triply symmetric penteracts} 

The automorphism $\Phi_{53214}$ from \S \ref{sec:23sympentauts},
applied to the general member $X$ of the family of K3 surfaces coming
from doubly triply symmetric penteracts, acts on (the known part of)
$\overline{\NS}(X)$ by \eqref{eq:23sympentaut}.  It has characteristic
polynomial $(\lambda-1)^3 (\lambda + 1) (\lambda^2 - 3\lambda + 1)
(\lambda^2 + \lambda + 1)^8$ and largest eigenvalue $\frac{3 +
  \sqrt{5}}{2}$ (just as in the triply symmetric penteract case).  The
same argument shows that this gives fixed-point-free automorphisms
with entropy $2\log{(\frac{\sqrt{5}+1}{2})}$.  Moreover, the
automorphism $\Phi_{52413}$ is fixed-point-free with entropy
$2\log{(\sqrt{2}+1)}$.

\subsection*{Acknowledgements}

We thank Noam Elkies, Dick Gross, Curt McMullen, Bjorn Poonen, Igor
Rivin, Peter Sarnak, and Ichiro Shimada for many helpful
conversations. We are very grateful to Serge Cantat and Genya Zaytman
for careful readings of earlier drafts and many useful comments. We
also thank Arthur Baragar, Igor Dolgachev, and Brendan Hassett for
feedback on earlier drafts. Finally, we are very thankful to the
anonymous referees for careful readings of the paper, and for many
insightful remarks and detailed references that improved it. MB was
supported by a Simons Investigator Grant and NSF grant~DMS-1001828. WH
was supported by NSF grant DMS-1406066. AK was supported in part by
National Science Foundation grant DMS-0952486 and by a grant from the
MIT Solomon Buchsbaum Research Fund.

{\small
\bibliography{K3OrbitBib}
\bibliographystyle{amsplain}
}

\end{document}